\documentclass[12pt]{amsart}
\usepackage[T2A]{fontenc}
\usepackage[english]{babel}
\usepackage{amsaddr}
\usepackage{amsmath}
\usepackage{amssymb}
\usepackage{amsfonts}
\usepackage{epsfig}
\usepackage{srcltx}
\usepackage{subfigure}
\usepackage{cite}
\usepackage{float}
\usepackage[a4paper,  mag=1000, includefoot,  left=2cm, right=2cm, top=2cm, bottom=1.5cm, headsep=1cm, footskip=1cm ]{geometry}

\newtheorem{Th}{Theorem}
\newtheorem{Lem}{Lemma}

\newtheorem{Rem}{Remark}

\begin{document}
\thispagestyle{empty}

\title[ ]
{
Bifurcations in asymptotically autonomous Hamiltonian systems under oscillatory perturbations
}
\author{Oskar A. Sultanov}
\address{Chebyshev Laboratory, St. Petersburg State University, 14th Line V.O., 29, Saint Petersburg 199178 Russia.}

\email{oasultanov@gmail.com}

\maketitle
{\small

{\small
\begin{quote}
\noindent{\bf Abstract.} The effect of decaying oscillatory perturbations on autonomous Hamiltonian systems in the plane with a stable equilibrium is investigated. It is assumed that perturbations preserve the equilibrium and satisfy a resonance condition. The behaviour of the perturbed trajectories in the vicinity of the equilibrium is investigated. Depending on the structure of the perturbations, various asymptotic regimes at infinity in time are possible. In particular, a phase locking and a phase drifting can occur in the systems. The paper investigates the bifurcations associated with a change of Lyapunov stability of the equilibrium in both regimes. The proposed stability analysis is based on a combination of the averaging method and the construction of Lyapunov functions.

\medskip


\end{quote}
}

\section{Introduction}
In this paper, the influence of time-dependent perturbations on the stability of solutions to autonomous Hamiltonian systems is investigated. It is assumed that the perturbations fade with time such that the disturbed systems are asymptotically autonomous. Note that such systems have been considered in many papers. In particular, the relations between the trajectories of  asymptotically autonomous system and the solutions of the corresponding limiting system were discussed in~\cite{Markus56}. In some cases, the trajectories of disturbed and limiting systems have the same asymptotic behavior~\cite{Theim92}. However, this is not true in general~\cite{Theim94}. The qualitative and asymptotic properties of solutions to such systems depend both on the form of the unperturbed system and on the structure of decaying perturbations~\cite{WongBurton65,Grimmer69,Rassmusen08,CP10}.

The effect of perturbations with a small parameter on local properties of dynamical systems is considered as well-studied problem~\cite{BM61,GH83,Hap93,GL94,AKN06,Hans07}. In this paper, the presence of a small parameter is not assumed. We consider autonomous systems in the plane with non-autonomous oscillatory perturbations vanishing at infinity in time.
The behaviour of the perturbed trajectories in the vicinity of the equilibrium is investigated. Note that the influence of such perturbations on the solutions of autonomous equations and systems have been discussed, for example, in~\cite{Wintner46,Atkinson54,Simon69,HL75,DF78,BD79,AK96,UYG02,PN06,BN10,Lukic14}. However, to the best of our knowledge, the bifurcations associated with decaying and nonlinear terms in systems have not been thoroughly analyzed.

The paper is organized as follows. In section~\ref{sec1}, the mathematical formulation of the problem is given and the class of damped perturbations is described. The proposed method of stability and bifurcation analysis is based on a change of variables that simplifies the system in the first asymptotic terms and on the construction of suitable Lyapunov functions. The construction of this transformation is described in section~\ref{sec2}. Possible asymptotic regimes in the system, depending on the structure of the simplified equations, are described in section~\ref{sec3}. Bifurcations associated with a change of the stability of the equilibrium are discussed in sections~\ref{sec4} and~\ref{sec5}. In section~\ref{secEX}, the proposed theory is applied to the examples of non-autonomous systems with oscillating and damped perturbations. The paper concludes with a brief discussion of the results obtained.

\section{Problem statement}
\label{sec1}

Consider the non-autonomous system of two differential equations:
\begin{gather}
    \label{FulSys}
        \frac{dx}{dt}=\partial_y H(x,y,t), \quad \frac{dy}{dt}=-\partial_x H(x,y,t)+F(x,y,t), \quad t>0.
\end{gather}
It is assumed that the functions $H(x,y,t)$ and $F(x,y,t)$ are infinitely differentiable and for every compact $D\subset\mathbb R^2$
\begin{gather*}
    \lim_{t\to\infty} H(x,y,t) = H_0(x,y), \quad \lim_{t\to\infty} F(x,y,t) = 0
\end{gather*}
for all $(x,y)\in D$. The limiting autonomous system
\begin{gather}
    \label{LimSys}
        \frac{dx}{dt}=\partial_y H_0(x,y), \quad \frac{dy}{dt}=-\partial_x H_0(x,y)
\end{gather}
is assumed to have the isolated fixed point $(0,0)$ of center type. Without loss of generality, it is assumed that
\begin{gather}\label{H0}
        H_0(x,y)=\frac{  x^2+y^2}{2}+\mathcal O(r^3), \quad r=\sqrt{x^2+y^2}\to 0,
\end{gather}
and for all $E\in(0,E_0]$, $E_0={\hbox{\rm const}}$, the level lines $\{(x,y)\in\mathbb R^2: H_0(x,y)=E\}$ lying in  $D_0=\{(x,y)\in\mathbb R^2: r\leq r_0\}$, $r_0={\hbox{\rm const}}$,  define a family of closed curves on the phase space $(x,y)$ parameterized by the parameter $E$. To each closed curve there correspond a periodic solution $x_0(t,E)$, $y_0(t,E)$ of system \eqref{LimSys} with a period $T(E)=2\pi/\omega(E)$, where $\omega(E)\neq 0$ for all $E\in[0,E_0]$ and $\omega(E)=1+\mathcal O(E)$ as $E\to 0$.  The value $E=0$ corresponds to the fixed point $(0,0)$.

The perturbations of the limiting system are described by the functions with power-law asymptotics:
\begin{eqnarray}
    \label{Has}
        H(x,y,t)&=&H_0(x,y)+\sum_{k=1}^\infty t^{-\frac kq} H_k(x,y,S(t)), \\
    \label{Fas}
        F(x,y,t)&=&\sum_{k=1}^\infty t^{-\frac kq} F_k(x,y,S(t)),
\end{eqnarray}
 as $t\to\infty$ for all $(x,y)\in D_0$, where $q\in\mathbb Z_+$, the coefficients $H_k(x,y,S)$, $F_k(x,y,S)$ are $2\pi$-periodic with respect to $S$ and
\begin{gather*}
        S(t)=\sum_{k=0}^{q-1} s_k t^{1-\frac kq}+s_q \log t+\mathcal O(t^{-\frac{1}{q}}), \quad t\to\infty, \quad s_k={\hbox{\rm const}}, \quad s_0>0.
\end{gather*}
It is also assumed that the perturbations preserve the fixed point $(0,0)$ such that
\begin{gather*}
\partial_x H(0,0,t)\equiv0, \quad  \partial_y H(0,0,t)\equiv 0,  \quad F(0,0,t)\equiv 0,\\
H_k(x,y,S)=\mathcal O(r^2), \quad F_k(x,y,S)=\mathcal O(r), \quad r\to 0, \quad \forall\, k\geq 1, S\in\mathbb R,
\end{gather*}
and satisfy the resonance condition:
\begin{gather}
\label{pcond}
s_0=\varkappa \omega(0)
\end{gather}
with some positive integer $\varkappa\neq 0$. Note that such and similar systems arise in the study of various problems of mathematical physics. For example, phase synchronization models~\cite{PRK02,LK14}, autoresonance models~\cite{LKRMS08,KGT17}, the Painlev\'{e} equations~\cite{IKNF06} and their perturbations are reduced to systems of the form \eqref{FulSys} with right-hand sides having power-law asymptotics at infinity.

The simplest example is given by the following equation:
\begin{gather}
    \label{Example}
        \frac{d^2x}{dt^2}+x= t^{-1} \Big( a \cos(s_0 t+s_1\log t) x+\lambda \frac{dx}{dt}\Big), \quad  a, \lambda \in\mathbb R.
\end{gather}
This equation in the variables $x,y= \dot x$ takes form \eqref{FulSys} with $q=1$, $H =  [x^2+y^2-a t^{-1} x^2\cos S(t)]/2$ and $F= \lambda t^{-1}y $. It can easily be checked that the unperturbed equation with $a=\lambda=0$ has $2\pi$-periodic general solution $x(t;E,\varphi)=\sqrt{2E} \cos(t+\varphi)$ with $\omega(E)\equiv 1$. Numerical analysis of equation \eqref{Example} shows that the decaying perturbations can change significantly the behaviour of solutions (see Fig.~\ref{Fig12}). In this case, the stability conditions for the trivial solution depend on the value of the parameter $\varkappa$ in \eqref{pcond}. Indeed, if $\varkappa=1$, the stability is determined by the sign of the coefficient $\lambda$ in the decaying dissipative term as in the case of $a=0$. However, if $\varkappa=2$, the stability of the trivial solution changes as the parameter $\lambda$ passes through a certain critical value $\lambda_a$. In this case, the shift of the stability boundary occurs due to the presence of a non-autonomous version of a parametric resonance~\cite{VB07,BN10}. More sophisticated examples are considered in section~\ref{secEX}.

In the general case, the behaviour of solutions to non-autonomous systems of the form \eqref{FulSys} depends on nonlinear terms of equations. The goal of this paper is to describe the stability conditions for system \eqref{FulSys} and to reveal the role of decaying perturbations in the corresponding local bifurcations, associated with a change of Lyapunov stability of the trivial solution $x(t)\equiv 0$, $y(t)\equiv 0$.
\begin{figure}
\centering
\subfigure[$\varkappa=1$ ]{\includegraphics[width=0.4\linewidth]{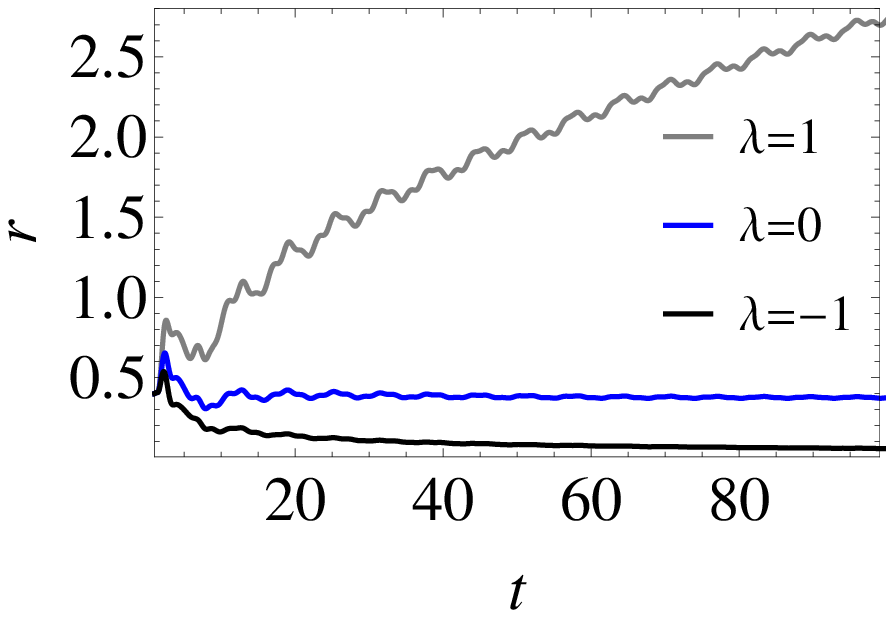}}
\hspace{4ex}
\subfigure[$\varkappa=2$]{\includegraphics[width=0.4\linewidth]{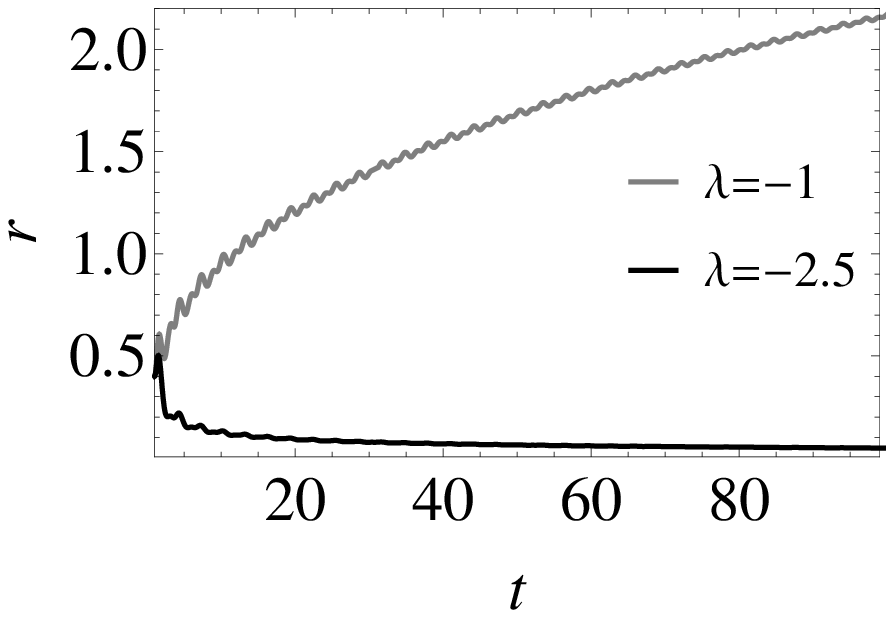}}
\caption{\small The evolution of $r(t)=\sqrt{x^2(t)+y^2(t)}$ for solutions of \eqref{Example} with $x(1)=0.4$, $y(1)=0$, $a=4$, and $s_1=1$.} \label{Fig12}
\end{figure}

\section{Change of variables}
\label{sec2}
In this section, a suitable transformation of variables is constructed to simplify system \eqref{FulSys}. First, define auxiliary $2\pi$-periodic functions $X(\varphi,E)=x_0(\varphi/\omega(E),E)$ and $Y(\varphi,E)=y_0(\varphi/\omega(E),E)$, satisfying the system:
\begin{gather*}
    \omega(E)\frac{\partial X}{\partial \varphi}=\partial_Y H_0(X,Y), \quad
    \omega(E)\frac{\partial Y}{\partial \varphi}=-\partial_X H_0(X,Y).
\end{gather*}
These functions are used for rewriting system \eqref{FulSys} in the action-angle variables $(E,\varphi)$:
\begin{gather}
    \label{exch1}
        x(t)=X (\varphi(t),E(t)), \quad y(t)=Y(\varphi(t),E(t)).
\end{gather}
In the new coordinates $(E,\varphi)$, perturbed system \eqref{FulSys} takes the form:
\begin{gather}
    \label{FulSys2}
        \frac{dE}{dt}=f(E,\varphi,t), \quad  \frac{d\varphi}{dt}=\omega(E)+g(E,\varphi,t),
\end{gather}
where
\begin{eqnarray*}
    f(E,\varphi,t) & \equiv &
    -\omega(E) \Big(\partial_\varphi  H\big(X(\varphi,E),Y(\varphi,E),t\big) - F\big(X(\varphi,E),Y(\varphi,E),t\big)\partial_\varphi X (\varphi,E)\Big), \\
    g(E,\varphi,t) & \equiv & \omega(E)\Big(\partial_E  H\big(X(\varphi,E),Y(\varphi,E),t\big) -1- F\big(X(\varphi,E),Y(\varphi,E),t\big)\partial_E X(\varphi,E)\Big)
\end{eqnarray*}
are $2\pi$-periodic functions with respect to $\varphi$. Since $(0,0)$ is the equilibrium of system \eqref{FulSys}, we see that $E=0$ is the fixed point of the first equation in \eqref{FulSys2}: $f(0,\varphi,t)\equiv 0$ for all $\varphi\in \mathbb R$ and $t>0$. Moreover, from \eqref{Has} and
\eqref{Fas} it follows that
\begin{gather*}
    f(E,\varphi,t)=\sum_{k=1}^\infty t^{-\frac kq}  f_k\big(E,\varphi,S(t)\big), \quad
    g(E,\varphi,t)=\sum_{k=1}^\infty t^{-\frac kq}  g_k\big(E,\varphi,S(t)\big), \quad t\to\infty,
\end{gather*}
where
\begin{eqnarray*}
        f_k(E,\varphi,S)&\equiv &\omega(E) \Big(-\partial_\varphi H_k\big(X(\varphi,E),Y(\varphi,E),S\big)+ F_k\big(X(\varphi,E),Y(\varphi,E),S\big)\partial_\varphi X(\varphi,E)  \Big),
        \\
        g_k(E,\varphi,S) &\equiv& \omega(E) \Big(\partial_E H_k\big(X(\varphi,E),Y(\varphi,E),S\big)- F_k\big(X(\varphi,E),Y(\varphi,E),S\big)\partial_E X(\varphi,E) \Big)
 \end{eqnarray*}
are $2\pi$-periodic functions with respect to $\varphi$ and $S$. Since $X(\varphi,E)=\sqrt{2E}\cos\varphi + \mathcal O(E^{3/2})$, $Y(\varphi,E)=-\sqrt{2E}\sin\varphi + \mathcal O(E^{3/2})$, we see that
\begin{gather*}
f_k(E,\varphi,S)=\sum_{j=2}^\infty E^{\frac j2} f_{k,j}(\varphi,S), \quad g_k(E,\varphi,S)=\sum_{j=0}^\infty E^{\frac j2} g_{k,j}(\varphi,S)
\end{gather*}
as $E\to 0$ uniformly for all $(\varphi,S)\in\mathbb R^2$. From the identity $H_0(X(\varphi,E),Y(\varphi,E))\equiv E$ it follows that
\begin{gather*}
    \begin{vmatrix}
        \partial_\varphi X & \partial_E X\\
        \partial_\varphi Y& \partial_E Y
    \end{vmatrix} = \frac{1}{\omega(E)}\neq 0.
\end{gather*}
The last inequality guarantees the reversibility of transformation \eqref{exch1} for all $E\in [0,E_0]$ and $\varphi\in \mathbb R$.

To study the influence of oscillating disturbances on the stability of the fixed point $E=0$, we consider the following change of variables:
\begin{gather}\label{epsdel}
    E(t)=t^{-\frac{l}{q}}\mathcal E(t), \quad \varphi(t)= \theta(t)+ \varkappa^{-1}S(t)
\end{gather}
as $t\geq 1$, where $l$ is some integer. We take $l=0$ if $\omega(E)\equiv {\hbox{\rm const}}$ and $l\geq 1$ if $\omega(E)\not\equiv{\hbox{\rm const}}$.  It can easily be checked that in new variables $(\mathcal E,\theta)$ system \eqref{FulSys2} takes the form:
\begin{gather}
\label{FulSys2m}
    \frac{d\mathcal E}{dt}=\mathcal F(\mathcal E,\theta,t)+t^{-1}\frac{l}{q}\mathcal E, \quad  \frac{d\theta}{dt}=\mathcal G(\mathcal E,\theta,t),
\end{gather}
where
\begin{gather*}
\mathcal F(\mathcal E,\theta,t)\equiv t^{\frac{l}{q}}f\big(t^{-\frac{l}{q}}\mathcal E,\theta+\varkappa^{-1}S(t),t\big), \quad
\mathcal G(\mathcal E,\theta,t)\equiv \omega\big(t^{-\frac{l}{q}}\mathcal E\big)- \varkappa^{-1}S'(t)+g\big(t^{-\frac{l}{q}}\mathcal E,\theta+\varkappa^{-1}S(t),t\big).
\end{gather*}
The right-hands sides of \eqref{FulSys2m} have the following asymptotics:
\begin{eqnarray*}
    \mathcal F(\mathcal E,\theta,t)=
        \sum_{k=2}^\infty t^{-\frac {k}{2q}}\mathcal F_{k}\big(\mathcal E,\theta,S(t)\big)  , \quad
    \mathcal G(\mathcal E,\theta,t)=\sum_{k=2}^\infty t^{-\frac {k}{2q}}  \mathcal G_k\big(\mathcal E,\theta,S(t)\big)
\end{eqnarray*}
as $t\to\infty$ for all $0\leq \mathcal E\leq {\hbox{\rm const}}$, $\theta\in\mathbb R$, where the coefficients $\mathcal F_k\big(\mathcal E,\theta,S\big)$ and $\mathcal G_k\big(\mathcal E,\theta,S\big)$ are $2\pi$-periodic with respect to $\theta$ and $2\pi \varkappa$-periodic with respect to $S$.
 In particular, if $l=1$, we have
\begin{align*}
 &    \mathcal F_k\big(\mathcal E,\theta,S\big)= \sum_{\substack{
            2i+j=k\\
            i\geq 1, \, j\geq 0
                }}  f_{i,j+2}\big(\theta+\varkappa^{-1} S,S\big)\mathcal E^{1+\frac{j}{2}},\\
&    \mathcal G_k\big(\mathcal E,\theta,S\big)=\frac{\omega^{(k/2)}(0)}{(k/2)!}\mathcal E^{\frac{k}{2}}-\varkappa^{-1}\Big(1-\frac{k}{2q}+\delta_{k,2q}\Big)s_{k/2} +\sum_{\substack{
            2i+j=k\\
            i\geq 1,\, j\geq 0
                }}  g_{i,j}\big(\theta+\varkappa^{-1} S,S\big) \mathcal E^{\frac j 2},
\end{align*}
where it is assumed that $s_j=0$ for $j>q$, $\omega^{(k/2)}(0)=0$ for odd $k$, and $\delta_{k,q}$ is the Kronecker delta. Since $S(t)$ changes rapidly in comparison to possible variations of $\mathcal E$ and $\theta$ at infinity, we average system  \eqref{FulSys2m} over $S$ in order to obtain simplified equations giving the first approximation to solutions. This technique is usually used in perturbation theory (see, for example,~\cite{BM61,AKN06,Hap93,AN84,BDP01,DM10}).

Consider the near-identity transformation of the variables $\mathcal E$ and $\theta$ in the form:
\begin{gather}\label{VF}
    V_N(\mathcal E,\theta,t)=\mathcal E+\sum_{k=2}^N t^{-\frac {k}{2q}} v_k\big(\mathcal E,\theta,S(t)\big), \quad
    \Psi_M(\mathcal E,\theta,t)=\theta+\sum_{k=2}^M t^{-\frac {k}{2q}} \psi_k\big(\mathcal E,\theta,S(t)\big).
\end{gather}
The coefficients $v_k(\mathcal E,\theta,S)$ and $\psi_k(\mathcal E,\theta,S)$ are chosen in such a way that the right-hand sides of the equations for the new variables $v(t)\equiv V_N(\mathcal E(t),\varphi(t),t)$  and $\psi(t)\equiv \Psi_M(\mathcal E(t),\theta(t),t)$ do not depend on $S$ at least in the first terms of the asymptotics:
\begin{gather}
    \label{VEq}
       \frac{dv}{dt}=\sum_{k=2}^N t^{-\frac {k}{2q}} \Lambda_k(v,\psi) + \tilde\Lambda_{N}(v,\psi,t), \quad
        \frac{d\psi}{dt}=\sum_{k=2}^M t^{-\frac {k}{2q}} \Omega_k(v,\psi) + \tilde \Omega_{M}(v,\psi,t),
\end{gather}
where $\tilde\Lambda_{N}(v,\psi,t)=\mathcal O(t^{-{(N+1)}/{2q}})$ and $\tilde\Omega_{M}(v,\psi,t)=\mathcal O(t^{-{(M+1)}/{2q}})$ as $t\to\infty$ for all $\psi\in\mathbb R$ and $0\leq v\leq {\hbox{\rm const}}$. To find suitable coefficients $v_k(\mathcal E,\theta,S)$, $\psi_k(\mathcal E,\theta,S)$ and to derive the functions $\Lambda_k(v,\psi)$, $\Omega_k(v,\psi)$, we calculate the total derivative of $V_N(\mathcal E,\theta,t)$ and $\Psi_M(\mathcal E,\theta,t)$ with respect to $t$ along the trajectories of system \eqref{FulSys2m}:
\begin{gather}\label{DFS}
    \begin{split}\displaystyle
    \frac{d}{dt}
        \begin{pmatrix} V_N \\ \Psi_M \end{pmatrix}\Big|_{\eqref{FulSys2m}}
        &
        \displaystyle = \Big(\big(\mathcal F(\mathcal E,\theta,t)+t^{-1}\frac{l}{q}\mathcal E\big)\partial_{\mathcal E}+\mathcal G(\mathcal E,\theta,t) \partial_\theta +\partial_t\Big)
        \begin{pmatrix} V_N \\ \Psi_M \end{pmatrix} \\
        &=  \sum_{k=2}^\infty t^{-\frac{k}{2q}}
        \left[
            \begin{pmatrix}
                \mathcal F_k +{\displaystyle \delta_{k,2q}\frac{l}{q}\mathcal E}\\
                \mathcal G_k
        \end{pmatrix} + s_0 \partial_S
        \begin{pmatrix} v_k \\ \psi_k \end{pmatrix} -\frac{k-2q}{2q}
        \begin{pmatrix} v_{k-2q} \\ \psi_{k-2q}\end{pmatrix} \right]\\
        & + \sum_{k=4}^\infty t^{-\frac {k}{2q}}\sum_{i+j=k} \Big(\big(\mathcal F_{j}+\delta_{j,2q}\frac{l}{q}\mathcal E\big)\partial_{\mathcal E}  + \mathcal G_{j}\partial_\theta +s_{  j/2}\Big(1-\frac{j}{2q}+\delta_{j,2q}\Big)\partial_S\Big)\begin{pmatrix} v_i \\ \psi_i \end{pmatrix},
    \end{split}
\end{gather}
where it is assumed that $v_i\equiv 0$, $\psi_j\equiv 0$, $\mathcal F_k\equiv 0$, $\mathcal G_k\equiv 0$ for $i,j,k<2$, $i>N$, $j>M$. Substituting \eqref{VF} into the right-hand side of \eqref{VEq} and matching the result with \eqref{DFS}, we obtain the chain of differential equations:
\begin{gather}\label{RSys}
   s_0 \partial_S \begin{pmatrix} v_k \\ \psi_k \end{pmatrix} =
\begin{pmatrix}
        \Lambda_k({\mathcal E},\theta) -\mathcal F_k(\mathcal E,\theta,S)- {\displaystyle \delta_{k,2q}\frac{l}{q}\mathcal E}+\hat{\mathcal F}_k(\mathcal E,\theta,S)\\
        \Omega_k({\mathcal E},\theta) -\mathcal G_k(\mathcal E,\theta,S) +\hat{\mathcal G}_k(\mathcal E,\theta,S)
\end{pmatrix},
\quad k\geq 2,
\end{gather}
where the functions $\hat{\mathcal F}_k$, $\hat{\mathcal G}_k$ are expressed through $\{v_i,\psi_i,\Lambda_i,\Omega_i\}_{i=2}^{k-1}$. In particular, $\hat{\mathcal F}_{2,3}\equiv 0$, $\hat{\mathcal G}_{2,3}\equiv 0$,
\begin{eqnarray*}
\begin{pmatrix}\hat{\mathcal F}_{4}\\ \hat{\mathcal G}_{4}  \end{pmatrix}
        &\equiv &
     \Big( v_2 \partial_{\mathcal E}+ \psi_2 \partial_\theta\Big)
    \begin{pmatrix} \Lambda_2\\  \Omega_2  \end{pmatrix}   \\
        &       &
    - \Big(\big(\mathcal F_{2}+\delta_{1,q}\frac{l}{q}\mathcal E\big) \partial_{\mathcal E}  + \mathcal G_{2} \partial_\theta +s_{1}\big(1-\frac{1}{q}+\delta_{1,q}\big)\partial_S\Big)
     \begin{pmatrix}v_2 \\  \psi_2  \end{pmatrix}+\frac{2-q}{q}
      \begin{pmatrix} v_{4-2q} \\ \psi_{4-2q}\end{pmatrix},\\
\begin{pmatrix}\hat{\mathcal F}_{5}\\ \hat{\mathcal G}_{5}  \end{pmatrix}
        &\equiv &
     \sum_{i+j=5}\Big( v_i \partial_{\mathcal E}+ \psi_i \partial_\theta\Big)
    \begin{pmatrix} \Lambda_j\\  \Omega_j  \end{pmatrix}   \\
        &&-\sum_{i+j=5} \Big(\big(\mathcal F_{j}+\delta_{j,2q}\frac{l}{q}\mathcal E\big) \partial_{\mathcal E}  + \mathcal G_{j} \partial_\theta +s_{j/2}\big(1-\frac{j}{2q}+\delta_{j,2q}\big)\partial_S\Big)  \begin{pmatrix}v_i \\  \psi_i  \end{pmatrix}+\frac{5-2q}{2q}\begin{pmatrix} v_{5-2q} \\ \psi_{5-2q}\end{pmatrix},\\
\begin{pmatrix}\hat{\mathcal F}_{6}\\ \hat{\mathcal G}_{6}  \end{pmatrix}
        &\equiv &
     \sum_{i+j=6}\Big( v_i \partial_{\mathcal E}+ \psi_i \partial_\theta\Big)
    \begin{pmatrix} \Lambda_j\\  \Omega_j  \end{pmatrix}   +
       \frac{1}{2}\Big( v_2^2 \partial_{\mathcal E}^2+ 2v_2\psi_2 \partial_{\mathcal E} \partial_\psi+ \psi_2^2\partial_\theta^2\Big)
    \begin{pmatrix} \Lambda_2\\  \Omega_2  \end{pmatrix}  \\
        &&-\sum_{i+j=6} \Big(\big(\mathcal F_{j}+\delta_{j,2q}\frac{l}{q}\mathcal E\big) \partial_{\mathcal E}  + \mathcal G_{j} \partial_\theta +s_{j/2}\big(1-\frac{j}{2q}+\delta_{j,2q}\big)\partial_S\Big)  \begin{pmatrix}v_i \\  \psi_i  \end{pmatrix}+\frac{3-q}{q}\begin{pmatrix} v_{6-2q} \\ \psi_{6-2q}\end{pmatrix},\\
   \begin{pmatrix}\hat{\mathcal F}_{k}\\ \hat{\mathcal G}_{k}  \end{pmatrix} &\equiv &
    \sum_{
        \substack{
            z+a_1+\ldots+ia_i+b_1+\ldots+jb_j=k\\
            a_1+\ldots+a_i+b_1+\ldots+b_j\geq 1
                }
            } C_{ij}^{ab}   v_1^{a_1}\cdots v_{i}^{a_i}\psi_1^{b_1}  \cdots \psi_{j}^{b_j} \partial^{a_1+\ldots+a_i}_{\mathcal E}\partial^{b_1+\ldots+b_j}_\theta  \begin{pmatrix} \Lambda_z \\  \Omega_z  \end{pmatrix}  \\
            && - \sum_{i+j={k}} \Big(\big(\mathcal F_{j}+\delta_{j,2q}\frac{l}{q} \mathcal E\big)\partial_{\mathcal E}  + \mathcal G_{j} \partial_\theta +s_{j/2}\big(1-\frac{j}{2q}+\delta_{j,2q}\big)\partial_S\Big)  \begin{pmatrix}v_i \\  \psi_i  \end{pmatrix}+\frac{k-2q}{2q}\begin{pmatrix} v_{k-2q} \\ \psi_{k-2q}\end{pmatrix} ,
\end{eqnarray*}
where $C_{ij}^{ab}={\hbox{\rm const}}$.
Define
\begin{gather}
    \label{Lambda}
        \begin{split}
            &\Lambda_k(v,\psi)=\langle\mathcal F_k(v,\psi,S) -  \hat{\mathcal F}_k (v,\psi,S) \rangle_{\varkappa S}+\delta_{k,2q}\frac{l}{q} v,\\
            &\Omega_k(v,\psi)=\langle\mathcal G_k(v,\psi,S)   -  \hat{\mathcal G}_k (v,\psi,S) \rangle_{\varkappa S},
        \end{split}
\end{gather}
where
\begin{gather*}
    \langle  Z(v,\psi,S) \rangle_{\varkappa S}\stackrel{def}{=} \frac{1}{2\pi}\int\limits_0^{2\pi}   Z(v,\psi,\varkappa S) \, dS= \frac{1}{2\pi\varkappa}\int\limits_0^{2\pi\varkappa}   Z(v,\psi,s) \, ds.
\end{gather*}
From \eqref{Lambda} it follows that the functions $\Lambda_k(v,\psi)$ and $\Omega_k(v,\psi)$ are $2\pi$-periodic in $\psi$ such that $\Lambda_k(v,\psi)=\mathcal O(v)$ and $\Omega_k(v,\psi)=\mathcal O(1)$ as $v\to 0$ uniformly for all $\psi\in\mathbb R$.

Hence, for every $k\geq 2$ the right-hand side of \eqref{RSys} is $2\pi\varkappa$-periodic with respect to $S$ with zero average. Integrating \eqref{RSys} yields
\begin{gather*}
    \begin{pmatrix} v_k (\mathcal E,\theta,S)   \\ \psi_k (\mathcal E,\theta,S)  \end{pmatrix} = - \frac{1}{ s_0 }\int\limits_{0}^S \begin{pmatrix}\{ \mathcal F_k (\mathcal E,\theta,s)  - \hat{\mathcal F}_k (\mathcal E,\theta,s) \}_{\varkappa s}  \\  \{\mathcal G_k  (\mathcal E,\theta,s)- \hat{\mathcal G}_k (\mathcal E,\theta,s) \}_{\varkappa s}    \end{pmatrix} \, ds+ \begin{pmatrix} \tilde v_k(\mathcal E,\theta) \\ \tilde \psi_k(\mathcal E,\theta) \end{pmatrix},
\end{gather*}
where $\{Z\}_{\varkappa s}:=Z - \langle Z \rangle_{\varkappa s}$, and the functions $\tilde v_k, \tilde\psi_k$ are chosen such that $\langle v_k\rangle_{\varkappa S}=\langle\psi_k\rangle_{\varkappa S}=0$.
It can easily be checked that the functions $v_k(\mathcal E,\theta,S)$ and $\psi_k(\mathcal E,\theta,S)$ are smooth and periodic with respect to $\theta$ and $S$ such that $v_k(0,\theta,S)\equiv 0$.

The remainders $\tilde\Lambda_{N}(v,\psi,t)$ and $\tilde\Omega_{M}(v,\psi,t)$  have the following form:
\begin{eqnarray*}
   \tilde\Lambda_{N}
            & \equiv  &\sum_{k=N+1}^\infty t^{-\frac{k}{2q}}\Big(\mathcal F_k  -\frac{k-2q}{2q} v_{k-2q}\Big) \\
            &           & + \sum_{k=N+1}^\infty t^{-\frac{k}{2q}}\sum_{i+j=k} \Big(\big(\mathcal F_{j}+\delta_{j,2q}\frac{l}{q}\mathcal E\big)\partial_{\mathcal E}  + \mathcal G_{j} \partial_\theta +s_{j/2}\big(1-\frac{j}{2q}+\delta_{j,2q}\big)\partial_S\Big) v_i,\\
  \tilde\Omega_{M}  &\equiv& \sum_{k=M+1}^\infty t^{-\frac{k}{2q}}
    \Big(\mathcal G_k  -\frac{k-2q}{2q}\psi_{k-2q}\Big) \\
            &           &  + \sum_{k=M+1}^\infty t^{-\frac{k}{2q}} \sum_{i+j=k} \Big(\big(\mathcal F_{j} +\delta_{j,2q}\frac{l}{q}\mathcal E\big)\partial_{\mathcal E}  + \mathcal G_{j} \partial_\theta +s_{j/2}\big(1-\frac{j}{2q}+\delta_{j,2q}\big)\partial_S\Big)\psi_i .
\end{eqnarray*}
It is clear that $\tilde\Lambda_{N}(0,\psi,t)\equiv 0$ and $\tilde\Lambda_{N}(v,\psi,t) =\mathcal O(t^{-(N+1)/2q})$, $\tilde\Omega_{M}(v,\psi,t) =\mathcal O(t^{-(M+1)/2q})$ as $t\to\infty$ for all $\psi\in \mathbb R$ and $0\leq v\leq {\hbox{\rm const}}$.

\begin{Rem}
Let $l=1$ and $m_1,m_2\geq 2$ be the least natural numbers such that $\Omega_{m_1}(v,\psi)\not\equiv 0$, $\Omega_j(v,\psi)\equiv 0$ for $m_1<j<m_2$ and $\Omega_{m_2}(v,\psi)\not\equiv 0$. If $\Omega_{m_1}(0,\psi) \equiv 0$ and $\Omega_{m_2}(0,\psi)\not\equiv 0$, we take $l=1+m_2-m_1$ in \eqref{epsdel}. Then, in new variables, $\Omega_{m_0}(0,\psi)\not\equiv 0$, where $m_0\geq 2$ is the least natural number such that $\Omega_{m_0}(v,\psi)\not\equiv 0$ in \eqref{VEq}
\end{Rem}

From \eqref{VF} it follows that for all $\epsilon\in (0,1)$ there exists $t_0>1$ such that
\begin{gather}
\label{estVE}
    \begin{split}
    &  |V_N(\mathcal E,\theta,t)-\mathcal E|\leq \epsilon \mathcal E, \quad |\partial_{\mathcal E} V_N(\mathcal E,\theta,t)-1|\leq \epsilon, \quad  |\partial_{\mathcal \theta} V_N(\mathcal E,\theta,t)|\leq \epsilon, \\
    & |\Psi_M(\mathcal E,\theta,t)-\theta|\leq \epsilon,\quad |\partial_{\mathcal E} \Psi_M(\mathcal E,\theta,t)|\leq \epsilon, \quad  |\partial_\theta \Psi_M(\mathcal E,\theta,t)-1|\leq \epsilon
    \end{split}
 \end{gather}
for all $\theta\in \mathbb R$, $t\geq t_0$ and $\mathcal E\in [0, \mathcal E_0]$ with some $\mathcal E_0={\hbox {\rm const}}$. Hence, the mapping $(\mathcal E,\theta)\mapsto (v,\psi)$ is invertible for all $t\geq t_0$, $\psi\in\mathbb R$ and $v\in [0,\Delta_0]$ with $\Delta_0=(1-\epsilon)\mathcal E_0$. We choose $\mathcal E_0= E_0 t_0^{l/2q}$, then for all $t\geq t_0$ the transformation described by \eqref{epsdel} is valid for all $0\leq \mathcal E\leq \mathcal E_0$ and $\theta\in\mathbb R$.

Let $\mathcal D_0=\{(x,y): H_0(x,y)\leq E_0\}\cap D_0$. Then, we have the following.
\begin{Lem}
There exist $l\geq 0$ and $t_0>1$ such that for all $t\geq t_0$ and $(x,y)\in \mathcal D_0$  system \eqref{FulSys} can be transformed into \eqref{VEq} by the transformations \eqref{exch1}, \eqref{epsdel} and \eqref{VF}.
\end{Lem}

It is readily seen that the stability of the trivial solution $v(t)\equiv 0$ of system \eqref{VEq} ensures the stability of the fixed point $(0,0)$ in system \eqref{FulSys}. However, if $v(t)\equiv 0$ is unstable and $l=1$, then due to the damping factor in \eqref{epsdel} some additional estimates or assumptions may be required to guarantee the instability of the equilibrium in the original variables $(x,y)$.

\section{Two asymptotic regimes}
\label{sec3}
Let $2\leq n,m\leq 2q$ be the least natural numbers such that $\Lambda_n(v,\psi)\not\equiv 0$ and $\Omega_m(v,\psi)\not\equiv 0$. We take  $N\geq n$, $M=m$ and system \eqref{VEq} takes the form:
\begin{gather}\label{simsys}
    \frac{dv}{dt}= \sum_{k=n}^N t^{-\frac {k}{2q}} \Lambda_k(v,\psi) + \tilde\Lambda_{N}(v,\psi,t),   \quad
    \frac{d\psi}{dt}=  t^{-\frac {m}{2q}} \Omega_m(v,\psi) +    \tilde\Omega_{m}(v,\psi,t),
\end{gather}
Note that system \eqref{simsys} admits at least two asymptotic regimes with $v(t)$ in the vicinity of the zero.
One class of solutions has the phase difference $\psi(t)$ tending to a constant at infinity, while another has the unboundedly growing phase difference: $|\psi(t)|\to\infty$ as $t\to\infty$. The intermediate situation with bounded $\psi(t)$ is not considered in this paper. The type of solutions depends on the properties of the function $\Omega_m(v,\psi)$.

Let us assume that
\begin{gather*}
    \Omega_m(v,\psi)=\omega_{m0}(\psi)+\omega_{m1}(\psi)\sqrt v+\mathcal O(v)
\end{gather*}
as $v\to 0$ for all $\psi\in\mathbb R$, where $\omega_{m0}(\psi)\not\equiv 0$ and $\omega_{m1}(\psi)$ are $2\pi$-periodic functions.  Consider two different cases:
\begin{align}\label{as01}
& \exists\, \psi_\ast\in\mathbb R: \quad \omega_{m0}(\psi_\ast)=0, \quad \vartheta_m:=\omega_{m0}'(\psi_\ast)<0; \\
\label{as02}
&\exists\, \Delta_\ast\in (0,\Delta_0]: \quad \Omega_{m}(v,\psi)\neq 0 \quad \forall\, v\in [0,\Delta_\ast], \ \ \psi\in\mathbb R.
\end{align}
We have the following
\begin{Lem}
\label{Lem01}
Let assumption \eqref{as01} holds. Then system \eqref{simsys} has a particular solution $v(t)\equiv 0$, $\psi(t)\equiv\hat\psi(t)$ such that $\hat\psi(t)= \psi_\ast+\mathcal O(t^{-\kappa})$ as $t\to\infty$ with $\kappa={\hbox{\rm const}}>0$.
\end{Lem}
\begin{proof}
Consider the second equation of \eqref{simsys} with $v\equiv 0$.
It can easily be checked that $\psi_\ast$ is a stable equilibrium of the corresponding reduced equation: $d\psi/dt=t^{-m/2q}\Omega_{m}(0,\psi)$. Let us show that this solution is stable under the persistent perturbation $\tilde\Omega_{m}(0,\psi,t)$.
The change of the variable $\psi(t)=\psi_\ast+\phi(t)$ leads to the following equation:
\begin{gather}\label{phieq}
    \frac{d\phi}{dt}= t^{-\frac {m}{2q}} \Omega_m(0,\psi_\ast+\phi)+\tilde\Omega_{m}(0,\psi_\ast+\phi,t).
\end{gather}
 Consider  $\ell(t)=|\phi|$ as a Lyapunov function candidate for \eqref{phieq}. Its total derivative is given by
\begin{gather*}
    \frac{d\ell}{dt}\Big|_{\eqref{phieq}}=t^{-\frac {m}{2q}}\omega_m(\psi_\ast+\phi)    {\hbox{\rm sgn}} (\phi)+\tilde\Omega_{m}(0,\psi_\ast+\phi,t)    {\hbox{\rm sgn}} (\phi).
\end{gather*}
It can easily be checked that there exist $\delta_1>0$, $t_1\geq t_0$ and $M_1>0$ such that $\omega(\psi_\ast+\phi)    {\hbox{\rm sgn}} (\phi) \leq -|\vartheta_m||\phi|/2$ and $|\tilde\Omega_{m}(0,\psi_\ast+\phi,t)    {\hbox{\rm sgn}} (\phi)|\leq M_1 t^{-(m+1)/2q}$ for all $|\phi|\leq \delta_1$ and $t\geq t_1$. Hence, $d\ell/dt\leq t^{-m/2q} (-|\vartheta_m|\ell/2 + M_1 t^{-1/2q})$ as $t\geq t_1$. Integrating the last inequality in case $m=2q$ yields
\begin{gather*}
    0\leq \ell (\phi(t))\leq \ell (\phi(t_1)) \Big(\frac{t}{t_1}\Big)^{-\frac{|\vartheta_m|}{2}}+M_1 t^{-\frac{|\vartheta_m|}{2}}\int\limits_{t_1}^t \tau^{\frac{q|\vartheta_m|-1}{2q}-1}\, d\tau
\end{gather*}
with $|\phi(t_1)|\leq \delta_1$. We see that $ \phi(t)=\mathcal O(t^{-1/2q})$ as $t\to\infty$ if $q|\vartheta_m|>1$, $ \phi(t) =\mathcal O(t^{-|\vartheta_m|/2})$ if $q|\vartheta_m|<1$, and $ \phi(t) =\mathcal O(t^{-|\vartheta_m|/2}\log t)$ if $q|\vartheta_m|=1$. Similar estimates hold in case  $m< 2q$. Therefore, there exists a solution $\hat\phi(t)$ of equation \eqref{phieq} such that $\hat\phi(t)=o(1)$ as $t\to \infty$.
\end{proof}

\begin{Lem}
Let assumption \eqref{as02} holds. Then $|\psi(t)|\to\infty$ as $t\to\infty$ for solutions of \eqref{simsys} with $v(t)\in [0,\Delta_\ast]$ as $t\geq t_0$.
\end{Lem}
\begin{proof}
Since $\Omega_m(v,\psi)\neq 0$, it follows that $\Omega_m^-\leq | \Omega_m(v,\psi)|\leq \Omega_m^+$ for all $v\in [0,\Delta_\ast]$ and $\psi\in\mathbb R$ with $\Omega_m^\pm={\hbox{\rm const}}>0$. Then, there exists $t_1\geq t_0$ such that $\dot\psi \geq t^{-m/2q}\Omega_m^-/2$ as $t\geq t_1$ for all $v\in [0,\Delta_\ast]$, $\psi\in\mathbb R$ if $\Omega_m(v,\psi)>0$, or $\dot\psi \leq -t^{-m/2q}\Omega_m^-/2$ if $\Omega_m(v,\psi)<0$. Integrating the last inequalities, we obtain $|\psi(t)|\to \infty$ as $t\to \infty$.
\end{proof}

Note that the case of \eqref{as01} corresponds to a phase locking, while the case \eqref{as02} is associated with a phase drifting (see, for example,~\cite{RA46,CR88,AEC90,LJ13}). In both cases, the stability of the solution $v(t)\equiv 0$ and the equilibrium $(0,0)$ of system \eqref{FulSys} depends on the structure of the first equation in \eqref{simsys}. In the next sections, stability conditions are discussed separately for each asymptotic regime.

Before formulating the main results, let us introduce two more assumptions on the structure of simplified system \eqref{simsys}:
\begin{gather}
    \label{as1}
  \begin{split}
        &\Lambda_j(v,\psi)\equiv 0, \quad   j<n, \quad \Omega_i(v,\psi)\equiv 0, \quad i<m,\\
    &\Lambda_n(v,\psi)\equiv v\Big(\lambda_n(\psi)+\widetilde \lambda_{n}(v,\psi)+\delta_{n,2q}\frac{l}{q}\Big),
    \end{split}
\end{gather}
and
\begin{gather}
\label{as2}
        \begin{split}
&\Lambda_j(v,\psi)\equiv 0, \quad   j<n, \quad \Omega_i(v,\psi)\equiv 0, \quad i<m,\\
  &            \Lambda_z(v,\psi)\equiv v^{\frac{\sigma+1}{2}} \Big(\lambda_{z,\sigma}(\psi) + \widetilde \lambda_{z,\sigma}(v,\psi)\Big)+v\delta_{z,2q}\frac{l}{q},  \quad n\leq z< n+d,\\
& \Lambda_{n+d}(v,\psi)\equiv  v \Big(\lambda_{n+d}(\psi)+ \widetilde \lambda_{n+d}(v,\psi) +\delta_{n+d,2q}\frac{l}{q}\Big),
        \end{split}
\end{gather}
where
$\lambda_z(\psi)$, $\widetilde \lambda_z(v,\psi)$,
$\lambda_{z,\sigma}(\psi)$, $\widetilde\lambda_{z,\sigma}(v,\psi)$,
$\lambda_{n+d}(\psi)$, $\widetilde\lambda_{n+d}(v,\psi)$ are $2\pi$-periodic functions with respect to $\psi$, $d$ and $\sigma$ are integers such that $\sigma\geq 2$, $d\geq 1$,
\begin{gather*}
  \widetilde \lambda_{n}(v,\psi)=\mathcal O(v^{\frac 12}),  \quad \widetilde \lambda_{z,\sigma}(v,\psi) = \mathcal O(v^{\frac12}), \quad \widetilde \lambda_{n}(v,\psi)=\mathcal O(v^{\frac12})
\end{gather*}
as $v\to 0$ uniformly for all $\psi\in\mathbb R$.

\section{Phase locking}
\label{sec4}

In this section we discuss the stability of the equilibrium $(0,0)$ of system \eqref{FulSys} when assumption \eqref{as01} holds. In this case, the stability of the equilibrium relates to the properties of the particular solutions $(0,\hat\psi(t))$ to system \eqref{simsys}.

\begin{Th}\label{Th1}
Let system \eqref{FulSys} satisfy \eqref{H0}, \eqref{Has}, \eqref{Fas}, \eqref{pcond}, and  $2\leq n,m\leq 2q$ be integers such that assumptions \eqref{as01}, \eqref{as1} hold.
\begin{itemize}
  \item If $n,m<2q$, and $\lambda_n(\psi_\ast)<0$, then  the equilibrium $(0,0)$ is exponentially stable.
  \item If either $n<m=2q$, $\lambda_n(\psi_\ast)<0$ or $m\leq n=2q$, $\displaystyle \lambda_n(\psi_\ast)+ {l}/{q}<0$, then the equilibrium $(0,0)$  is polynomially stable.
  \item If $\lambda_n(\psi_\ast)+\widetilde\lambda_{n}(v,\psi_\ast)>0$ for all $v\in [0,\Delta_0]$,  then  the equilibrium $(0,0)$  is unstable.
\end{itemize}
\end{Th}
\begin{proof}
We choose $N=n$ and \eqref{simsys} takes the form:
\begin{gather}
\label{simsys1}
    \frac{dv}{dt}=   t^{-\frac{n}{2q}} \Lambda_n(v,\psi) + \tilde\Lambda_{n}(v,\psi,t),   \quad
    \frac{d\psi}{dt}=  t^{-\frac{m}{2q}} \Omega_m(v,\psi) +    \tilde\Omega_{m}(v,\psi,t).
\end{gather}
From Lemma~\ref{Lem01} it follows that system \eqref{simsys1} has the particular solution $v(t)\equiv 0$, $\psi(t)\equiv \hat\psi(t)$ such that $ \hat\psi(t)=\psi_\ast+o(1)$ as $t\to\infty$. Substituting $v(t)=[u(t)]^2$, $\psi(t)=\hat\psi(t)+\phi(t)$ into \eqref{simsys1} yields the following  system with an equilibrium at $(0,0)$:
\begin{gather}
    \begin{split}\label{vusys}
          2u  \frac{du}{dt}&=  t^{-\frac{n}{2q}} \Lambda_n(u^2,\hat\psi(t)+\phi) + \tilde\Lambda_{n}(u^2,\hat\psi(t)+\phi,t),   \\
            \frac{d\phi}{dt}&=  t^{-\frac{m}{2q}}\Big( \Omega_m(u^2,\hat\psi(t)+\phi)-\Omega_m(0,\hat\psi(t))\Big) +    \tilde\Omega_{m}(u^2,\hat\psi(t)+\phi,t)-\tilde\Omega_{m}(0,\hat\psi(t),t).
    \end{split}
\end{gather}
It can easily be checked that the eigenvalues $a_1(t)$ and $a_2(t)$ of the linearized system
    \begin{gather*}
    \frac{d}{dt}
    \begin{pmatrix} u \\ \phi \end{pmatrix}=
    {\bf A}(t)
    \begin{pmatrix} u \\ \phi \end{pmatrix},
\quad
    {\bf A}(t) =
    (1+o(1))
    \begin{pmatrix}
        \displaystyle\frac 12  t^{-\frac{n}{2q}} \Big(\lambda_n(\psi_\ast)+\delta_{n,2 q}\frac {l}{q} \Big)
            &
        0 \\
        \displaystyle t^{-\frac {m}{2q}}\omega_{m1}(\psi_\ast)
            &
        \displaystyle t^{-\frac {m}{2q}}\vartheta_m
    \end{pmatrix}, \quad t\to\infty,
\end{gather*}
have the following asymptotics: $a_1(t)=2^{-1} t^{-n/2q} (\lambda_n(\psi_\ast)+\delta_{n,2q}l/q  +o(1))$, $a_2(t)= t^{-m/2q}(\vartheta_m+o(1))$ as $t\to \infty$. Since $a_{1,2}(t)\to 0$ as $t\to \infty$, the linear stability analysis fails (see, for example,~\cite{OS20arxiv}). We investigate the stability by constructing suitable Lyapunov and Chetaev functions.

First, consider the case $\lambda_n(\psi_\ast)+\widetilde\lambda_{n}(v,\psi_\ast)>0$ for all $v\in [0,\Delta_0]$. It can easily be proved that the fixed point $(0,0)$ of system \eqref{vusys} is unstable by taking $J(u)=u^2$ as a Chetaev function candidate. The derivative of $J$ along the trajectories of system \eqref{vusys} is given by
\begin{gather*}
    \frac{dJ}{dt}\Big|_{\eqref{vusys}}=
    t^{-\frac{n}{2q}}u^2\Big(\lambda_n(\psi_\ast)+\widetilde\lambda_n(u^2,\psi_\ast)+\delta_{n, 2q}\frac {l}{q}+\mathcal O(\phi)+\mathcal O(t^{-\frac{1}{2q}})+\mathcal O(t^{-\kappa})\Big)
\end{gather*}
as $t\to \infty$ and  $\phi\to 0$ for all $u \in [0, \Delta_0^{1/2}]$.
Therefore, for all $\epsilon\in (0,1)$ there exist $\phi_1>0$ and $t_1\geq t_0$ such that
\begin{gather*}
    \frac{dJ}{dt}\Big|_{\eqref{vusys}}\geq
    t^{-\frac{n}{2q}} J\Big((1-\epsilon)\lambda_n^\ast+\delta_{n, 2q}\frac {l}{q}\Big)>0
\end{gather*}
for all  $u \in (0,  \Delta_0^{1/2} ]$,  $|\phi|\leq \phi_1$ and $t\geq t_1$, where $\lambda_n^\ast:=\min_{v\in[0,\Delta_0]}(\lambda_n(\psi_\ast)+\widetilde\lambda_n(v,\psi_\ast))>0$. Integrating the last inequality over $[t_1,t]$ yields
\begin{gather}
    \label{vest1}
        \begin{split}
          u^2(t) &
          \geq  u^2(t_1) \exp\left((1-\epsilon) \lambda_n^\ast\Big(t^{1-\frac{n}{2q}}-t_1^{1-\frac{n}{2q}}\Big)\frac{2q}{2q-n}  \right),
        \quad n <2q,  \\
        u^2(t)&
        \geq  u^2(t_1)  \Big(\frac{t}{t_1}\Big)^{(1-\epsilon)\lambda_n^\ast+\frac {l}{q}},
        \quad n = 2q.
\end{split}
\end{gather}
Hence, there exists $\varepsilon\in (0,  \Delta_0^{1/2}]$ such that for all $\delta\in (0,\varepsilon)$ the solution $(u(t),\phi(t))$ with initial data $|u(t_1)|\leq \delta_1$, $|\phi(t_1)|\leq \phi_1$ leaves $\varepsilon$-neighbourhood of $(0,0)$: $|u(t_\ast)|\geq \varepsilon$ at some $t_\ast>t_1$. From \eqref{vest1}, \eqref{estVE} and \eqref{epsdel} it follows that the trivial solution is also unstable in the original variables $(x,y)$. Indeed, if $n=2q$, we have
\begin{gather*}
    E(t)\geq
        \left(\frac{1-\epsilon}{1+\epsilon}\right)
        E(t_1) \Big(\frac{t}{t_1}\Big)^{(1-\epsilon)\lambda_n^\ast}, \quad t\geq t_1.
\end{gather*}

Note that the conditions that guarantee the stability of the fixed point $(0,0)$ depend on the values of $n$ and $m$.

{\bf 1}. Consider first the case $n=m$. Let $\lambda^\ast_{n}:=\lambda_n(\psi_\ast)+\delta_{n,2q}l/q<0$. If $\omega_{m1}(\psi_\ast)=0$, the stability can be proved by using a Lyapunov function of the form $L_0(u,\phi)=u^2+\phi^2\lambda^\ast_{n}(2\vartheta_m)^{-1}$.
The derivative of $L_0(u,\phi)$ with respect to $t$ along the trajectories of \eqref{vusys} is given by
\begin{gather*}
        \frac{dL_0}{dt}\Big|_{\eqref{vusys}}=
         \lambda_n^\ast t^{-\frac {n}{2q}}(u^2+\phi^2)
        \Big(1+\mathcal O(w) +\mathcal O(t^{-\frac{1}{2q}})+\mathcal O( t^{-\kappa})\Big)
\end{gather*}
as $w=\sqrt{u^2+\phi^2}\to 0$ and $t\to\infty$. Here, the asymptotic estimates $\mathcal O(t^{-1/2q})$ and $\mathcal O(w)$ are uniform with respect to $(u,\phi,t)$ in the domain $\mathcal I(w_\ast,t_\ast):=\{(u,\phi,t)\in\mathbb R^3: w\leq w_\ast, t\geq t_\ast\}$ with some positive constants $w_\ast$ and $t_\ast$. It can easily be checked that there exists $w_1>0$ and $t_1\geq t_0$ such that
\begin{gather*}
       A_0 w^2 \leq L_0(u,\phi)\leq B_0w^2, \quad
        \frac{dL_0}{dt}\Big|_{\eqref{vusys}}\leq -\frac{|\lambda_n^\ast|}{2} t^{-\frac{n}{2q}} w^2 \leq -  \gamma_0 t^{-\frac {n}{2q}}  L_0\leq 0
\end{gather*}
for all $(u,\phi,t)\in\mathcal I(w_1,t_1)$, where $A_0=\min\{1,\lambda^\ast_{n}(2\vartheta_m)^{-1} \}$, $B_0=\max\{1,\lambda^\ast_{n}(2\vartheta_m)^{-1} \}$, and $\gamma_0=|\lambda^\ast_{n}| (2B_0)^{-1}>0$. Hence, the fixed point $(0,0)$ of system \eqref{vusys} is stable. Moreover, by integrating the last inequality over $[t_1,t]$, we obtain
\begin{gather*}
        \begin{split}
            u^2(t)+\phi^2(t)\leq C_0 \exp\left(-    \frac{ 2q \gamma_0  }{2q-n} t^{1-\frac {n}{2q}}\right),  \quad & n=m< 2q, \\
  u^2(t)+\phi^2(t)\leq C_0 t^{-\gamma_0},  \quad & n=m= 2q,
\end{split}
\end{gather*}
where $C_0$ is a positive constant dependent on $u(t_1)$, $\phi(t_1)$ and $t_1$.
Therefore, the trajectories starting in the vicinity of the equilibrium
 tend to it as $t\to\infty$. In particular, if $n< 2q$, the equilibrium of system \eqref{vusys} and the particular solution $v(t)\equiv 0$, $\psi(t)\equiv \hat\psi(t)$ of system \eqref{simsys} are exponentially stable. If $n=2q$, the stability is polynomial. Taking into account \eqref{exch1}, \eqref{epsdel}, and \eqref{estVE}, we obtain the corresponding results on the stability of the equilibrium $(0,0)$ in system \eqref{FulSys}.

If $\omega_{m1}(\psi_\ast)\neq 0$, a Lyapunov function is constructed in the form $L_1(u,\phi)=u^2+\alpha_1(\phi-\beta_1 u)^2$, where  $\alpha_1=\lambda^\ast_{n} (4 \vartheta_m \beta_1 ^2)^{-1}>0$,  $\beta_1=2\omega_{m1}(\psi_\ast)(2\vartheta_m+\lambda^\ast_{n})^{-1}\neq 0$. Its derivative is given by
\begin{gather*}
\frac{dL_1}{dt}\Big|_{\eqref{vusys}}=\frac{\lambda^\ast_{n}}{2}  t^{-\frac {n}{2q}} (u^2+\beta^{-2}_1\phi^2)\Big(1+\mathcal O(w)+\mathcal O(t^{-\frac{1}{2q}})+\mathcal O( t^{-\kappa})\Big)
\end{gather*}
as $w\to 0$ and $t\to \infty$.
There exist $w_1>0$ and $t_1\geq t_0$ such that
\begin{gather*}
    A_1\big(u^2+\beta^{-2}_1\phi^2\big) \leq L_1(u,\phi)\leq B_1 \big(u^2+\beta_1^{-2}\phi^2\big), \quad
    \frac{dL_1}{dt}\Big|_{\eqref{vusys}}\leq
        -\frac{|\lambda^\ast_{n}|}{4} t^{-\frac{n}{2q}} \big(u^2+\beta_1^{-2}\phi^2\big)  \leq -\gamma_1 t^{-\frac {n}{2q}} L_1
\end{gather*}
for all $(u,\phi,t)\in\mathcal I(w_1,t_1)$, where $A_1=\lambda_n^\ast (4c_1)^{-1}$, $B_1=c_1(2\vartheta_m)^{-1}$, $\gamma_1=\lambda^\ast_{n}\vartheta_m(2c_1)^{-1}>0$, $c_1=\lambda_n^\ast+4\vartheta_m$. Hence, as in the previous case, the fixed point of system \eqref{FulSys} is asymptotically stable.

{\bf 2.} Consider the case $m<n\leq 2q$. Let us prove the stability of the equilibrium when $\lambda_n^\ast<0$.  Consider $L_2(u,\phi,t)= t^{(n-m)/(2q)} u^{2p}+ \alpha_2 ( \phi+\beta_2 u)^{2p} $ as a Lyapunov function candidate. Here $p\geq 1$ is an integer such that  $\lambda_{np}^\ast:=\lambda_n^\ast+(n-m)/(2pq)<0$, $\alpha_2:=\lambda_{np}^\ast (2\vartheta_m)^{-1}>0$, $\beta_2:=\omega_{m1}(\psi_\ast)\vartheta_m^{-1}$.  The derivative of $L_2(u,\phi,t)$ along the trajectories of system \eqref{vusys} is given by
    \begin{gather*}
        \frac{dL_2}{dt}\Big|_{\eqref{vusys}}=  p
        t^{-\frac {m}{2q}} \Big(\lambda_n^\ast u^{2p}+\lambda_{np}^{\ast} (\phi+\beta_2 u)^{2p}+\frac{n-m}{2pq} u^{2p}t^{\frac{n}{2q}-1}+ \mathcal O(w^{2p+1})\Big)(1+
\mathcal O( t^{-\frac{1}{2q}})+\mathcal O( t^{-\kappa}))
    \end{gather*}
as $w\to 0$ and $t\to\infty$. It can easily be checked that $L_2$ and $d L_2/dt$ satisfy the following estimates:
\begin{gather}
\label{L2}
\begin{split}
A_2 w^{2p}\leq L_2(u,\phi,t)\leq  t^{\frac{n-m}{2q}}B_2 \Big(u^{2p}+  ( \phi+\beta_2 u)^{2p}\Big), \\
\frac{dL_2}{dt}\Big|_{\eqref{vusys}}\leq - \frac {p|\lambda_{np}^\ast|}{2} t^{-\frac {m}{2q}}  \Big(u^{2p}+ ( \phi+\beta_2 u)^{2p}\Big) \leq  -  p \gamma_2 t^{-\frac {n}{2q}} L_2\leq 0
\end{split}
\end{gather}
for all $(u,\phi,t)\in\mathcal I(w_2,t_2)$ with some constants $w_2>0$ and $t_2\geq t_0$, where $\gamma_2=|\lambda^\ast_{np}|(2B_2)^{-1}>0$, $A_2=4^{1-p}\min\{1,\alpha_2\}/(1+\alpha_2 (\beta_2^2/2)^p)$, $B_2=\max\{1,\alpha_2\}$. Hence, the equilibrium $(0,0)$ is stable. In addition, by integrating \eqref{L2}, we obtain the following estimates
\begin{gather*}
        \begin{split}
            u^2(t)+\phi^2(t)\leq C_2 \exp\left(-  \frac{2q  \gamma_2 }{2q-n} t^{1-\frac {n}{2q}} \right),  \quad & n< 2q, \\
  u^2(t)+\phi^2(t)\leq C_2 t^{-\gamma_2},  \quad & n= 2q,
\end{split}
\end{gather*}
with initial data $u^2(t_2)+\phi^2(t_2)< w_2^2$ and $C_2={\hbox {\rm const}}>0$.
Therefore, if $m<n<2q$ and $\lambda_n(\psi_\ast)<0$, the equilibrium of system \eqref{FulSys} is exponentially stable; if $m<n=2q$ and $\lambda_n(\psi_\ast)+l/(2q)<0$, the equilibrium is polynomially stable.

{\bf 3}. Finally, consider the case $n<m\leq 2q$. Let $\lambda_n(\psi_\ast)<0$ and $p\geq 1$ be an integer such that $\vartheta_{mp}^\ast:=\vartheta_m+(m-n)/(4pq)<0$.
We use
$L_3(u,\phi,t)=t^{(m-n)/(2q)}\phi^{2p}+\alpha_3 u^{2p}+\beta_3 u\phi^{2p-1}+c_3 \phi^{2p}$
with
$\alpha_3:= 2\vartheta^\ast_{mp} (\lambda_n(\psi_\ast))^{-1}>0$,
$\beta_3:=4p\omega_{m1}(\psi_\ast)(\lambda_n(\psi_\ast))^{-1}$,
$c_3={\hbox{\rm const}}>0$
as a Lyapunov function candidate.
The derivative of $L_3$ with respect to $t$ along the trajectories of system \eqref{vusys} is given by
\begin{gather*}
    \frac{dL_3}{dt}\Big|_{\eqref{vusys}}=2p t^{-\frac {n}{2q}} \Big(\vartheta_{mp}^\ast  u^{2p}+\vartheta_m \phi^{2p}+\frac{m-n}{4pq} t^{\frac{m}{2q}-1}\phi^{2p}+\mathcal O(w^{2p+1})\Big)(1+\mathcal O( t^{-\frac{1}{2q}})+\mathcal O( t^{-\kappa}))
\end{gather*}
as $w\to 0$ and $t\to\infty$.
It follows that there exist $w_3>0$, $t_3\geq t_0$ and $c_3>0$ such that
\begin{gather*}
A_3 w^{2p}\leq L_3(u,\phi,t)\leq   B_3t^{\frac{m-n}{2q}} (u^{2p}+\phi^{2p}), \quad
\frac{dL_3}{dt}\Big|_{\eqref{vusys}}\leq - p |\vartheta_{mp}^\ast| t^{-\frac{n}{2q}}  (u^{2p}+\phi^{2p}) \leq  -p \gamma_3 t^{-\frac {m}{2q}} L_3
\end{gather*}
for all $(u,\phi,t)\in\mathcal I(w_3,t_3)$, where
$A_3=2^{-p}\min\{1,\alpha_3\}$,
$B_3=2\max\{1, \alpha_3\}$,
$\gamma_3= |\vartheta_{mp}^\ast| B_3^{-1}>0$.
Hence, the solution $v(t)\equiv0$, $u(t)\equiv 0$ is stable. Furthermore, integrating the last inequality over $[t_3,t]$ yields
\begin{gather*}
        \begin{split}
         u^{2}(t)+\phi^2(t)\leq C_3
\exp\Big(- \frac{2 q \gamma_3}{2q-m} t^{1-\frac {m}{2q}} \Big),\quad
            & m< 2q, \\
        u^2(t)+\phi^2(t)\leq C_3 t^{-\gamma_3},  \quad
            & m=2q,
    \end{split}
\end{gather*}
where $C_3$ is a positive constant depending on the initial conditions: $u(t_3)$, $\phi(t_3)$ and $t_3$. Therefore,  if $n<m<2q$ and $\lambda_n(\psi_\ast)<0$, the equilibrium $(0,0)$ of system \eqref{FulSys} is exponentially stable; if $n<m=2q$ and $\lambda_n(\psi_\ast)<0$, the equilibrium is polynomially stable.
\end{proof}

\begin{Rem}
Let us note that stability of the equilibrium $(0,0)$ of system \eqref{FulSys} is not justified if $m\leq n=2q$, $0<\lambda_n(\psi_\ast)+l/q<l/q$, and $\omega(E)\not\equiv {\hbox{\rm const}}$ $(l\neq 0)$. It follows from the proof of Theorem~\ref{Th1} that the particular solution $v(t)\equiv 0$, $\psi(t)\equiv \hat\psi(t)$ of system \eqref{vusys} is unstable, and  $E(t)t^{l/q}\geq \Delta_0$ at some $t\geq t_\ast$ for arbitrarily small initial data.  However, such estimate cannot ensure the instability of the equilibrium in the original coordinates. In this case, we can say that the equilibrium $(0,0)$ of system \eqref{FulSys} is unstable with the weight $t^{l/2q}$.
\end{Rem}

Now, we consider the case when  $\Lambda_n(v,\psi)$ is nonlinear with respect to $v$. Define $\nu=d/(q (\sigma-1))>0$. Then we have the following:

\begin{Th}\label{Th2}
Let system \eqref{FulSys} satisfy \eqref{H0}, \eqref{Has}, \eqref{Fas}, \eqref{pcond}, and $2\leq n,m\leq 2q$, $\sigma\geq 2$, $d\geq 1$ be integers such that  assumptions \eqref{as01}, \eqref{as2} hold.
\begin{itemize}
\item If $n+d\leq 2q$,  $\lambda_{n+d}(\psi_\ast)+\widetilde\lambda_{n+d}(v,\psi_\ast)>0$ and $\lambda_{n,\sigma}(\psi_\ast)+\widetilde\lambda_{n,\sigma}(v,\psi_\ast)>0$ for all $v\in [0,\Delta_0]$,
        then the equilibrium $(0,0)$ is unstable.
\item If $m\leq n+d= 2q$ and
        either $\displaystyle\lambda_{n+d}(\psi_\ast)+\nu+l/q<0$
        or $\displaystyle \lambda_{n+d}(\psi_\ast)+\nu+l/q>0$, $\lambda_{n,\sigma}(\psi_\ast)<0$,
    then the equilibrium $(0,0)$ is polynomially stable.
  \item If $m\leq n+d< 2q$ and
  \begin{itemize}
    \item $\lambda_{n+d}(\psi_\ast)<0$, then the equilibrium $(0,0)$ is exponentially stable.
    \item $\displaystyle \lambda_{n+d}(\psi_\ast)>0$, $\lambda_{n,\sigma}(\psi_\ast)<0$, then the equilibrium $(0,0)$ is polynomially stable.
  \end{itemize}
 \item If $n+d<m=2q$ and either $\displaystyle \lambda_{n+d}(\psi_\ast)<0$ or
    $\displaystyle \lambda_{n+d}(\psi_\ast)>0$, $\lambda_{n,\sigma}(\psi_\ast)<0$, then the equilibrium $(0,0)$ is polynomially stable.
    \item If $n+d<m<2q$ and
  \begin{itemize}
    \item $\lambda_{n+d}(\psi_\ast)<0$, then the equilibrium $(0,0)$ is exponentially stable.
    \item $\displaystyle \lambda_{n+d}(\psi_\ast)>0$, $\lambda_{n,\sigma}(\psi_\ast)<0$, then the equilibrium $(0,0)$ is polynomially stable. \end{itemize}
 \end{itemize}
\end{Th}
\begin{proof}
We choose $N=n+d$; then system \eqref{simsys} takes the form:
\begin{gather}
\label{simsys2}
    \frac{dv}{dt}=   \sum_{k=n}^{n+d}t^{-\frac{k}{2q}} \Lambda_k(v,\psi) + \tilde\Lambda_{n+d}(v,\psi,t),   \quad
    \frac{d\psi}{dt}=  t^{-\frac{m}{2q}} \Omega_m(v,\psi) +    \tilde\Omega_{m}(v,\psi,t).
\end{gather}
Let us show that the particular solution $(0,\hat\psi(t))$ of system \eqref{simsys2} is unstable if $\lambda_{n+d}(\psi_\ast)+\widetilde\lambda_{n+d}(v,\psi_\ast)>0$, $\lambda_{n,\sigma}(\psi_\ast)+\widetilde\lambda_{n,\sigma}(v,\psi_\ast)>0$ for all $v\in [0,\Delta_0]$. The proof is similar to that of Theorem~\ref{Th1}.
Substituting $v(t)=[u(t)]^2$ and $\psi(t)=\hat\psi(t)+\phi(t)$ into \eqref{simsys2} yields the following system with an equilibrium at $(0,0)$:
\begin{gather}
    \begin{split}\label{vusys2}
          2u  \frac{du}{dt}&=  \sum_{k=n}^{n+d} t^{-\frac{k}{2q}} \Lambda_k(u^2,\hat\psi(t)+\phi) + \tilde\Lambda_{n+d}(u^2,\hat\psi(t)+\phi,t),   \\
            \frac{d\phi}{dt}&=  t^{-\frac{m}{2q}}\Big( \Omega_m(u^2,\hat\psi(t)+\phi)-\Omega_m(0,\hat\psi(t))\Big) +    \tilde\Omega_{m}(u^2,\hat\psi(t)+\phi,t)-\tilde\Omega_{m}(0,\hat\psi(t),t).
    \end{split}
\end{gather}
Consider $J(u)=u^2$ as a Chetaev function candidate. The derivative of $J$ along the trajectories of system \eqref{vusys2} is given by
\begin{eqnarray*}
    \frac{dJ}{dt}\Big|_{\eqref{vusys2}}&=&
    t^{-\frac{n}{2q}}u^{\sigma+1}\Big(\lambda_{n,\sigma}(\psi_\ast)+ \widetilde\lambda_{n,\sigma}(u^2,\psi_\ast)+\delta_{n, 2q}\frac {l}{q}+\mathcal O(\phi )+\mathcal O(t^{-\frac{1}{2q}})+\mathcal O(t^{-\kappa})\Big) \\
    &&+  t^{-\frac{n+d}{2q}}u^2\Big(\lambda_{n+d}(\psi_\ast)+ \widetilde\lambda_{n+d}(u^2,\psi_\ast)+\delta_{n+d, 2q}\frac {l}{q}+\mathcal O(\phi )+\mathcal O(t^{-\frac{1}{2q}})+\mathcal O(t^{-\kappa})\Big)
\end{eqnarray*}
as $t\to \infty$ and  $\phi\to 0$ for all $u\in [0, \Delta_0^{1/2}]$.
Therefore, for all $\epsilon\in (0,1)$ there exist $\phi_1>0$ and $t_1\geq t_0$ such that
\begin{gather*}
    \frac{dJ}{dt}\Big|_{\eqref{vusys2}}\geq
    t^{-\frac{n+d}{2q}} J\Big((1-\epsilon)\lambda_{n+d}^\ast+\delta_{n+d, 2q}\frac {l}{q}\Big)>0
\end{gather*}
for all  $u\in (0,  \Delta_0^{1/2}]$,  $|\phi|\leq \phi_1$ and $t\geq t_1$, where $\lambda_{n+d}^\ast=\min_{v\in[0,\Delta_0]}(\lambda_{n+d}(\psi_\ast)+\widetilde\lambda_{n+d}(v,\psi_\ast))>0$. Integrating the last inequality over $[t_1,t]$ yields
\begin{gather*}
        \begin{split}
          u^2(t) &
          \geq  u^2(t_1) \exp\left((1-\epsilon) \lambda_{n+d}^\ast \Big(t^{1-\frac{n+d}{2q}}-t_1^{1-\frac{n+d}{2q}}\Big)\frac{2q}{2q-n-d}  \right),
        \quad n +d <2q,  \\
        u^2(t)&
        \geq  u^2(t_1)  \Big(\frac{t}{t_1}\Big)^{(1-\epsilon)\lambda_{n+d}^\ast+\frac {l}{q}},
        \quad n+d = 2q.
\end{split}
\end{gather*}
Hence, for all $\delta\in (0,\Delta_0^{1/2})$ the solution $(u(t),\phi(t))$ of system \eqref{simsys2} with initial data $|u(t_1)|\leq \delta$, $|\phi(t_1)|\leq \phi_1$ satisfies $[u(t)]^2\geq \Delta_0$ at some $t>t_1$. From \eqref{vest1}, \eqref{estVE} and \eqref{epsdel} it follows that the fixed point $(0,0)$ of system \eqref{FulSys} is unstable.

To proof the stability, consider the behaviour of trajectories in the vicinity of the particular solution $v(t)\equiv0$, $\psi(t)\equiv \hat\psi(t)$ of system \eqref{simsys2}.
The change of variables $v(t)=t^{-\nu}[\xi(t)]^2$, $\psi(t)=\hat\psi(t)+\eta(t)$ yields
\begin{gather}
\label{xesys}
    \frac{d\xi}{dt}=t^{-\frac{n+d}{2q}} Q_{n+d}(\xi,\eta)+\tilde Q_{n+d}(\xi,\eta,t), \quad \frac{d\eta}{dt}=t^{-\frac{m}{2q}} P_m(\eta)+\tilde P_{m}(\xi,\eta,t),
\end{gather}
where
\begin{eqnarray*}
    Q_{n+d}(\xi,\eta)&\equiv &\frac{\xi}{2}\Big(\lambda_{n+d}(\psi_\ast+\eta)+\delta_{n+d,2q} \Big(\nu+\frac{l}{q}\Big)  + \lambda_{n,\sigma}(\psi_\ast+\eta) \xi^{\sigma-1}\Big), \\
     P_{m}(\eta)&\equiv& \Omega_m(0,\psi_\ast+\eta)-\Omega_m(0,\psi_\ast)=\vartheta_m \eta+\mathcal O(\eta^2), \quad \eta\to 0.
\end{eqnarray*}
The functions $\tilde Q_{n+d}(\xi,\eta,t)$, $\tilde P_m(\xi,\eta,t)$ satisfy the following estimates:
\begin{gather*}
    |\tilde Q_{n+d}(\xi,\eta,t)|\leq K t^{-\frac{n+d}{2q}} \rho \big(t^{-\frac{1}{2q}}+ t^{-\kappa}+ t^{-\frac{\nu}{2}}\big),
        \quad
    |\tilde P_{m}(\xi,\eta,t)|\leq K t^{-\frac{m}{2q}}   \rho  \big(t^{-\frac{1}{2q}} + t^{-\kappa}+ t^{-\frac{\nu}{2}}\big)
\end{gather*}
for all $(\xi,\eta,t)\in \mathcal J(\rho_\ast,t_\ast):=\{(\xi,\eta,t)\in\mathbb R^3: \rho=\sqrt{\xi^2+\eta^2}\leq \rho_\ast, t\geq t_\ast\}$ with some constants $t_\ast\geq t_0$, $\rho_\ast>0$ and $K>0$.

We divide the remainder of the proof into three parts.

{\bf 1.} First, consider the case  $n+d=m\leq 2q$.
Let $\lambda_{n\nu}^\ast:=\lambda_{n+d}(\psi_\ast)+\delta_{n+d,2q}  (\nu+ {l}/{q} ) <0$.
By using the Lyapunov function $W_1(\xi,\eta)=\xi^2+\eta^2(2\vartheta_m)^{-1}\lambda_{n\nu}^\ast$, it can be shown that the equilibrium $(0,0)$ of system \eqref{xesys} is stable.
The derivative of $W_1(\xi,\eta)$ is given by
$d W_1/dt= \lambda_{n\nu}^\ast t^{-n/(2q)}\rho^2(1+\mathcal O(\rho)+\mathcal O(t^{-1/2q}))$ as $\rho\to 0$ and $t\to\infty$.  It follows, as in the proof of Theorem~\ref{Th1}, that there exist $\rho_1<\rho_\ast$ and $t_1\geq t_\ast$ such that the solution $(\xi(t),\eta(t))$ of system \eqref{xesys} with initial data $\xi^2(t_1)+\eta^2(t_1)\leq \rho_1^2$ satisfies the following inequalities:
\begin{eqnarray*}
    \xi^2(t)+\eta^2(t)\leq K_1 \exp \Big(-\frac{2q\gamma_1 }{2q-m} t^{1-\frac{m}{2q}}\Big), &\quad & n+d=m<2q, \\
    \xi^2(t)+\eta^2(t)\leq K_1   t^{-\gamma_1}, &\quad &n+d=m=2q,
\end{eqnarray*}
as $t\geq t_1$, where $\gamma_1=|\lambda_{n\nu}^\ast\vartheta_m|/ (|\lambda_{n\nu}^\ast|+2|\vartheta_m|)$,  $K_1$ is a positive parameter depending on $\xi(t_1)$, $\eta(t_1)$, $t_1$.
Hence, the fixed point $(0,0)$ of system \eqref{FulSys} is exponentially stable if $m<2q$, $\lambda_{n+d}(\psi_\ast)<0$, and polynomially stable if $m=2q$, $\lambda_{n+d}(\psi_\ast)+\nu+l/q<0$.

If $\lambda_{n+d}(\psi_\ast)+\delta_{n+d,2q}  (\nu+ {l}/{q} ) >0$ and $\lambda_{n,\sigma}(\psi_\ast)<0$, the equation $Q_{n+d}(\xi,0)=0$ has a solution
$\xi_\ast:=(\lambda_{n \nu}^\ast/|\lambda_{n,\sigma}(\psi_\ast)|)^{1/(\sigma-1)}>0$
such that $\partial_\xi Q_{n+d}(\xi_\ast,0)=-  (\sigma-1) \lambda_{n \nu}^\ast/2<0$.
The change of variable $\xi(t)=\xi_\ast+\zeta(t)$ transforms system \eqref{xesys} into
\begin{gather}
        \label{zesys}
            \frac{d\zeta}{dt}=t^{-\frac{n+d}{2q}} Q_{n+d}(\xi_\ast+\zeta,\eta)+\tilde Q_{n+d}(\xi_\ast+\zeta,\eta,t),
        \quad
            \frac{d\eta}{dt}=t^{-\frac{m}{2q}} P_m(\eta)+\tilde P_{m}(\xi_\ast+\zeta,\eta,t).
\end{gather}
It can easily be checked that the unperturbed system
\begin{gather}
        \label{zesysun}
            \frac{d\zeta}{dt}=t^{-\frac{n+d}{2q}} Q_{n+d}(\xi_\ast+\zeta,\eta),
        \quad
            \frac{d\eta}{dt}=t^{-\frac {m}{2q}} P_m(\eta)
\end{gather}
has a stable trivial solution $\zeta(t)\equiv 0$, $\eta(t)\equiv 0$. Let us show that this solution is stable with respect to the perturbations $\tilde Q_{n+d}$ and $\tilde P_m$. Consider $\ell(\zeta,\eta)=(\zeta^2+\eta^2)/2$ as a Lyapunov function candidate. Its derivative along the trajectories of system \eqref{zesys} is given by
\begin{gather}\label{ellineq}
    \frac{d\ell}{dt}\Big|_{\eqref{zesys}}= t^{-\frac{m}{2q}}\Big( \zeta Q_{n+d}(\xi_\ast+\zeta,\eta) + \eta P_m(\eta)\Big)+\zeta\tilde Q_{n+d}(\xi_\ast+\zeta,\eta,t)+ \eta \tilde P_{m}(\xi_\ast+\zeta,\eta,t).
\end{gather}
There exists $\varrho_1>0$ such that
$\zeta Q_{n+d}(\xi_\ast+\zeta,\eta) + \eta P_m(\eta)\leq - \gamma_1^\ast\varrho^2/2$ and
$\zeta\tilde Q_{n+d}+ \eta \tilde P_{m}\leq 2K \varrho t^{-\varsigma-m/2q}$
for all $(\zeta,\eta,t)$: $\varrho:=\sqrt{\zeta^2+\eta^2}\leq \varrho_1$ and $t\geq t_\ast$, where $\gamma_1^\ast=\min\{|\partial_\xi Q_{n+d}(\xi_\ast,0)|,|\vartheta_m|\}>0$, $\varsigma=\min\{1/(2q) ,\kappa,\nu/2\}$.
Hence, for all $\varepsilon\in (0,\varrho_1)$ there exist
$\delta>0$ and $t_1= t_\ast+((8 K/(\delta\gamma_1^\ast))^{{1}/{\varsigma}}$ such that
\begin{gather*}
  \frac{d\ell}{dt}\Big|_{\eqref{zesys}}\leq -t^{-\frac{m}{2q}} \frac{\varrho^2}{2} \Big(\gamma_1^\ast - \frac{4K}{ \delta t_1^{\varsigma}}\Big)\leq  -t^{-\frac{m}{2q}} \frac{\gamma_1^\ast\varrho^2}{4} <0
\end{gather*}
for all $(\zeta,\eta)$: $\delta\leq \varrho\leq \varepsilon$ and $t\geq t_1$. This implies that any solution $\zeta(t)$, $\eta(t)$ with initial data $\zeta^2(\tau_1)+\eta^2(\tau_1)\leq \delta^2$ with $\tau_1\geq t_1$ cannot leave the domain $\{(\zeta,\eta): \varrho\leq \varepsilon\}$ as $t>\tau_1$.
Furthermore, it follows from \eqref{ellineq} that $d\ell/ dt\leq t^{-m/2q} (-\gamma_1^\ast \ell + 2K \varrho_1 t^{-\varsigma}) $ as $\varrho\leq \varrho_1$ and $t\geq t_\ast$. Integrating the last inequality in the case $m=2q$, we get
\begin{gather*}
0\leq \ell(\zeta(t),\eta(t))\leq \ell(\zeta(t_\ast),\eta(t_\ast)) t^{-\gamma_1^\ast}+ 2K \varrho_1 t^{-\gamma_1^\ast} \int\limits_{t_\ast}^t \tau^{\gamma_1^\ast-\varsigma-1}\, d\tau, \quad t\geq t_\ast.
\end{gather*}
Hence, $\ell(\zeta(t),\eta(t))=\mathcal O(t^{-\gamma_1^\ast})+\mathcal O(t^{-\varsigma})$ as $t\to\infty$ for solutions $\zeta(t),\eta(t)$ starting from $\{(\zeta,\eta): \varrho<\varrho_1\}$. Similar estimate holds in the case $m<2q$. Returning to the variables $(v,\psi)$, we see that $v(t)=\mathcal O(t^{-\nu})$ and $\psi(t)=\psi_\ast(t)+\mathcal O(t^{-\gamma_1^\ast/2})+\mathcal O(t^{-\varsigma/2})$ as $t\to \infty$. Thus, the fixed point $(0,0)$ of system \eqref{FulSys} is polynomially stable.

{\bf 2.} Consider the case $m<n+d\leq 2q$. Let us show that if $\lambda_{n\nu}^\ast<0$, the trivial solution $\xi(t)\equiv 0$, $\eta(t)\equiv 0$ of system \eqref{xesys} is stable. Consider $W_2(\xi,\eta,t)= t^{(n+d-m)/2q}\xi^{2p}+ b_2\eta^{2p}$ as a Lyapunov function, where $p\geq 1$ is an integer such that $\lambda_{n\nu p}^\ast:=\lambda_{n \nu}^\ast+\delta_{n+d,2q}(n+d-m)(2pq)^{-1}<0$ and $b_2:=\lambda_{n\nu p}^\ast (2\vartheta_m)^{-1}>0$. The derivative of $W_2(\xi,\eta,t)$ along the trajectories of system \eqref{xesys} is given by
\begin{eqnarray*}
    \frac{dW_2}{dt}\Big|_{\eqref{xesys}}
    &=&
        p t^{-\frac{m}{2q}} \Big(\lambda_{n \nu}^\ast \xi^{2p}+\lambda_{n\nu p}^\ast \eta^{2p}+\frac{n+d-m}{2qp} \xi^{2p} t^{\frac {n+d}{2q}-1}+\mathcal O(\rho^{2p+1})\Big)(1 + \mathcal O( t^{-\varsigma}))
\end{eqnarray*}
as $\rho\to 0$ and $t\to\infty$.
Therefore, there exists $0<\rho_2\leq \rho_\ast$ and $t_2\geq t_\ast$ such that
\begin{gather*}
A_2 \rho^{2p} \leq W_2(\xi,\eta,t)\leq   B_2 t^{\frac{n+d-m}{2q}} (\xi^{2p}+\eta^{2p}), \quad
\frac{dW_2}{dt}\Big|_{\eqref{xesys}}\leq -\frac{p \lambda_{n\nu p}^\ast}{2} t^{-\frac {m}{2q}} (\xi^{2p}+\eta^{2p})\leq  -p \gamma_2 t^{-\frac{n+d}{2q}} W_2
\end{gather*}
for all $(\xi,\eta,t)\in\mathcal J(\rho_2,t_2)$, where
$A_2=2^{1-p}\min\{1,b_2\}$,
$B_2=\max\{1,b_2\}$,
$\gamma_2=\lambda_{n\nu p}^\ast (2B_2)^{-1}>0$.
Integrating the last inequality, we obtain the following estimates:
\begin{eqnarray*}
    \xi^2(t)+\eta^2(t)\leq  K_2 \exp\Big(- \frac{2q\gamma_2 }{2q-n-d}t^{1-\frac {n+d}{2q}} \Big), &\quad & m<n+d< 2q,\\
   \xi^2(t)+\eta^2(t) \leq  K_2 t^{-\gamma_2}, &\quad & m<n+d= 2q,
\end{eqnarray*}
as $t\geq t_2$, where $K_2$ is a positive constant depending on $\xi(t_2)$, $\eta(t_2)$ and $t_2$. Thus, the equilibrium $(0,0)$ of system \eqref{FulSys} is exponentially stable if $m<n+d<2q$, $\lambda_{n+d}(\psi_\ast)<0$, and polynomially stable if $m<n+d=2q$, $\lambda_{n+d}(\psi_\ast)+\nu+l/q<0$.

If $\lambda_{n\nu}^\ast>0$ and $\lambda_{n,\sigma}(\psi_\ast)<0$, the stability of the particular solution $(0,\hat\psi(t))$ of system \eqref{simsys2}  follows from the behaviour of the trajectories of system \eqref{xesys} in the vicinity of the point $(\xi_\ast,0)$. We construct a Lyapunov function for \eqref{zesys} in the form $\ell_2(\zeta,\eta,t)=t^{(n+d-m)/2q}\zeta^{2p}+\eta^{2p} $, where
$p\geq 1 $ is an integer such that $Q^\ast_{np}:=\partial_\xi Q_{n+d}(\xi_\ast,0)+(n+d-m)/(4pq)<0$.
It can easily be checked that there exists $\varrho_2>0$ such that
\begin{eqnarray*}
    \frac{d\ell_2}{dt}\Big|_{\eqref{zesys}}&=&
        2 p t^{-\frac{m}{2q}}\Big(   \zeta^{2p-1} Q_{n+d}(\xi_\ast+\zeta,\eta)+ \eta^{2p-1} P_m(\eta)+\frac{n+d-m}{4qp}t^{\frac{n+d}{2q}-1}\zeta^{2p}\Big)\\
    &&+2p \Big(t^{\frac{n+d-m}{2q}} \zeta^{2p-1}\tilde Q_{n+d}(\xi_\ast+\zeta,\eta,t)+  \eta^{2p-1}\tilde P_{m}(\xi_\ast+\zeta,\eta,t)\Big) \\
     &\leq &  -  pt^{-\frac{m}{2q}} \varrho^{2p-1}\Big(   \gamma_2^\ast   \varrho  - 4  K   t^{-\varsigma}\Big)
 \end{eqnarray*}
as $\varrho\leq \varrho_2$ and $t\geq t_\ast$, where $\gamma_2^\ast=\min\{|Q_{np}^\ast|,|\vartheta_m|\}$.
Therefore, for all $\varepsilon\in (0,\varrho_2)$ there exist
$\delta>0$ and $t_2= t_\ast+((8 K/(\delta\gamma_2^\ast))^{{1}/{\varsigma}}$ such that
\begin{gather*}
  \frac{d\ell_2}{dt}\Big|_{\eqref{zesys}}\leq -p t^{-\frac{m}{2q}} \varrho^{2p} \Big(\gamma_2^\ast - \frac{4K}{ \delta t_2^{\varsigma}}\Big)\leq  -t^{-\frac{m}{2q}}p\varrho^{2p}  \frac{\gamma_2^\ast}{2} <0
\end{gather*}
for all $(\zeta,\eta)$: $\delta\leq \varrho\leq \varepsilon$ and $t\geq t_2$. This implies that the trivial solution of \eqref{zesysun} is stable under disturbances $\tilde Q_{n+d}$ and $\tilde P_m$.
In particular, any solution $\zeta(t)$, $\eta(t)$ of system \eqref{zesys} with initial data $\zeta^2(\tau_1)+\eta^2(\tau_1)\leq \varrho_2^2/4$ at $\tau_1\geq t_1$ cannot leave the domain $\{(\zeta,\eta): \varrho<\varrho_2\}$ for all $t>\tau_1$.
Returning to the variables $v$ and $\psi$, we see that the particular solution $(0,\psi(t))$ of system \eqref{simsys2} is stable and $v(t)=\mathcal O(t^{-\nu})$ as $t\to \infty$. Hence, the fixed point $(0,0)$ of system \eqref{FulSys} is polynomially stable.

{\bf 3.} Finally, consider the case $n+d<m\leq 2q$. Let $\lambda_{n+d}(\psi_\ast)<0$. In this case, we use
$W_3(\xi,\eta,t)= a_3\xi^{2p}+ t^{(m-n-d)/2q}\eta^{2p}$
as a Lyapunov function candidate, where
$p\geq 1 $ is an integer such that
$\vartheta^\ast_{mp}:=\vartheta_m+(m-n-d)/(4pq)<0$ and $a_3:=\vartheta_{mp}^\ast (\lambda_{n+d}(\psi_\ast))^{-1}>0$.
The derivative of $W_3(\xi,\eta,t)$ along the trajectories of system \eqref{xesys} satisfies
\begin{eqnarray*}
    \frac{dW_3}{dt}\Big|_{\eqref{xesys}}
    &=&
 -2p t^{-\frac{n+d}{2q}} \Big (\vartheta_{mp}^\ast \xi^{2p} + \vartheta_m \eta^{2p} + \frac{m-n-d}{4pq} t^{\frac{m}{2q}-1}\eta^{2p}+\mathcal O(\rho^{2p+1})\Big)(1+ \mathcal O( t^{-\varsigma}))
\end{eqnarray*}
as $\rho\to 0$ and $t\to\infty$. It follows that there exists $\rho_3\leq \rho_\ast$ and $t_3\geq t_\ast$ such that
\begin{gather*}
A_3\rho^{2p}\leq W_3(\xi,\eta,t)\leq    B_3 t^{\frac{m-n-d}{2q}} (\xi^{2p}+\eta^{2p}), \quad
\frac{dW_3}{dt}\Big|_{\eqref{xesys}}\leq -p \vartheta_{mp}^\ast t^{-\frac{n+d}{2q}} (\xi^{2p}+\eta^{2p}) \leq  -p \gamma_3 t^{-\frac {m}{2q}} W_3
\end{gather*}
for all $(\xi,\eta,t)\in\mathcal J(\rho_3,t_3)$, where
$A_3=2^{1-p} \min\{1,a_3\}$,  $B_3= \max\{1,a_3\}$,
$\gamma_3=\vartheta_{mp}^\ast B_3^{-1}>0$. Hence, the equilibrium $(0,0)$ is stable. Integrating the last inequality, we get
\begin{gather*}
     \begin{split}
        \xi^2(t)+\eta^2(t)\leq K_3 \exp\Big(- \frac{2q\gamma_3}{2q-m} t^{1-\frac{m}{2q}} \Big),  \quad &  n+d<m< 2q, \\
        \xi^2(t)+\eta^2(t) \leq K_3 t^{-\gamma_3},  \quad &  n+d<m= 2q
    \end{split}
\end{gather*}
as $t\geq t_3$, where $K_3$ is a positive constant depending on $\xi(t_3)$, $\eta(t_3)$ and $t_3$.  Thus, the equilibrium $(0,0)$ of system \eqref{FulSys} is exponentially stable if $n+d<m<2q$, and polynomially stable if $n+d<m=2q$.

Let us show that the conditions $\lambda_{n+d}(\psi_\ast)>0$, $\lambda_{n,\sigma}(\psi_\ast)<0$ ensure the polynomial stability.
Consider $\ell_3(\zeta,\eta)=\zeta^{2p}+t^{(m-n-d)/{2q}}\eta^{2p}$ as a Lyapunov function candidate for system \eqref{zesys}. Here, $p$ is a positive integer such that $\vartheta_{mp}^\ast<0$. It is easily shown that there exists $\varrho_3>0$ such that the derivative satisfies
\begin{eqnarray*}
     \frac{d\ell_3}{dt}\Big|_{\eqref{zesys}}&= &2p t^{-\frac{n+d}{2q}}\Big(  \zeta^{2p-1} Q_{n+d}+    \eta^{2p-1} P_m +\frac{m-n-d}{4pq}t^{\frac{m}{2q}-1}\eta^{2p} +\zeta^{2p-1}\tilde Q_{n+d}+  t^{\frac{m}{2q}}\eta^{2p-1}\tilde P_{m}\Big)\\
&\leq & - p  t^{-\frac{n+d}{2q}} \varrho^{2p-1} (\gamma_3^\ast\varrho - 4 K t^{-\varsigma})
\end{eqnarray*}
as $\varrho\leq \varrho_3$ and $t\geq t_\ast$ with $\gamma_3^\ast=\min\{|\partial_\xi Q_{n+d}(\xi_\ast)|,|\vartheta_{mp}^\ast|\}>0$.
Hence, for all $\varepsilon\in (0,\varrho_3)$ there exist $\delta>0$ and $t_3=t_\ast+(8K/(\delta \gamma_3^\ast))^{1/\varsigma}$ such that
$d\ell_3/dt \leq -t^{-{(n+d)}/{2q}} p \varrho^{2p}\gamma_3^\ast/2< 0
$
for all $\delta\leq \varrho \leq \varepsilon$ and $t\geq t_3$. Therefore, any solution $\zeta(t)$, $\eta(t)$ starting from $\{(\zeta,\eta): \varrho\leq \delta\}$ at $\tau_3\geq t_3$ cannot leave the domain $\{(\zeta,\eta): \varrho\leq \varepsilon\}$ as $t>\tau_3$. Thus, the equilibrium $(0,0)$ of system \eqref{FulSys} is polynomially stable.
\end{proof}

\begin{Th}
Let system \eqref{FulSys} satisfy \eqref{H0}, \eqref{Has}, \eqref{Fas}, \eqref{pcond}, and  $2\leq n,m\leq 2q$, $\sigma\geq 2$, $d\geq 1$ be integers such that  assumptions \eqref{as01}, \eqref{as2} hold, $\omega(E)\not\equiv {\hbox{\rm const}}$ and $m\leq 2q<n+d$.  If $\lambda_{2q,\sigma}(\psi_\ast)<0$, then  the equilibrium $(0,0)$ is polynomially stable.
\end{Th}
\begin{proof}
We take $N=2q$ in \eqref{VF}; then \eqref{simsys} takes the following form:
\begin{gather}
\label{simsys3}
    \frac{dv}{dt}=   \sum_{k=n}^{2q}t^{-\frac{k}{2q}} \Lambda_{k}(v,\psi) + \tilde\Lambda_{2q}(v,\psi,t),   \quad
    \frac{d\psi}{dt}=  t^{-\frac{m}{2q}} \Omega_m(v,\psi) +    \tilde\Omega_{m}(v,\psi,t),
\end{gather}
where $\Lambda_{2q}(v,\psi)=l v/q+\mathcal O(v^{(\sigma+1)/2})$ as $v\to 0$. Since $\omega(E)\not\equiv {\hbox{\rm const}}$, we have $l\geq 1$. The change of variables $v(t)=t^{-\mu}[\xi(t)]^2$, $\psi(t)=\hat\psi(t)+\eta(t)$ with $\mu=(2q-n)/(q(\sigma-1))\geq 0$ transforms system \eqref{simsys3} into
\begin{gather}
\label{xesys4}
    \frac{d\xi}{dt}=t^{-1} Q_{2q}(\xi,\eta)+\tilde Q_{2q}(\xi,\eta,t), \quad
    \frac{d\eta}{dt}=t^{-\frac {m}{2q}} P_m(\eta)+\tilde P_{m}(\xi,\eta,t),
\end{gather}
where $ Q_{2q}(\xi,\eta)= ((\mu+l/q) \xi + \lambda_{n,\sigma}(\psi_\ast+\eta) \xi^\sigma)/2$, $ P_{m}(\eta)=\Omega_m(0,\psi_\ast+\eta)-\Omega_m(0,\psi_\ast)$, and the remainder functions satisfy the estimates: $|\tilde Q_{2q}(\xi,\eta,t)|\leq K t^{-1} \varrho (t^{-{1}/{2q}} +t^{-\kappa}+ t^{-\mu/2})$, $ |\tilde P_{m}(\xi,\eta,t)|\leq K t^{-m/2q} \varrho (t^{-{1}/{2q}}+t^{-\kappa} + t^{-\mu/2})$ as $t\geq t_\ast$ and $\varrho:=\sqrt{\xi^2+\eta^2}\leq \varrho_\ast$ with some constants $\varrho_\ast>0$, $t_\ast\geq t_0$, $K>0$. We see that the equation $Q_{2q}(\xi)=0$ has a positive root $\xi_\ast:=((\mu+l/q)/|\lambda_{n,\sigma}(\psi_\ast)|)^{1/(\sigma-1)}$ such that $\partial_\xi Q_{2q}(\xi_\ast,0)=- (\sigma-1)(\mu+l/q)/2<0$. Let us show that the solutions of system \eqref{xesys4} starting near the point $(\xi_\ast,0)$ remain close to it.
The change of variable $\xi(t)=\xi_\ast+\zeta(t)$ yields
\begin{gather}
        \label{zesys4}
            \frac{d\zeta}{dt}=t^{-1} Q_{2q}(\xi_\ast+\zeta,\eta)+\tilde Q_{2q}(\xi_\ast+\zeta,\eta,t),
        \quad
            \frac{d\eta}{dt}=t^{-\frac{m}{2q}} P_m(\eta)+\tilde P_{m}(\xi_\ast+\zeta,\eta,t).
\end{gather}
Suppose $p$ is an integer such that $Q_{2qp}^\ast:=\partial_\xi Q_{2q}(\xi_\ast,0)+(2q-m)/(4pq)<0$. We use $\ell(\zeta,\eta,t)=t^{1-m/2q}\zeta^{2p} +\eta^{2p}$ as a Lyapunov function candidate for system \eqref{zesys4}. It is not hard to see that there exists $\varrho_0\leq \varrho_\ast$ such that
\begin{eqnarray*}
    \frac{d\ell}{dt}\Big|_{\eqref{zesys4}}&=& 2p t^{-\frac{m}{2q}}\Big(  \zeta^{2p-1}Q_{2q}+\frac{2q-m}{4pq}\zeta^{2p}t^{\frac{m}{2q}-1}+ \eta^{2p-1}P_m +  \zeta^{2p-1}t \tilde Q_{2q}+  \eta^{2p-1}  t^{\frac{m}{2q}} \tilde P_{m}\Big)\\
&\leq & -p t^{-\frac{m}{2q}}\varrho^{2p-1} (\gamma_0^\ast \varrho- 4Kt^{-\varsigma})
\end{eqnarray*}
as $\varrho\leq \varrho_0$ and $t\geq t_\ast$, where $\gamma_0^\ast=\min\{|Q_{2qp}^\ast|,|\vartheta_m|\}>0$, $\varsigma= \min\{1/(2q),\kappa,\mu/2\}$.
Hence, for all $\varepsilon\in (0,\varrho_0)$ there exist $\delta>0$ and $t_4=t_\ast+(8K/(\delta \gamma_0^\ast))^{1/\varsigma}$ such that $d\ell/dt\leq -p t^{-m/2q}\varrho^{2p}\gamma_0^\ast/2<0$
for all $\delta\leq \varrho \leq \varepsilon$ and $t\geq t_4$.
Therefore, any solution $\zeta(t)$, $\eta(t)$ starting from $\{(\zeta,\eta): \varrho\leq \delta\}$ at $\tau_4\geq t_4$ cannot leave the domain $\{(\zeta,\eta): \varrho\leq \varepsilon\}$ as $t>\tau_4$. Thus, returning to system \eqref{FulSys}, we see that the equilibrium $(0,0)$ is polynomially stable.
\end{proof}

\section{Phase drifting}
\label{sec5}
Here, we consider system \eqref{FulSys}, when assumption \eqref{as02} holds. In this case, the stability of the equilibrium $(0,0)$ relates to the properties of a one-parametric family of solutions to system \eqref{simsys} such that $v(t)\equiv 0$ and $|\psi(t)|\to\infty$ as $t\to\infty$.

\begin{Th}\label{Th4}
Let system \eqref{FulSys} satisfy \eqref{H0}, \eqref{Has}, \eqref{Fas}, \eqref{pcond} and  $2\leq n,m\leq 2q$ be integers such that assumptions \eqref{as02}, \eqref{as1} hold.
\begin{itemize}
  \item  If $n<2q$ and $\lambda_n(\psi)<0$ for all $\psi\in\mathbb R$, then the equilibrium $(0,0)$  is exponentially stable.
  \item If $n=2q$ and $  \lambda_n(\psi)+ {l}/{q}<0$ for all $\psi\in\mathbb R$, then the equilibrium $(0,0)$ is polynomially stable
  \item If $\lambda_n(\psi)>0$, $\widetilde \lambda_n(v,\psi)\geq 0$ for all $v\in [0,\Delta_0]$ and $\psi\in\mathbb R$, then the equilibrium $(0,0)$ is unstable.
\end{itemize}
\end{Th}
\begin{proof}
We choose $N=n$ and consider the first equation in \eqref{simsys1}. It can easily be checked that
\begin{gather*}
\frac{dv}{dt}=t^{-\frac{n}{2q}} v \Big(\lambda_n(\psi)+\widetilde\lambda_n(v,\psi)+\delta_{n,2q}\frac{l}{q}+\mathcal O(t^{-\frac {1}{2q}})\Big)
\end{gather*}
as $t\to \infty$, for all $v\in [0,\Delta_0]$ and $\psi\in\mathbb R$. Let $\lambda_n(\psi)+\delta_{n,2q} {l}/{q}<0 $. Then, for all $\epsilon\in (0,1)$ there exist $0<\Delta_1\leq \Delta_0$ and $t_1\geq t_0$ such that
\begin{align}
\label{ineq1}
    \frac{dv}{dt}\leq
         \displaystyle  -t^{-\frac{n}{2q}} v(1-\epsilon)\mu_{n}^\ast  < 0
\end{align}
for all $v\in (0,\Delta_1]$, $t\geq t_1$ and $\psi\in\mathbb R$  with \begin{gather*}
\mu_{n}^\ast:=\min_{\psi\in\mathbb R}\Big|\lambda_n(\psi)+\delta_{n,2q}\frac{l}{q}\Big|>0.
\end{gather*}
Integrating \eqref{ineq1} over $[t_1,t]$, we obtain
\begin{align*}
0\leq v(t)\leq C_1 \exp \Big(-  (1-\epsilon) \frac{2q\mu_{n}^\ast }{2q-n} t^{1-\frac{n}{2q}}\Big), \quad\ \ n< 2q,\\
0\leq v(t)\leq C_1 t^{-(1-\epsilon)\mu_{n}^\ast}, \quad n=2q,
\end{align*}
where $C_1>0$ is a positive constant depending on $v(t_1)$ and $t_1$.
Hence, the solution $v(t)\equiv 0$ is stable for all $\psi$.
Moreover, the stability is exponential if $n < 2q$ and polynomial if $n = 2q$.

Let $\lambda_n(\psi)>0 $ and $\widetilde \lambda_n(v,\psi)\geq 0$ for all $v\in [0,\Delta_0]$,  $\psi\in\mathbb R$. Define $\mu_n:=\min_{\psi\in\mathbb R}\lambda_n(\psi)>0$. Then, for all $\epsilon\in (0,1)$ there exists $t_2\geq t_0$ such that
\begin{align}
\label{ineq2}
    \frac{dv}{dt}\geq   \displaystyle t^{-\frac{n}{2q}} v \Big( (1-\epsilon)\mu_{n}+\delta_{n,2q}\frac{l}{q}\Big)  > 0
\end{align}
for all $v\in (0,\Delta_0]$, $t\geq t_2$ and $\psi\in\mathbb R$.
Integrating \eqref{ineq2}, we get the following estimates as $t\geq t_2$:
\begin{align*}
  v(t)\geq  v(t_2) \exp \left((1-\epsilon)\mu_{n} \Big(t^{1-\frac{n}{2q}}-t_2^{1-\frac{n}{2q}}\Big)\frac{ 2q }{2q-n} \right), \quad n< 2q,\\
  v(t)\geq v(t_2) \Big(\frac{t}{t_2}\Big)^{(1-\epsilon)\mu_{n}+\frac{l}{q}}, \quad  n=2q.
\end{align*}
Hence, for all $\delta\in (0,\Delta_0)$ the solution $v(t)$ with initial data $v(t_2)=\delta$ hits $\Delta_0$ at some $t_2^\ast>t_2$.
Taking into account \eqref{vest1}, \eqref{estVE} and \eqref{epsdel}, we obtain instability of the equilibrium $(0,0)$ in system \eqref{FulSys}.
\end{proof}

\begin{Rem}
Note that stability of the equilibrium $(0,0)$ of system \eqref{FulSys} is not justified, when $n=2q$, $0<\lambda_n(\psi)+l/q<l/q$ for all $\psi\in\mathbb R$, and $\omega(E)\not\equiv {\hbox{\rm const}}$ $(l\neq 0)$. Arguing as in Theorem~\ref{Th4}, we see that the trivial solution $v(t)\equiv 0$ is unstable in system \eqref{simsys1}. In this case, we can say that the equilibrium $(0,0)$ is unstable with the weight $t^{l/2q}$ in the original variables $(x,y)$.
\end{Rem}

Let us consider the case when $\lambda_n(\psi)$ is sign-changing.
Define
\begin{gather*}
          \gamma_{n,m}(\psi) \equiv \frac{\lambda_n(\psi)}{|\omega_m(\psi)|}, \quad
            \widetilde \gamma_{n,m}(\psi) \equiv   \gamma_{n,m}(\psi)-\widehat \gamma_{n,m}, \quad
            \widehat\gamma_{n,m} := \langle \gamma_{n,m}(\psi) \rangle_\psi, \quad
            Z_{n,m}(\psi)\equiv \int\limits_0^\psi \widetilde \gamma_{n,m}(s)\, ds,\\
           \chi_{m}(\psi) \equiv \frac{1}{|\omega_m(\psi)|}, \quad
            \widetilde\chi_{m}(\psi) \equiv \chi_{m}(\psi) - \widehat\chi_{m},\quad
            \widehat\chi_{m} := \langle \chi_{m}(\psi) \rangle_\psi, \quad
            X_m(\psi)\equiv \int\limits_0^\psi \widetilde \chi_{m}(s)\, ds, \\
  Z_{n,m}^+:=\max_{\psi\in[0,2\pi)} | Z_{n,m}(\psi)|,
                \quad X_{m}^+:=\max_{\psi\in[0,2\pi)} | X_{m}(\psi)|,
                        \quad \omega^-_m:=\min _{\psi\in[0,2\pi)} |\omega_m(\psi)|>0,\\
  \Gamma_{11}:=\{(n,m)\in\mathbb N^2: 2\leq n<2q, 2\leq m<2q\} \cup \{(2q,2q)\}, \\
  \Gamma_{01}:=\{(n,m)\in\mathbb N^2:  2\leq m<n=2q\},
        \quad \Gamma_{10}:=\{(n,m)\in\mathbb N^2: 2\leq n<m=2q\}.
\end{gather*}
We have the following:
\begin{Th}\label{Th5}
Let system \eqref{FulSys} satisfy \eqref{H0}, \eqref{Has}, \eqref{Fas}, \eqref{pcond}, and  $2\leq n,m\leq 2q$ be integers such that  assumptions \eqref{as02}, \eqref{as1} hold.
\begin{itemize}
  \item  If either $\widehat\gamma_{n,m}<0$, $(n,m)\in\Gamma_{11}$ or $\displaystyle\widehat\gamma_{n,m}+Z_{n,m}^+(m-n)/ (q \omega_m^-)<0$, $(m,n)\in\Gamma_{10}$,  then  the equilibrium $(0,0)$ is  exponentially stable
  \item If $\widehat\gamma_{n,m}+\widehat\chi_m  {l}/{q}<0$ and $(n,m)\in\Gamma_{01}$,  then  the equilibrium $(0,0)$ is polynomially stable.
  \item If $\widehat \gamma_{n,m}>0$  and $\widetilde \lambda_n(v,\psi)\geq 0$ for all $v\in [0,\Delta_0]$ and $\psi\in\mathbb R$, then the equilibrium $(0,0)$ is unstable.
\end{itemize}
\end{Th}
\begin{proof}
We take $N=n$ in system \eqref{simsys}. Consider first the case when $m<n\leq 2q$. Let $\widehat \gamma_{n,m}>0$  and $\widetilde \lambda_n(v,\psi)\geq 0$ for all $v\in [0,\Delta_0]$ and $\psi\in\mathbb R$. We use
\begin{gather*}
    W_1(v,\psi,t)=v-t^{-\frac{n-m}{2q}} v \Big( Z^+_{n,m}+{\hbox{\rm sgn}}(\omega_m(\psi))Z_{n,m}(\psi)\Big)
\end{gather*}
as a Chetaev function candidate for system \eqref{simsys1}. It can easily be checked that
\begin{gather*}
\frac{dW_1}{dt}\Big|_{\eqref{simsys1}}=t^{-\frac{n}{2q}}
v\Big(\widehat\gamma_{n,m}|\omega_m(\psi)|+\widetilde \lambda_n(v,\psi)+\delta_{n,2q}\frac lq  +\mathcal O(t^{-\frac{1}{2q}})\Big)
\end{gather*}
as $t\to\infty$ for all $v\in[0,\Delta_0]$, $\psi\in\mathbb R$. Hence, for all $\epsilon\in (0,1)$ there exists $t_1\geq t_0$ such that
$v/2\leq W_1(v,\psi,t)\leq v$ and
\begin{eqnarray}
\label{wineq2}
\frac{dW_1}{dt}\Big|_{\eqref{simsys1}}&\geq &  t^{-\frac{n}{2q}}\Big( (1-\epsilon) \widehat\gamma_{n,m} \omega_m^-+\delta_{n,2q}\frac lq\Big)W_1 >0
\end{eqnarray}
 for all $v\in(0,\Delta_0]$, $t\geq t_1$  and $\psi\in\mathbb R$.  Integrating \eqref{wineq2} over $[t_1,t]$, we obtain
\begin{align*}
 v(t) \geq  W_1(v(t),\psi(t),t)
        \geq
             v(t_1)\exp \Big((1-\epsilon) \widehat\gamma_{n,m} \omega_m^- \frac{ 2q }{2q-n} t^{1-\frac{n}{2q}}\Big), \quad  m<n< 2q,\\
 v(t) \geq W_1(v(t),\psi(t),t)
        \geq
             v(t_1)  \Big(\frac{t}{t_1}\Big)^{(1-\epsilon)\widehat\gamma_{n,m}\omega_m^-+l/q}, \quad m<n=2q.
\end{align*}
 These inequalities justify the instability of the solution $v(t)\equiv 0$ for all $\psi\in\mathbb R$.

Let $\widehat\gamma_{n,m}+\delta_{n,2q} \widehat\chi_m l/q<0$.
Consider
\begin{gather*}
    W_2(v,\psi,t)=v-t^{-\frac{n-m}{2q}} v \Big( Z^+_{n,m}+X^+_m+{\hbox{\rm sgn}}(\omega_m(\psi)) \big(Z_{n,m}(\psi)+X_m(\psi)\big)\Big)
\end{gather*}
as a Lyapunov function candidate for system \eqref{simsys1}.
The total derivative of $W_2(v,\psi,t)$ with  respect to $t$ satisfies:
\begin{gather*}
\frac{dW_2}{dt}\Big|_{\eqref{simsys1}}=t^{-\frac{n}{2q}}|\omega_m(\psi)|
v\Big(\widehat\gamma_{n,m}+\delta_{n,2q}\widehat\chi_m\frac lq +\mathcal O(v^{\frac{1}{2}})+\mathcal O(t^{-\frac{1}{2q}})\Big)
\end{gather*}
as $v\to 0$, $t\to\infty$ for all $\psi\in\mathbb R$. Hence, for all $\epsilon\in (0,1)$ there exist $0<\Delta_2\leq \Delta_0$ and $t_2\geq t_0$ such that
$v/2\leq W_2(v,\psi,t)\leq v$ and
\begin{eqnarray}
\label{wineq1}
\frac{dW_2}{dt}\Big|_{\eqref{simsys1}}&\leq &  t^{-\frac{n}{2q}} (1-\epsilon)  \omega_m^-  \Big(\widehat\gamma_{n,m}+\delta_{n,2q}\widehat\chi_m \frac lq\Big)W_2 <0
\end{eqnarray}
 for all $v\in(0,\Delta_2]$, $t\geq t_2$ and $\psi\in\mathbb R$.
Integrating \eqref{wineq1} over $[t_2,t]$ yields
\begin{align*}
 v(t) \leq 2 W_2(v(t),\psi(t),t)  \leq C_2 \exp \Big(-(1-\epsilon) \omega_m^-|\widehat\gamma_{n,m}| \frac{ 2q }{2q-n}t^{1-\frac{n}{2q}}\Big), \quad m<  n<2q,\\
 v(t) \leq 2 W_2(v(t),\psi(t),t)  \leq C_2 \exp \Big(-(1-\epsilon)\omega_m^-\Big|\widehat\gamma_{n,m}+\widehat\chi_m\frac{l}{q}\Big| \log t\Big), \quad m<n=2q,
\end{align*}
where  $C_2>0$ is a constant depending on $v(t_2)$ and $t_2$. Hence, the solution $v(t)\equiv 0$ is exponentially stable for all $\psi\in\mathbb R$ if $m<n < 2q$ and $\widehat\gamma_{n,m}<0$. If $m<n = 2q$ and $\widehat\gamma_{n,m} + \widehat\chi_m l/q<0$, the solution is polynomially stable.

Now, let $n\leq m\leq 2q$. In this case, we use
$
W_3(v,\psi,t)=v\exp ( t^{ {(m-n)}/{2q}}  (Z_{n,m}^+-Z_{n,m}(\psi) ) )
$
as a Lyapunov function candidate.  Since $0\leq Z^+_{n,m}-Z_{n,m}(\psi)\leq 2Z^+_{n,m}$, we have $W_3(v,\psi,t)\geq v$ for all $v>0$,  $t\geq t_0$ and $\psi\in\mathbb R$.
The total derivative of $W_3(v,\psi,t)$ with respect to $t$ is given by
\begin{gather*}
\frac{dW_3}{dt}\Big|_{\eqref{simsys1}}=t^{-\frac{n}{2q}} W_3
 \Big(\widehat\gamma_{n,m}|\omega_m(\psi)|+\frac{m-n}{2q} \big(Z_{n,m}^+-Z_{n,m}(\psi)\big) t^{-1+\frac{m}{2q}} +\widetilde \lambda_n(v,\psi)+\mathcal O(t^{-\frac{1}{2q}})\Big)
\end{gather*}
as $v\to 0$, $t\to\infty$ for all $\psi\in\mathbb R$.
If $\widehat\gamma_{n,m}^\ast:=\widehat\gamma_{n,m}  \omega_m^-+\delta_{m,2q}Z_{m,n}^+(m-n)/q<0$, then for all $\epsilon\in (0,1)$ there exist $0<\Delta_3\leq \Delta_0$ and $t_3\geq t_0$ such that
\begin{eqnarray*}
\frac{dW_3}{dt}\Big|_{\eqref{simsys1}}&\leq & - t^{-\frac{n}{2q}}  (1-\epsilon) |\widehat\gamma_{n,m}^\ast|W_3 <0
\end{eqnarray*}
for all $v\in(0,\Delta_3]$, $t\geq t_3$ and $\psi\in\mathbb R$.
Arguing as above, we see that the solution $v(t)\equiv 0$ is exponentially stable if $\widehat\gamma_{n,m}<0$ and $n\leq m<2q$ or $n=m=2q$. If $n<m=2q$, the condition  $\widehat\gamma_{n,m}\omega_m^-+\delta_{m,2q}Z_{n,m}^+(m-n)/q<0$ guarantees the exponential stability. Similarly, if $\widehat\gamma_{n,m}>0$ and $\widetilde \lambda_n(v,\psi)\geq 0$ for all $v\in [0,\Delta_0]$ and $\psi\in\mathbb R$, we have
\begin{eqnarray*}
\frac{dW_3}{dt}\Big|_{\eqref{simsys1}}&\geq &  t^{-\frac{n}{2q}}   (1-\epsilon) \widehat\gamma_{n,m}\omega_m^- W_3 >0
\end{eqnarray*}
for all $v\in (0, \Delta_0]$, $t\geq t_3$ and $\psi\in\mathbb R$. In this case, the solution $v(t)\equiv 0$ is unstable for all $\psi\in\mathbb R$.

Taking into account  \eqref{vest1}, \eqref{estVE} and \eqref{epsdel}, we obtain the corresponding propositions on the stability of the equilibrium $(0,0)$ in system \eqref{FulSys}.
\end{proof}

\begin{Rem}\label{Rem4}
It follows from the proof of Theorem~\ref{Th5} that if $(n,m)\in\Gamma_n$, $\widehat \gamma_{n,m} + |\omega_m(\psi)|^{-1}l/q>0$ for all $\psi\in\mathbb R$ and $\omega(E)\not\equiv {\hbox{\rm const}}$ $(l\neq 0)$, then the equilibrium $(0,0)$ is unstable with the weight $t^{l/2q}$.
\end{Rem}

Now, consider the case when assumption \eqref{as2} holds. Recall that $\nu=d/(q (\sigma-1))>0$.

\begin{Th}\label{Th6}
Let system \eqref{FulSys} satisfy \eqref{H0}, \eqref{Has}, \eqref{Fas}, \eqref{pcond}, and  $2\leq n,m\leq 2q$, $\sigma\geq 2$, $d\geq 1$ be integers such that  assumptions \eqref{as02}, \eqref{as2} hold.
\begin{itemize}
  \item If $n+d<2q$, $\lambda_{n+d}(\psi)<0$ for all $\psi\in\mathbb R$, then the equilibrium $(0,0)$ is exponentially stable.
  \item  If $n+d=2q$ and either $\lambda_{n+d}(\psi)+\nu+l/q<0$  or $\lambda_{n+d}(\psi)+ l/q<0$, $\lambda_{n,\sigma}(\psi)<0$ for all $\psi\in\mathbb R$, then the equilibrium $(0,0)$ is polynomially stable.
  \item If $n+d\leq 2q$, $\lambda_{n+d}(\psi)>0$, $\lambda_{n,\sigma}(\psi)>0$, $\widetilde \lambda_{n+d}(v,\psi)\geq 0$, $\widetilde \lambda_{n,\sigma}(v,\psi)\geq 0$ for all $v\in[0,\Delta_0]$, $\psi\in\mathbb R$, then the equilibrium $(0,0)$ is unstable.
\end{itemize}
 \end{Th}
\begin{proof}
We take $N=n+d$  in system \eqref{simsys2}. By substituting $v(t)=t^{-\nu} [R(t)]^2$, we obtain
\begin{gather}
\label{rpsys}
    \frac{dR}{dt}=t^{-\frac{n+d}{2q}} A_{n+d}(R,\psi)+\tilde A_{n+d}(R,\psi,t), \quad \frac{d\psi}{dt}=t^{-\frac{m}{2q}} \omega_m(\psi)+\tilde B_{m}(R,\psi,t),
\end{gather}
where
$A_{n+d}(R,\psi)= R (\lambda_{n+d}(\psi)+\delta_{n+d,2q}  (\nu+ {l}/{q} )  + \lambda_{n,\sigma}(\psi) R^{\sigma-1} )/2, $
and
\begin{gather*}
    |\tilde A_{n+d}(R,\psi,t)|\leq K t^{-\frac{n+d}{2q}} R \big(t^{-\frac{1}{2q}} + t^{-\frac{\nu}{2}}\big),
        \quad
    |\tilde B_{m}(R,\psi,t)|\leq K t^{-\frac{m}{2q}}   R  \big(t^{-\frac{1}{2q}} + t^{-\frac{\nu}{2}}\big)
\end{gather*}
for all $ R \in[0, R_\ast]$, $t\geq t_\ast$, $\psi\in\mathbb R$ with some constants $R_\ast>0$, $t_\ast\geq t_0$, and $K>0$.
Note that the right-hand sides of system \eqref{rpsys} have the same form as that of \eqref{simsys1} with $n$ and $l/q$ replaced by $n+d$ and $\nu+l/q$, correspondingly. Therefore, by repeating the proof of Theorem~\ref{Th4}, it can be shown that the equilibrium $(0,0)$ of system \eqref{FulSys} is exponentially stable if $n+d<2q$, $\lambda_{n+d}(\psi)<0$  for all $\psi\in\mathbb R$ and polynomially stable if $n+d=2q$, $\lambda_{n+d}(\psi)+\nu+l/q<0$ for all $\psi\in\mathbb R$.

Let $\lambda_{n+d}(\psi)+\delta_{n+d,2q}l/q<0$ and $\lambda_{n,\sigma}(\psi)<0$ for all $\psi\in\mathbb R$. It can easily be checked that there exist $0<\Delta_1\leq \Delta_0$ and $t_1\geq t_0$ such that for all $v\in (0,\Delta_1]$, $t\geq t_1$ and $\psi\in\mathbb R$ the following inequalities hold:
\begin{align}
\label{ineqnd1}
    \frac{dv}{dt}\leq
         \displaystyle  -t^{-\frac{n+d}{2q}} \frac{ v}{2} \Big(\mu_{n,\sigma} v^{\frac{\sigma-1}{2}} t^{\frac{d}{2q}}+\mu_{n+d}^\ast  \Big)\leq  -t^{-\frac{n+d}{2q}}  \mu_{n+d}^\ast \frac{v}{2} < 0,
\end{align}
with $\mu_{n,\sigma}:=\min_{\psi\in\mathbb R}|\lambda_{n,\sigma}(\psi)|>0$ and $ \mu^\ast_{n+d}:=\min_{\psi\in\mathbb R}|\lambda_{n+d}(\psi)+\delta_{n+d,2q}{l}/{q}|>0$.
Integrating \eqref{ineqnd1} over $[t_1,t]$, we get
\begin{align*}
0\leq v(t)\leq C_1 \exp \Big(-  \frac{ q \mu_{n+d}^\ast }{2q-n-d} t^{1-\frac{n+d}{2q}}\Big), \quad n+d< 2q,\\
0\leq v(t)\leq C_1 t^{-\frac{\mu_{n+d}^\ast}{2}}, \quad n+d=2q,
\end{align*}
where $C_1>0$ is a positive constant depending on $v(t_1)$ and $t_1$.
Hence, the solution $v(t)\equiv 0$ of system \eqref{simsys2} is stable for all $\psi\in\mathbb R$. Moreover, the stability is exponential if $n +d< 2q$ and polynomial if $n+d = 2q$.

Let $\lambda_{n+d}(\psi)>0$,  $\lambda_{n,\sigma}(\psi)>0$, $\widetilde \lambda_{n+d}(v,\psi)\geq 0$, $\widetilde \lambda_{n,\sigma}(v,\psi)\geq 0$ for all $v\in[0,\Delta_0]$, $\psi\in\mathbb R$. Then there exists $t_2\geq t_0$ such that
\begin{align*}
    \frac{dv}{dt}\geq   \displaystyle t^{-\frac{n+d}{2q}} \frac{v}{2}\Big( \mu_{n,\sigma} v^{\frac{\sigma-1}{2}} t^{\frac{d}{2q}}+   \mu_{n+d}+2\delta_{n+d,2q}\frac{l}{q}\Big)\geq  t^{-\frac{n+d}{2q}}\Big( \frac{\mu_{n+d}}{2}+\delta_{n+d,2q}\frac{l}{q}\Big) v>0
\end{align*}
for all $v\in (0,\Delta_0]$, $t\geq t_2$ and $\psi\in\mathbb R$, with
$\mu_{n+d}:=\min_{\psi\in\mathbb R} \lambda_{n+d}(\psi) $.
Integrating the last inequality, we get the following estimates as $t\geq t_2$:
\begin{align*}
  v(t)\geq  v(t_2) \exp \left(\mu_{n+
d} \Big(t^{1-\frac{n+d}{2q}}-t_2^{1-\frac{n+d}{2q}}\Big)\frac{ q }{2q-n-d} \right), \quad n+d< 2q,\\
  v(t)\geq v(t_2) \Big(\frac{t}{t_2}\Big)^{\frac{\mu_{n+d}}{2}+\frac{l}{q}}, \quad  n+d=2q.
\end{align*}
Hence, for all $\delta\in (0,\Delta_0)$ the solution $v(t)$ with initial data $v(t_2)=\delta$ exceeds the value $\Delta_0$ at some $t_2^\ast>t_2$.
Taking into account \eqref{vest1}, \eqref{estVE} and \eqref{epsdel}, we obtain instability of the equilibrium $(0,0)$ in system \eqref{FulSys}.
\end{proof}

Similarly, by repeating the proof of Theorem~\ref{Th5} for system \eqref{rpsys}, we obtain the following:
\begin{Th}\label{Th7}
Let system \eqref{FulSys} satisfy \eqref{H0}, \eqref{Has}, \eqref{Fas}, \eqref{pcond}, and  $2\leq n,m\leq 2q$, $\sigma \geq 2$, $d\geq 1$ be integers such that $n+d\leq 2q$ and assumptions \eqref{as02}, \eqref{as2} hold.
\begin{itemize}
  \item If either $\widehat\gamma_{n+d,m}<0$,  $(n+d,m)\in\Gamma_{11}$ or $\displaystyle\widehat\gamma_{n+d,m}+Z_{n+d,m}^+ (m-n-d)/(q \omega_m^-)<0$, $(n+d,m)\in\Gamma_{10}$, then the equilibrium $(0,0)$ is exponentially stable.
  \item If $\displaystyle\widehat\gamma_{n+d,m}+\widehat\chi_m (\nu+l/q)<0$ and $(n+d,m)\in\Gamma_{01}$, then  the equilibrium $(0,0)$ is polynomially stable.
\end{itemize}
\end{Th}

Note that Theorem~\ref{Th6} and~\ref{Th7} do not give any information about the stability if $\widehat\gamma_{n+d,m}>0$, $\widehat\gamma_{n,\sigma,m}:=\langle\gamma_{n,\sigma}(\psi)|\omega_m(\psi)|^{-1}\rangle_\psi<0$ or $\lambda_{n+d}(\psi)>0$,  $\lambda_{n,\sigma}(\psi)<0$.  It turns out that in these cases polynomial stability can take place.

\begin{Th}\label{Th8}
Let system \eqref{FulSys} satisfy \eqref{H0}, \eqref{Has}, \eqref{Fas}, \eqref{pcond}, and  $2\leq n,m\leq 2q$, $\sigma \geq 2$, $d\geq 1$ be integers such that assumptions \eqref{as02}, \eqref{as2} hold.
If $m<n+d\leq 2q$, $\widehat\gamma_{n+d,m}>0$ and $\widehat\gamma_{n,\sigma,m}<0$, then the equilibrium $(0,0)$ is polynomially stable.
\end{Th}
\begin{proof} Consider system \eqref{rpsys}. Define
$\widehat A_{n+d,m}(R)\equiv \langle   A_{n+d}(R,\psi)|\omega_m(\psi)|^{-1}\rangle_\psi$ and $ \widetilde A_{n+d,m}(R,\psi) \equiv  A_{n+d}(R,\psi) |\omega_{m}(\psi)|^{-1} - \widehat A_{n+d,m}(R)$ for all $R\in [0,R_\ast]$, $\psi\in\mathbb R$.
Note that the equation $\widehat A_{n+d,m}(R)=0$ has a positive root $R=R_\ast$, where
\begin{gather}
\label{Rast}
R_\ast = \left(
                \frac{
                    \widehat\gamma_{n+d,m}+\delta_{n+d,2q}\big(\nu+{l}/{q}\big)
                }{
                    |\widehat\gamma_{n,\sigma,m}|
                    }\right)^{\frac{1}{\sigma-1}}
\end{gather}
such that $a_\ast:=\partial \widehat {A}_{n+d,m}(R_\ast)=-(\sigma-1) \big(\widehat\gamma_{n+d,m}+\delta_{n+d,2q} (\nu+{l}/{q})\big)<0$.
The change of the variable
\begin{gather*}
R(t)=R_\ast+\zeta(t)+{\hbox{\rm sgn}}\big(\omega_m\big(\psi(t)\big)\big) t^{-\frac{n+d-m}{2q}}\int\limits_0^{\psi(t)} \widetilde A_{n+d,m}(R_\ast,s)\, ds
\end{gather*}
transforms \eqref{rpsys} into
\begin{gather}
        \label{zpsys}
            \frac{d\zeta}{dt}=t^{-\frac{n+d}{2q}} \mathcal A_{n+d}(\zeta,\psi)+\tilde {\mathcal A}_{n+d}(\zeta,\psi,t),
        \quad
            \frac{d\psi}{dt}=t^{-\frac{m}{2q}} \omega_m(\psi)+\tilde {\mathcal B}_{m}(\zeta,\psi,t),
\end{gather}
where $\mathcal A_{n+d}(\zeta,\psi) =
        |\omega_m(\psi)|\big (a_\ast+\partial_R \widetilde A_{n+d,m}(R_\ast,\psi) \big) \zeta + \mathcal O(\zeta^2)$ as $\zeta\to 0$ and
\begin{gather*}
|\tilde {\mathcal A}_{n+d}(\zeta,\psi,t) |\leq K_1 t^{-\frac{n+d}{2q}}\Big(t^{-\frac{1}{2q}}+t^{-\frac{\nu}{2}}\Big), \quad |\tilde {\mathcal B}_{m}(\zeta,\psi,t) |\leq K_1 t^{-\frac{m}{2q}}\Big(t^{-\frac{1}{2q}}+t^{-\frac{\nu}{2}}\Big)
\end{gather*}
for all $|\zeta|\leq \zeta_1$, $t\geq t_1$, $\psi\in\mathbb R$ with some constants $\zeta_1>0$, $t_1\geq t_\ast$ and $K_1>0$.
It can easily be checked that the reduced equation $d\zeta/dt = t^{-(n+d)/2q}( |\omega_m(\psi)| a_\ast\zeta+\mathcal O(\zeta^2))$ has a stable trivial solution $\zeta(t)\equiv 0$ for all $\psi\in\mathbb R$.  Let us show that this solution is stable with respect to the persistent perturbations $t^{-(n+d)/2q} |\omega_m(\psi)|\partial_R\widetilde A_{n+d,m}(R_\ast,\psi)\zeta$ and  $\tilde A_{n+d}(R,\zeta,t)$.
Consider $ \ell_1(\zeta,\psi,t)=\zeta^2- t^{-{(n+d-m)}/{2q}} 2 \zeta^2 ( \Xi(\psi)+\Xi^+)$, where
\begin{gather*}
   \Xi(\psi)\equiv {\hbox{\rm sgn}}(\omega_m(\psi))\int\limits_0^\psi \partial_R\widetilde A_{n+d,m}(R_\ast,s)\, ds, \quad
     \Xi^+:=\max_{\psi\in\mathbb R}|\Xi(\psi)|,
\end{gather*}
as a Lyapunov function candidate for the system \eqref{zpsys}. Then there exist $\zeta_2\leq \zeta_1$ and $t_2\geq t_1$ such that for all $\varepsilon\in (0,\zeta_2]$ we have
\begin{gather*}
\frac{\zeta^2}{2}\leq\ell_1(\zeta,\psi,t)\leq \zeta^2, \quad
\frac{d\ell_1}{dt}\Big|_{\eqref{zpsys}}\leq t^{-\frac{n+d}{2q}} \zeta^2  \Big(-|a_\ast|+ 4 K_2 \varepsilon^{-1} t^{-\varsigma}\Big)\leq - t^{-\frac{n+d}{2q}} \frac{|a_\ast| }{2}\ell_1
\end{gather*}
for all $\varepsilon/4<|\zeta|<\varepsilon$, $t\geq \tau_2$ and $\psi\in\mathbb R$, where $\tau_2=t_2+(8K_2\varepsilon^{-1}|a_\ast|^{-1})^{1/\varsigma}$, $\varsigma=\min\{\nu,q^{-1}\}/2$, $K_2={\hbox{\rm const}}>0$. This implies that any solution $\zeta(t)$, $\psi(t)$ of system \eqref{zpsys} with initial data $|\zeta(\tau_2)|\leq
\varepsilon/4$, $\psi(\tau_2)\in\mathbb R$ cannot leave the domain $\{|\zeta|<\varepsilon\}$ for all $t>\tau_2$.
Returning to the variable $v$, we see that the solution $v(t)\equiv 0$ is polynomially stable: $v(t)=\mathcal O(t^{-\nu})$ as $t\to \infty$.
\end{proof}

\begin{Th}
Let system \eqref{FulSys} satisfy \eqref{H0}, \eqref{Has}, \eqref{Fas}, \eqref{pcond}, and  $2\leq n,m\leq 2q$, $\sigma \geq 2$, $d\geq 1$ be integers such that assumptions \eqref{as02}, \eqref{as2} hold.
If $n+d<m$, $\lambda_{n+d}(\psi)>0$ and $\lambda_{n,\sigma}(\psi)<0$ for all $\psi\in\mathbb R$, then  the equilibrium $(0,0)$ is polynomially stable.
\end{Th}
\begin{proof}
Consider system \eqref{rpsys}. It can easily be checked that the equation $A_{n+d}(R,\psi)\equiv 0$ has a nontrivial solution $R_\ast(\psi)\equiv (\lambda_{n+d}(\psi)|\lambda_{n,\sigma}(\psi)|^{-1})^{1/(\sigma-1)}$ such that $A_\ast(\psi)\equiv \partial_R A_{n+d}(R_\ast(\psi),\psi)=-(\sigma-1)\lambda_{n+d}(\psi)/2<0$ for all $\psi\in\mathbb R$. The change of the variable $R(t)=R_\ast(\psi(t))+\zeta(t)$ transforms \eqref{rpsys} into
\begin{gather}
        \label{zpsys2}
\begin{split}
           & \frac{d\zeta}{dt}-t^{-\frac{n+d}{2q}} A_{n+d}(R_\ast(\psi)+\zeta,\psi)=\tilde A_{n+d}(R_\ast(\psi)+\zeta,\psi,t)-\partial R_\ast(\psi)\frac{d\psi}{dt},
        \\
         &   \frac{d\psi}{dt}=t^{-\frac{m}{2q}} \omega_m(\psi)+\tilde B_{m}(R_\ast(\psi)+\zeta,\psi,t).
\end{split}
\end{gather}
Note that the first equation of \eqref{zpsys2} with  the right-hand side replaced by zero has a stable solution $\zeta(t)\equiv 0$ for all $\psi\in\mathbb R$. Let us show that this solution is stable in the perturbed system.
Consider $\ell(\zeta )=\zeta^2/2$ as a Lyapunov function candidate for system  \eqref{zpsys2}. Then there exists $\zeta_1>0$ and $t_1\geq t_\ast$ such that for all $\varepsilon\in (0,\zeta_1]$ we have
$ {d\ell}/{dt}\leq t^{- {(n+d)}/{2q}}  \zeta^2   (-A_\ast^- +  4K_1  \varepsilon^{-1} t^{-\varsigma} )/2 \leq  -t^{- {(n+d)}/{2q}} {A_\ast^-} \ell /2<0$
for all $\varepsilon/4<|\zeta|<\varepsilon$, $\psi\in\mathbb R$ and $t\geq \tau_1$, where $\tau_1=t_1+(8K_1 \varepsilon^{-1}/ A_\ast^-)^{1/\varsigma}$, $A_\ast^-=\min_{\psi\in\mathbb R}|A_\ast(\psi)|$,  $\varsigma=\min\{\nu,q^{-1}\}/2$, $K_1={\hbox{\rm const}}>0$.
Hence, any solution $\zeta(t)$, $\psi(t)$ with initial data $|\zeta(\tau_1)|\leq \varepsilon/4$, $\psi(\tau_1)\in\mathbb R$ cannot leave the domain $\{|\zeta|< \varepsilon\}$ as $t>\tau_1$. Returning to the variables $(v,\psi)$, we see that $v(t)=\mathcal O(t^{-\nu})$ as $t\to \infty$. Thus, the fixed point $(0,0)$ of system \eqref{FulSys} is polynomially stable.
\end{proof}

\begin{Th}
Let system \eqref{FulSys} satisfy \eqref{H0}, \eqref{Has}, \eqref{Fas}, \eqref{pcond}, and  $2\leq n,m\leq 2q$, $\sigma \geq 2$, $d\geq 1$ be integers such that assumptions \eqref{as02}, \eqref{as2} hold.
If  $\omega(E)\not\equiv {\hbox{\rm const}}$, $n+d>2q$, $\widehat\gamma_{n,\sigma,m}<0$, then the equilibrium $(0,0)$ is polynomially stable.
\end{Th}
\begin{proof}
  We choose $N=2q$ and system \eqref{simsys} takes the form \eqref{simsys3} with
$\Lambda_{2q}(v,\psi)=l v/q+\mathcal O(v^{(\sigma+1)/2})$ as $v\to 0$. Since $\omega(E)\not\equiv {\hbox{\rm const}}$, we have $l\geq 1$. The change of the variable $v(t)=t^{- \mu}[R(t)]^2$ with $\mu=(2q-n)/(q(\sigma-1))\geq 0$ transforms system \eqref{simsys3} into
\begin{gather}
\label{xpsys4}
    \frac{dR}{dt}=t^{-1} A_{2q}(R,\psi)+\tilde A_{2q}(R,\psi,t), \quad
    \frac{d\psi}{dt}=t^{-\frac {m}{2q}} \omega_m(\psi)+\tilde B_{m}(R,\psi,t),
\end{gather}
where $ A_{2q}(R,\psi)\equiv ((\mu+l/q) R + \lambda_{n,\sigma}(\psi) R^\sigma)/2$,
$
    |\tilde A_{2q}(R,\psi,t)|\leq t^{-1}K R  (t^{- {1}/{2q}} + t^{- {\mu}/{2}} )
$, $
    |\tilde B_{m}(R,\psi,t)|\leq t^{- {m}/{2q}} K R   (t^{- {1}/{2q}} + t^{- {\mu}/{2}} )
$
for all $R \in [0,  R_\ast]$, $t\geq t_\ast$, $\psi\in\mathbb R$ with some constants $R_\ast>0$, $t_\ast\geq t_0$ and $K>0$.
It can easily be checked that the equation $ \widehat {A}_{2q,m}(R):=\langle A_{2q}(R,\psi)/|\omega_m(\psi)|\rangle_\psi=0$ has a positive root $R_\ast: =((\mu+{l}/{q}) |\langle\gamma_{n,\sigma}/|\omega_m|\rangle_\psi|^{-1})^{1/(\sigma-1)}$
such that $a_\ast:=\partial \widehat {A}_{2q,m}(R_\ast)=-(\sigma-1) (\nu+{l}/{q})<0$.
Hence, arguing as in the proof of Theorem~\ref{Th8}, we obtain the following: for all $\varepsilon>0$ there exist $\delta_1>0$ and $t_1\geq t_\ast$ such that the solutions of \eqref{xpsys4} with initial data $|R(t_1)-R_\ast|\leq \delta_1$, $\psi(t_1)\in\mathbb R$ satisfy $|R(t)-R_\ast|<\varepsilon$ as $t\geq t_1$. Thus, $v(t)=\mathcal O(t^{-\mu})$ as $t\to \infty$. Returning to the original variables $(x,y)$ completes the proof of the theorem.
\end{proof}

\section{Examples}
\label{secEX}
Consider the limiting system \eqref{LimSys} with $H_0(x,y)\equiv (x^2+y^2)/2-h x^4/4 $ and $h\geq 0$. In this case the equilibrium $(0,0)$ is a center and the level lines $H_0(x,y)\equiv E$ with $E\in (0,(4h)^{-1})$, lying in the neighbourhood of the equilibrium, correspond to $T(E)$-periodic solutions such that $\omega(E)=1-3hE/4+\mathcal O(E^2)$ as $E\to 0$.

{\bf 1}. Consider the perturbed non-autonomous system in the following form:
\begin{gather}
\label{ex1}
\frac{dx}{dt}=\partial_y H_0(x,y), \quad \frac{dy}{dt}=-\partial_x H_0(x,y)+ t^{-\frac 12} \big(a(S(t))x+b(S(t))y\big), \quad t \geq 1,
\end{gather}
where $S(t)\equiv t+s_1 t^{1/2}+s_2 \log t$, $a(S)\equiv a_0+a_1\cos S$ and  $b(S)\equiv b_0+b_1\cos S$.
System \eqref{ex1} is of form \eqref{FulSys} with $q=2$ and $\varkappa=1$.
It can easily be checked that the change of the variables described in section~\ref{sec2} with $l=2$, $N=M=4$, $v_3 \equiv \psi_3 \equiv 0$,
\begin{eqnarray*}
 v_2 &\equiv& -\frac{\mathcal E}{6}
        \Big( 3c_0 \cos(2S+2\theta+\delta_0)
            +3c_1\cos (S+2\theta+\delta_1)
            +c_1  \cos(3S+2\theta+\delta_1)+6 b_1 \sin S\Big) ,\\
\psi_2 &\equiv& \frac{1}{12} \Big( 3 c_0  \sin(2 S+2\theta+\delta_0) +3 c_1 \sin(S+2 \theta+\delta_1)+ c_1 \sin(3 S + 2 \theta+\delta_1)  + 6 a_1 \sin S\Big),\\
 v_4 &\equiv&   \frac{\mathcal E}{144}
\Big(48 (a_0 a_1+b_0 b_1)  \cos S
-24( a_1 b_0 - a_0 b_1 - 3 b_1 s_1) \sin S +3(4 b_1^2- a_1^2+ 5 a_1 b_1 ) \sin( 4S + 2 \theta)\\
&&
-36( a_0 a_1+ b_0 b_1+a_1 s_1) \cos(S+2 \theta)
+36(a_1 b_0 + a_0 b_1 +  b_1 s_1 ) \sin(S+2 \theta)\\
&&
- 4 (5a_0 a_1- 9 b_0 b_1-a_1 s_1) \cos (3S+2 \theta)
+4 (11 a_0 b_1+3 a_1 b_0 - b_1 s_1 ) \sin(3 S + 2 \theta)\\
&&
-12(2 a_1^2   +  3 a_0^2   - 2 b_1^2 ) \cos(2 S + 2 \theta)
+12(3 a_0 b_0 + 4 a_1 b_1) \sin(2 S + 2 \theta)+12( a_1^2-2 b_1^2) \cos 2 S\Big),\\
\psi_4
                &\equiv&
(a_1 b_0 - a_0 b_1) \frac{\cos S }{12}- (a_1 s_1-\frac76 (a_0 a_1 +  b_0 b_1)  ) \frac{\sin S}{4}  +c_1^2 \frac{ \sin 2 S}{24}+ ( a_0^2   + \frac 23 a_1^2 )\frac{ \sin(2 S + 2 \theta)}{8}  \\
&&
+ (a_1 b_0 +  b_1(2 a_0 +  s_1))\frac{\cos(S+2 \theta)}{8}  + a_1(3 a_0  + s_1 ) \frac{ \sin(S+2\theta)}{8}  + (3 a_0 b_0+2 a_1 b_1) \frac{\cos(2 S + 2 \theta) }{24}
 \\
&&
+ a_1 b_1\frac{\cos(2 S + 4 \theta)}{16}  + (a_1^2-b_1^2) \frac{\sin(2 S + 4 \theta)}{32}+ a_1( b_1+a_1)\frac{\sin(4 S + 2 \theta)}{96}+  a_1 b_1 \frac{\cos(6S + 4 \theta) }{144}   \\
&&
+ ( a_1 b_0+\frac {b_1}{3}(2 a_0 - s_1)) \frac{\cos(3 S + 2 \theta) }{24}+ a_1(5  a_0  -   s_1 )\frac{ \sin(3 S + 2 \theta)}{72} + (a_0 a_1 - b_0 b_1)\frac{\sin(3 S + 4 \theta)}{16}  \\
&&
+ (a_1 b_0+a_0 b_1)\frac{\cos(3 S + 4 \theta)}{16}    + ( a_0 b_0 + \frac 23 a_1 b_1 )\frac{ \cos(4 S + 4 \theta)}{16} + ( a_1^2 - b_1^2 ) \frac{\sin(6 S + 4 \theta)}{288}  \\
&&
+ ( a_1 b_0 + a_0 b_1 )\frac{\cos(5 S + 4 \theta)}{48}  +  (a_0 a_1  - b_0 b_1) \frac{\sin(5 S + 4 \theta)}{48} + ((a_0^2-b_0^2)+\frac 23 (a_1^2-b_1^2)) \frac{ \sin(4 S + 4 \theta)}{32}
 \end{eqnarray*}
transforms system \eqref{ex1} to the following:
\begin{gather}
\label{e1vp}
\frac{dv}{dt}=t^{-\frac{1}{2}}\Lambda_2(v,\psi)+t^{-1}\Lambda_4(v,\psi)+\tilde \Lambda_4(v,\psi,t),\quad
\frac{d\psi}{dt}=t^{-\frac{1}{2}}\Omega_2(v,\psi)+t^{-1}\Omega_4(v,\psi)+\tilde \Omega_4(v,\psi,t),
\end{gather}
where
\begin{gather*}
     \Lambda_2(v,\psi)=b_0 v, \quad \Lambda_4(v,\psi)= \frac{v}{4}\Big(4 -  a_1 c_1 \sin (2\psi+\delta_1) \Big),\\
     \Omega_2(v,\psi)=-\frac 12 (s_1+a_0), \quad
      \Omega_4(v,\psi)=-s_2-\frac{1}{24}\Big(3c_0^2+ 2c_1^2+  3 a_1 c_1 \cos(2\psi+\delta_1)\Big)-\frac{3h}{4}v,\\
      c_0=\sqrt{a_0^2+b_0^2}, \quad c_1=\sqrt{a_1^2+b_1^2}, \quad \delta_0=\arccos \Big(\frac{a_0}{ c_0}\Big), \quad \delta_1=\arccos \Big(\frac{a_1}{ c_1}\Big),
\end{gather*}
 $\tilde \Lambda_4=\mathcal O(t^{-3/2})$, $\tilde \Omega_4=\mathcal O(t^{-3/2})$ as $t\to\infty$ for all $v\in [0,\Delta_0]$ and $\psi\in\mathbb R$.

Let us consider several possible cases.

(I) Let $s_1+a_0\neq 0$. If $b_0\neq 0$, then system \eqref{e1vp} satisfies the conditions of  Theorem~\ref{Th4} with $n=m=2<2q$, $\lambda_{2}(\psi)\equiv b_0$. In this case, phase drifting occurs: $|\psi(t)|\to \infty$ as $t\to \infty$, and the equilibrium $(0,0)$ is exponentially stable if $b_0<0$ and unstable if $b_0>0$ (see~Fig.~\ref{Fig01}).
If  $b_0=0$, system \eqref{e1vp} satisfies the conditions of  Theorem~\ref{Th5} with $n=4=2q$, $m=2$ and $\widehat\gamma_{4,2}=0$. From Remark~\ref{Rem4} it follows that the equilibrium $(0,0)$ of system \eqref{ex1} is at least unstable with a weight $t^{1/2}$.

\begin{figure}
\centering
\subfigure[$b_0=-0.5$]{
\includegraphics[width=0.28\linewidth]{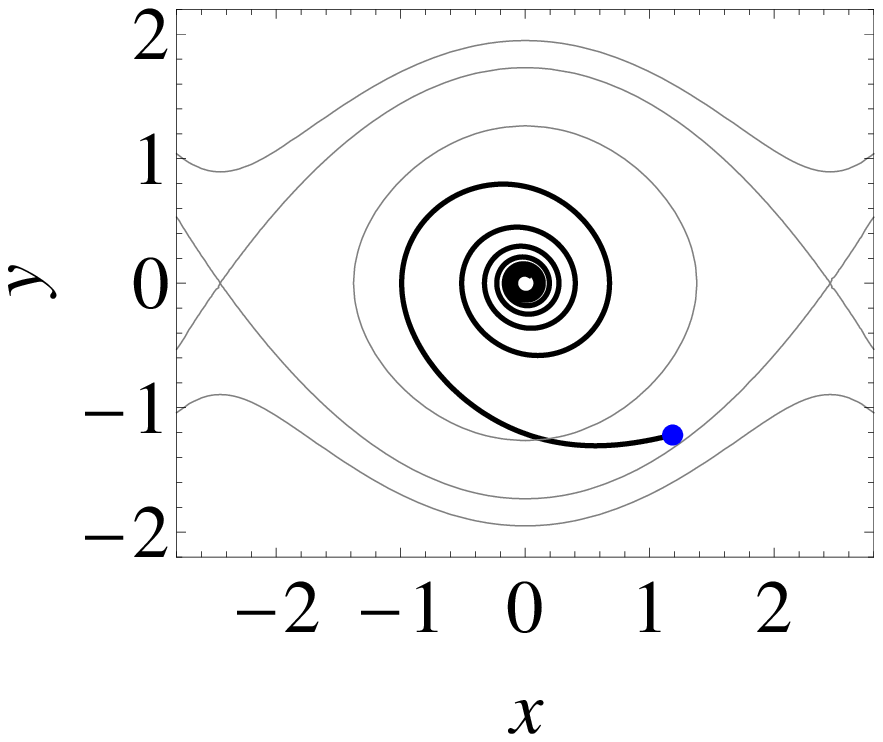} \ \
\includegraphics[width=0.32\linewidth]{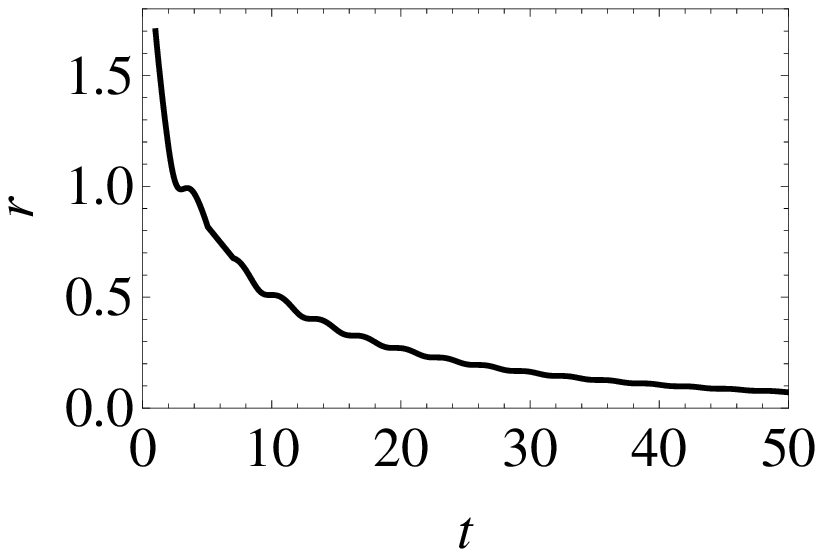} \ \
\includegraphics[width=0.32\linewidth]{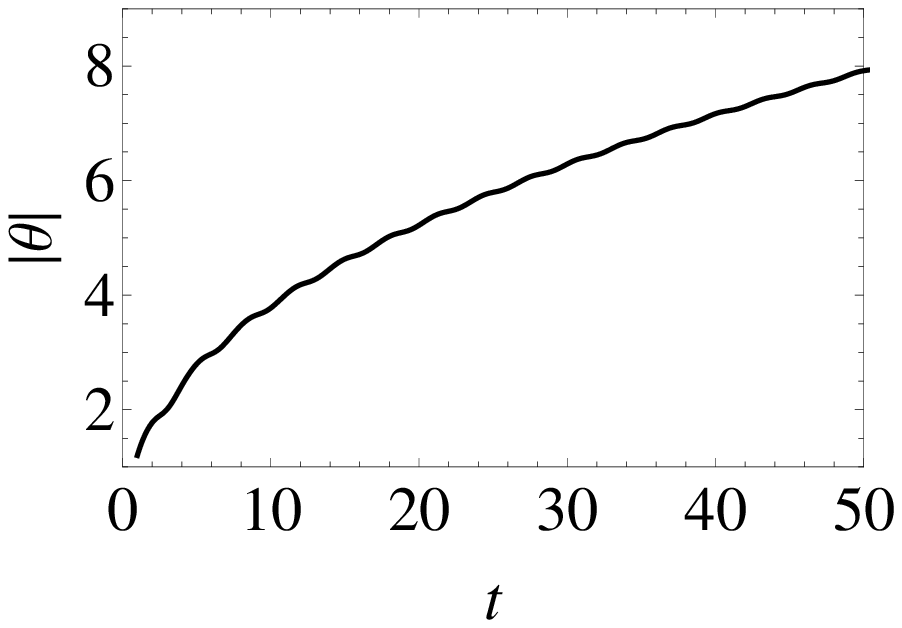}
}
\\
\subfigure[$b_0=0.5$]{
\includegraphics[width=0.28\linewidth]{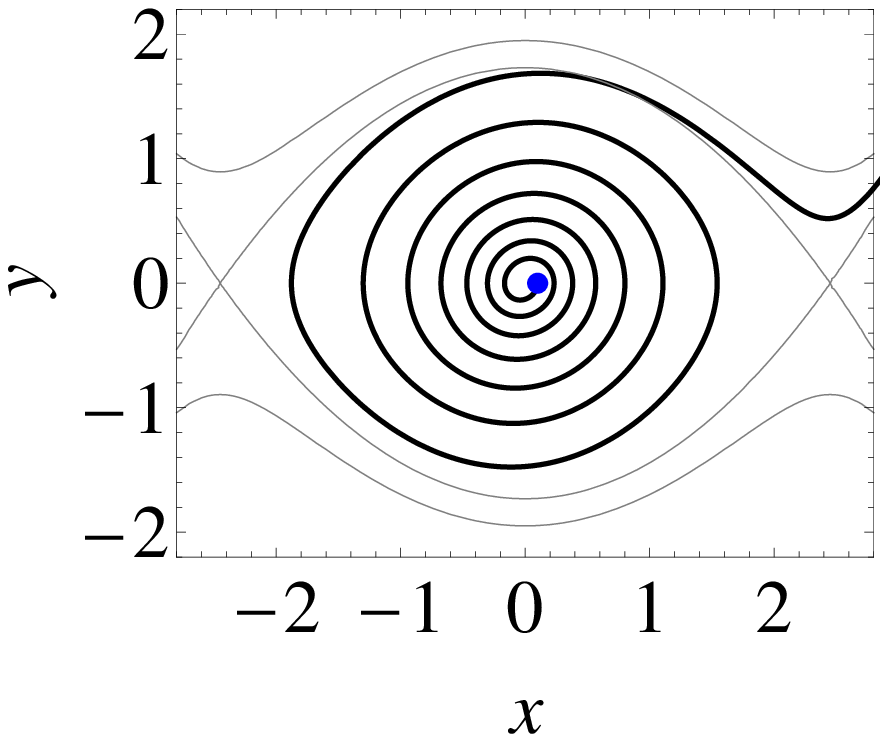} \ \
\includegraphics[width=0.32\linewidth]{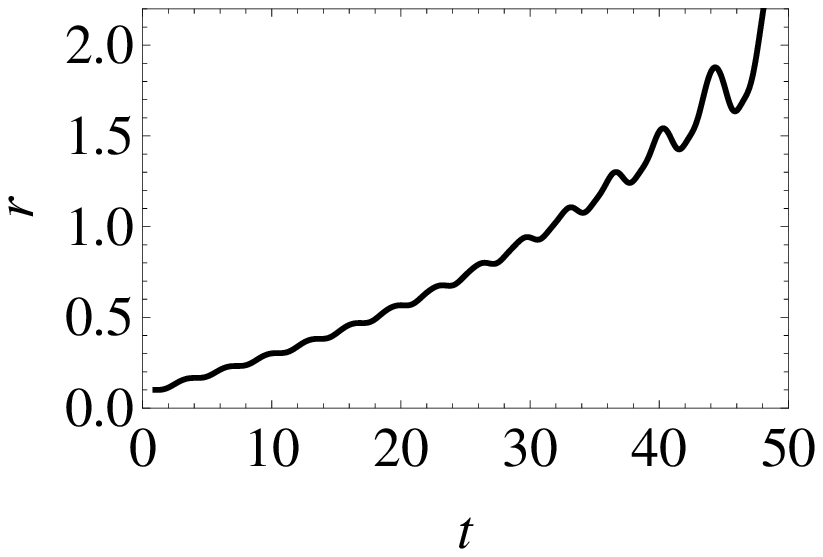} \ \
\includegraphics[width=0.32\linewidth]{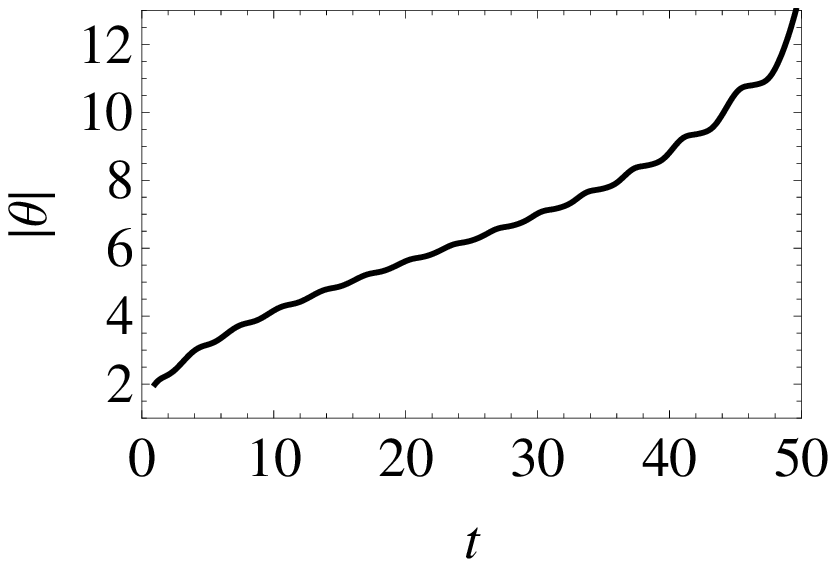}
}
\caption{\footnotesize The evolution of $(x(t),y(t))$,  $r(t)$, $|\theta(t)|$ for solutions of \eqref{ex1} with $a_0=a_1=b_1=s_2=0$, $h=1/6$, $s_1=1$, where $x(t)=r(t)\cos(\theta(t)+S(t))$, $y(t)=-r(t)\sin(\theta(t)+S(t))$. The blue points correspond to initial data $(x(1),y(1))$. Gray solid curves correspond to level lines of $H_0(x,y)$.} \label{Fig01}
\end{figure}

(II) Let $s_1+a_0=0$,  $a_1>0$ and
\begin{gather}\label{e1cond}
-\frac{1}{24} \big(3a_1c_1+ 3 c_0^2+2 c_1^2 \big)< s_2<\frac{1}{24}\big(3a_1 c_1- 3 c_0^2-2 c_1^2 \big).
\end{gather}
In this case $\Omega_2\equiv 0$ and $\Omega_4$ satisfies \eqref{as01} with
\begin{gather*}
    \psi_\ast=-\frac{1}{2}\arccos\Big(-\frac{24s_2+3c_0^2+2c_1^2}{3 a_1 c_1}\Big)-\frac{\delta_1}{2}+\pi j, \quad j\in\mathbb Z,
    \quad
    \vartheta_4=\frac{a_1c_1}{4}\sin(2\psi_\ast+\delta_1)<0.
 \end{gather*}
If $b_0\neq 0$, system \eqref{e1vp} satisfies the conditions of  Theorem~\ref{Th1} with $n=2$, $m=4=2q$, $\lambda_2(\psi_\ast)\equiv b_0$. In this case, phase locking occurs and the stability of the equilibrium $(0,0)$ depends on the sign of $b_0$ (see~Fig.~\ref{Fig02}).
If $b_0=0$, system \eqref{e1vp} satisfies the conditions of Theorem~\ref{Th1} with $n=m=4=2q$ and $\lambda_4(\psi_\ast)=-\vartheta_4>0$. Hence, the equilibrium $(0,0)$ is  unstable. Note that if $h=0$, then $\omega(E)\equiv 1$, $l=0$ and $v(t)\sim t^{|\vartheta_4|}$ as $t\to\infty$ (see Fig.~\ref{Fig03}).

\begin{figure}
\centering
\subfigure[$b_0=-0.6$]{
\includegraphics[width=0.28\linewidth]{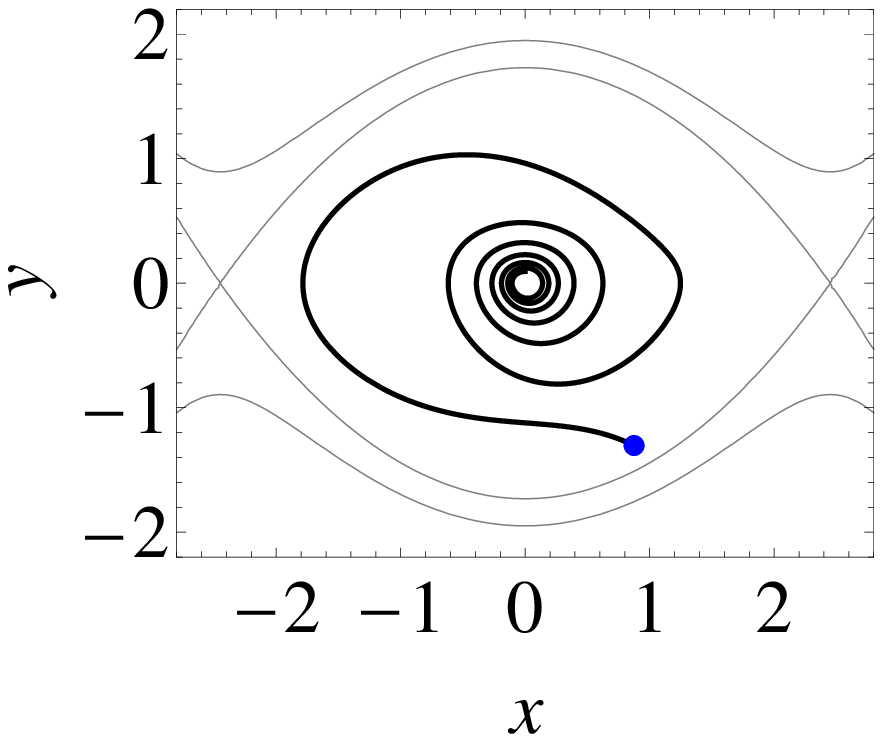} \ \
\includegraphics[width=0.32\linewidth]{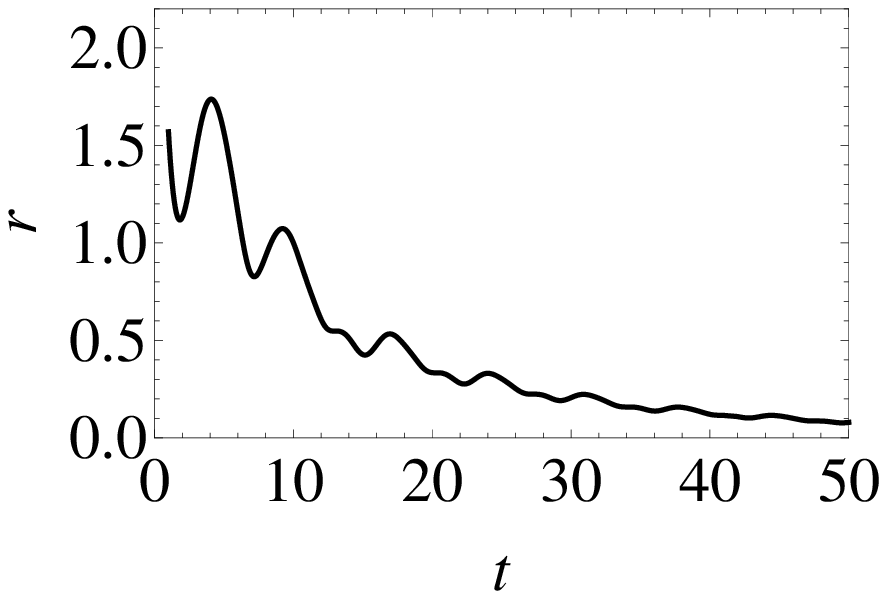} \ \
\includegraphics[width=0.32\linewidth]{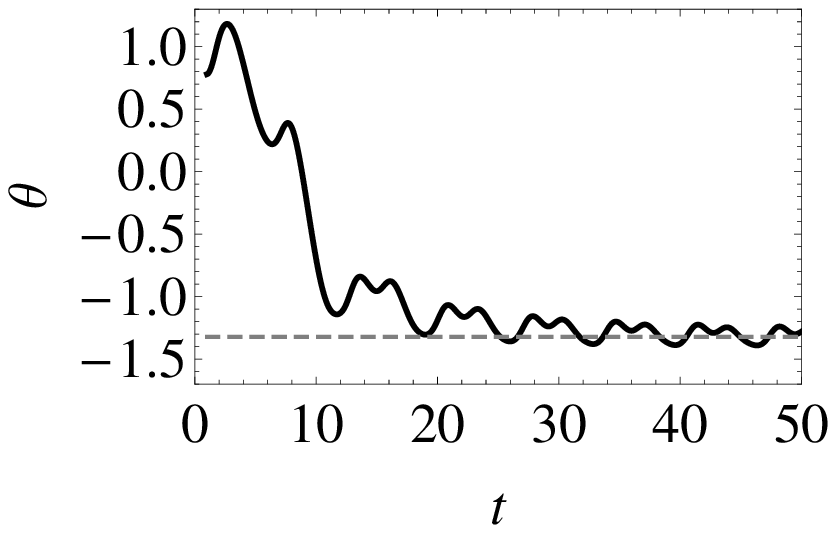}
}
\\
\subfigure[$b_0=0.6$]{
\includegraphics[width=0.28\linewidth]{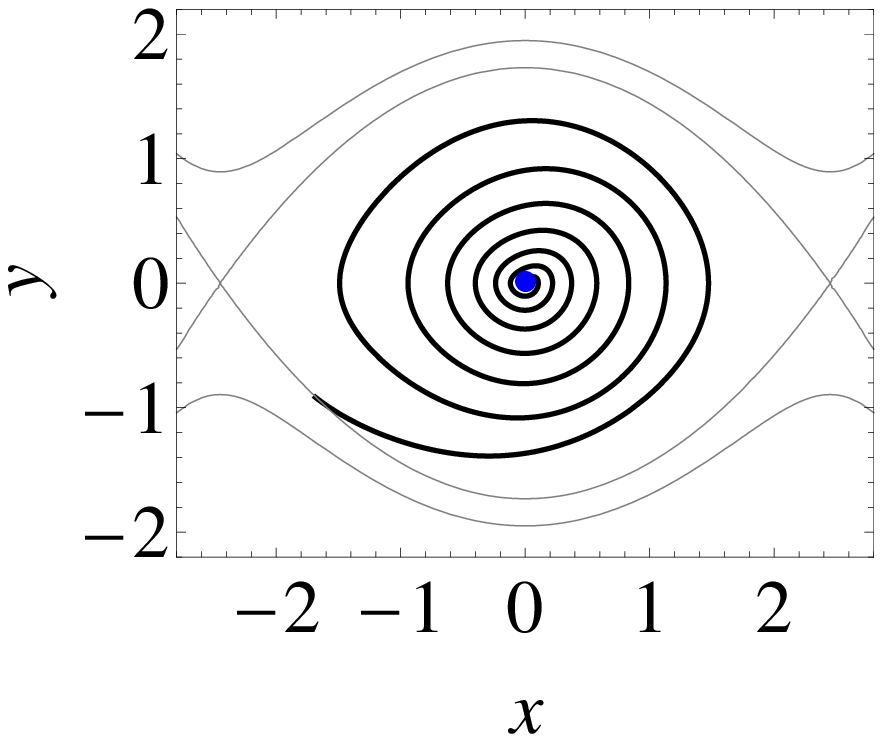} \ \
\includegraphics[width=0.32\linewidth]{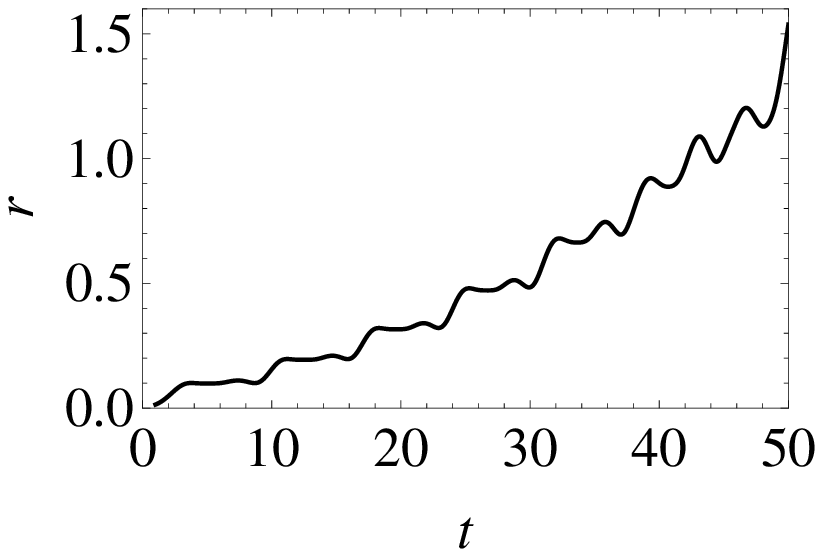} \ \
\includegraphics[width=0.32\linewidth]{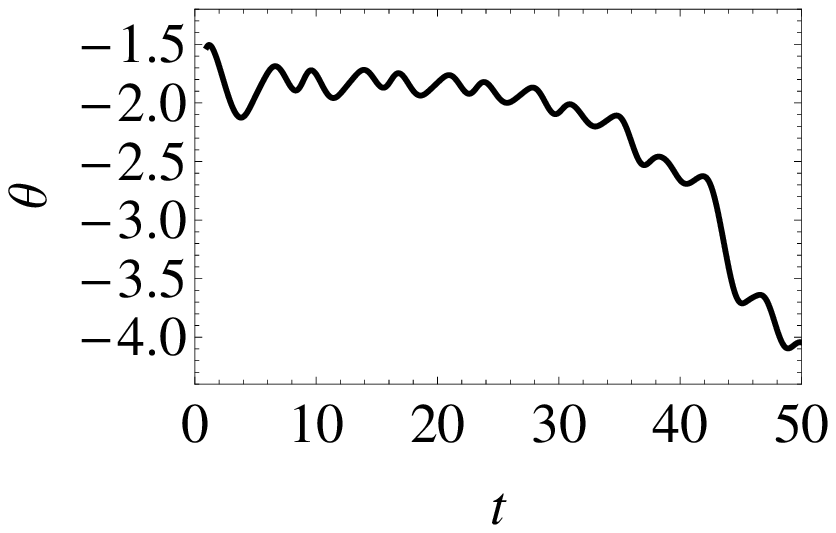}
}
\caption{\footnotesize The evolution of $(x(t),y(t))$, $r(t)$, $\theta(t)$ for solutions of \eqref{ex1} with $a_0=a_1=0.8$, $b_1=0.6$, $s_1=-0.8$, $s_2=-1/6$, $h=1/6$, where $x(t)=r(t)\cos(\theta(t)+S(t))$, $y(t)=-r(t)\sin(\theta(t)+S(t))$. The blue points correspond to initial data $(x(1),y(1))$. The gray solid curves correspond to level lines of $H_0(x,y)$. The gray dashed curve corresponds to $\theta= \psi_\ast$, where $\psi_\ast\approx -1.322$.} \label{Fig02}
\end{figure}

\begin{figure}
\centering
\includegraphics[width=0.25\linewidth]{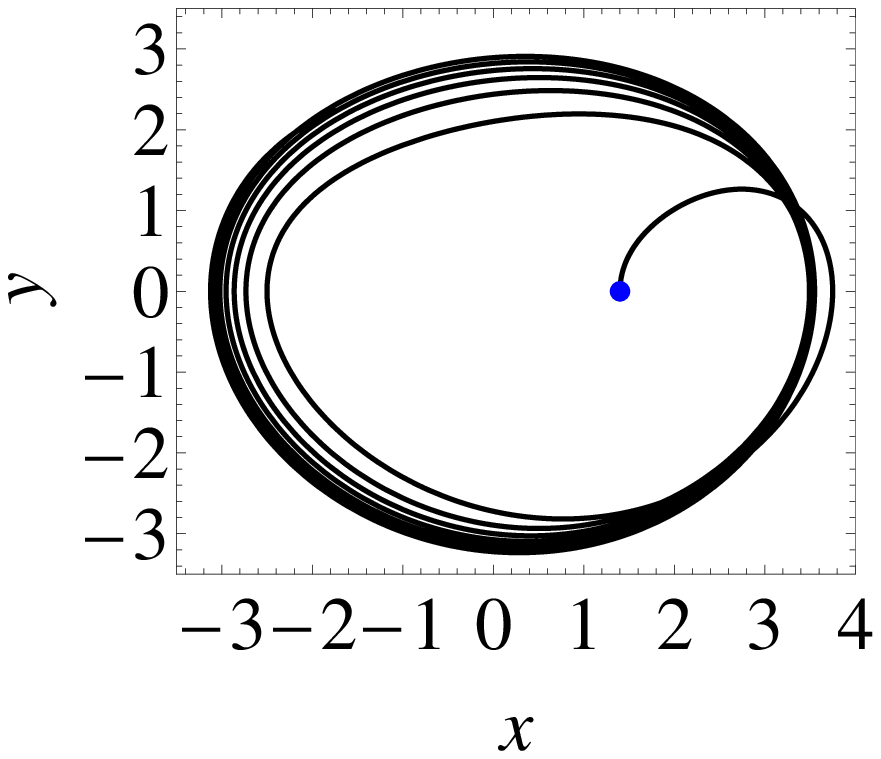} \ \
\includegraphics[width=0.32\linewidth]{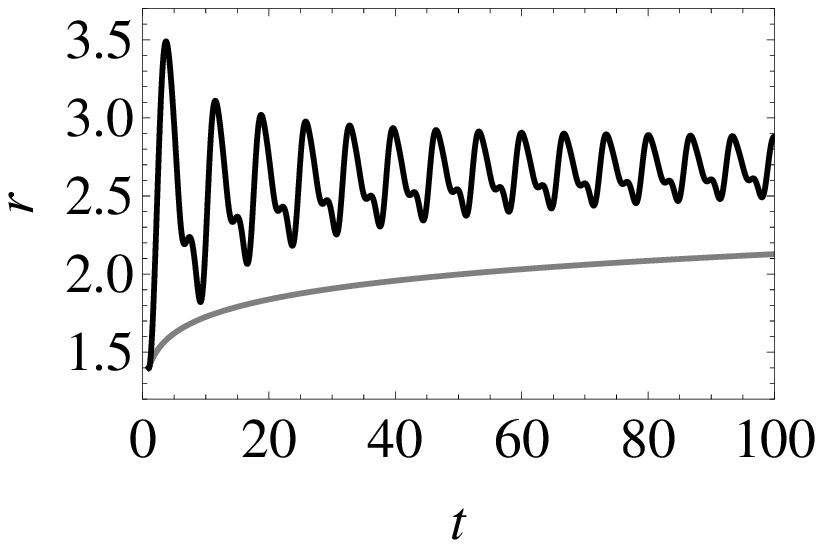} \ \
\includegraphics[width=0.32\linewidth]{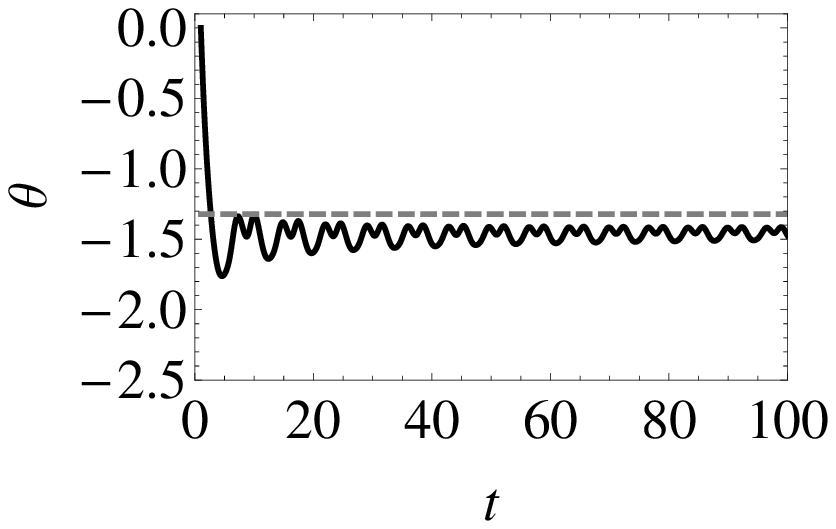}
\caption{\footnotesize The evolution of $(x(t),y(t))$, $r(t)$, $\theta(t)$ for solutions of \eqref{ex1} with $a_0=1$, $a_1=0.8$, $b_0=0$, $b_1=0.6$, $s_1=-1$, $s_2=-1/6$, $h=0$, where $x(t)=r(t)\cos(\theta(t)+S(t))$, $y(t)=-r(t)\sin(\theta(t)+S(t))$. The blue point corresponds to initial data $(x(1),y(1))$. The gray solid curve corresponds to $r=r(1) t^{{|\vartheta_4|}/{2}}$. The gray dashed curve corresponds to $\theta= \psi_\ast$, where $\psi_\ast\approx -1.322$.} \label{Fig03}
\end{figure}

(III) Let $s_1+a_0=0$, $a_1> 0$ and assumption \eqref{e1cond} does not hold such that
$
  |24s_2+3c_0^2+2c_1^2|>3 a_1 c_1
$.
Then it follows from Theorem~\ref{Th4} that the stability of the equilibrium $(0,0)$ is determined by the sign of $b_0$ (see Fig.~\ref{Fig04}).
\begin{figure}
\vspace{-2ex}
\centering
\subfigure[$b_0=-0.6$]{
\includegraphics[width=0.28\linewidth]{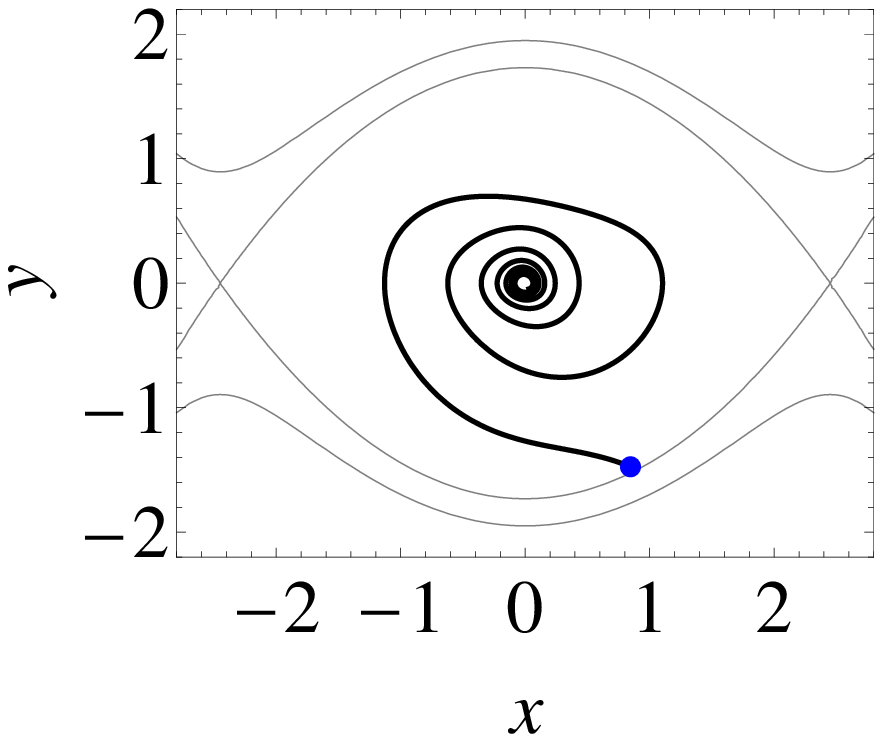} \ \
\includegraphics[width=0.32\linewidth]{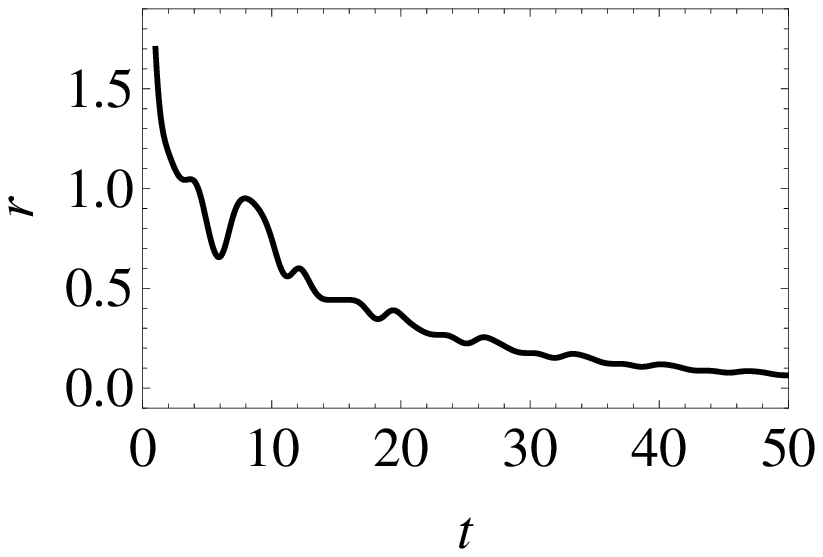} \ \
\includegraphics[width=0.32\linewidth]{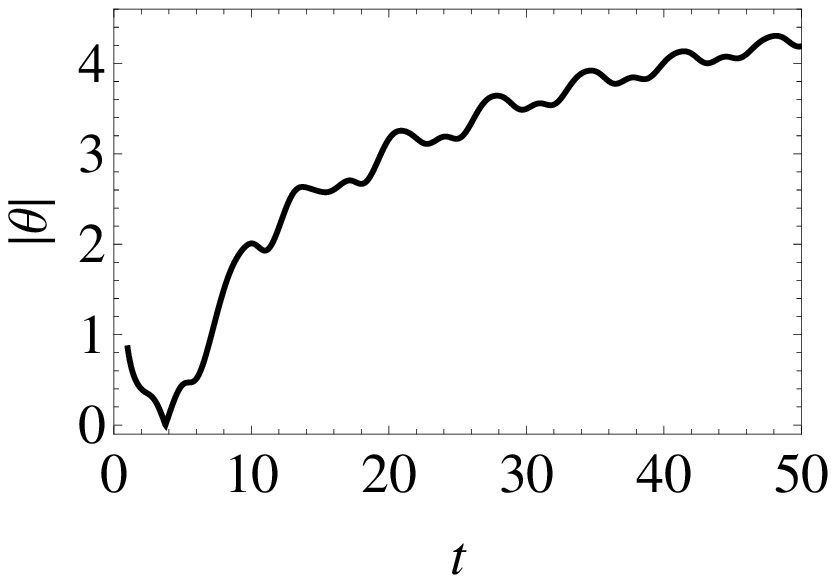}
}
\\
\subfigure[$b_0=0.6$]{
\includegraphics[width=0.28\linewidth]{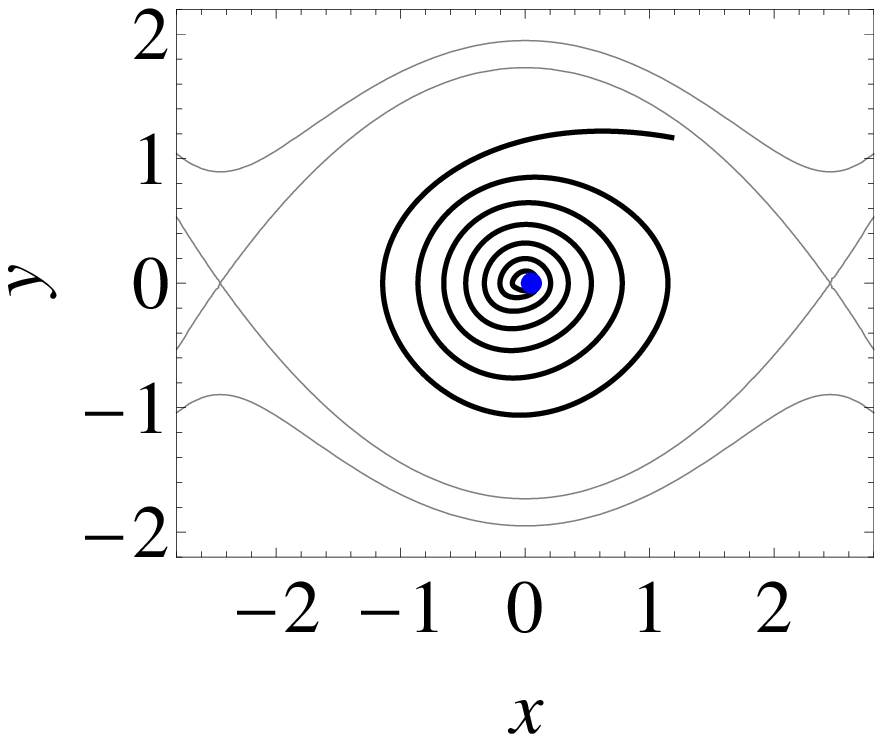} \ \
\includegraphics[width=0.32\linewidth]{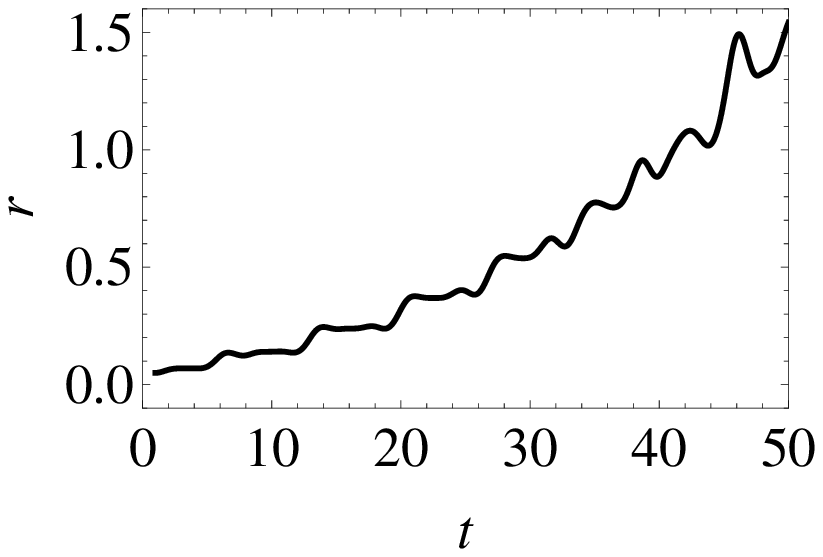} \ \
\includegraphics[width=0.32\linewidth]{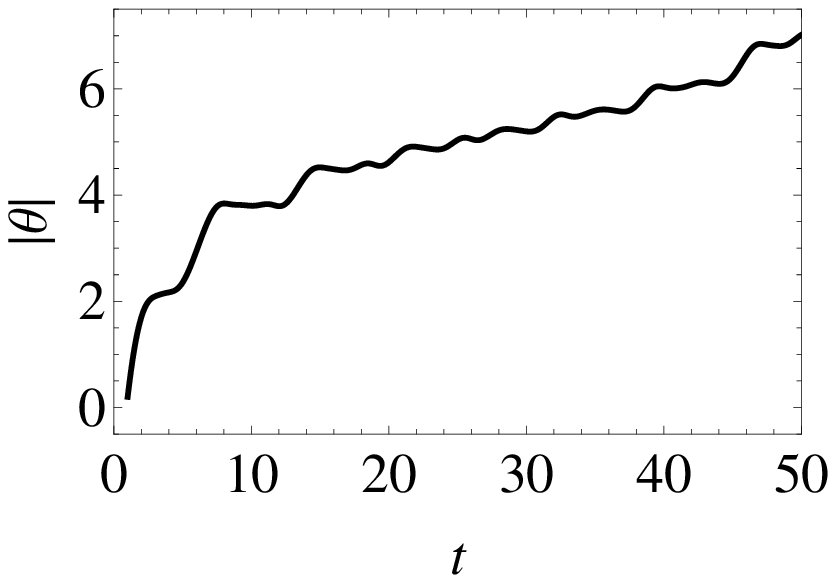}
}
\caption{\footnotesize The evolution of $(x(t),y(t))$, $r(t)$, $|\theta(t)|$ for solutions of \eqref{ex1} with $a_0=a_1=0.8$, $b_1=0.6$, $s_1=-0.8$, $s_2=1$, $h=1/6$, where $x(t)=r(t)\cos(\theta(t)+S(t))$, $y(t)=-r(t)\sin(\theta(t)+S(t))$. The blue points correspond to initial data $(x(1),y(1))$. The gray solid curves correspond to level lines of $H_0(x,y)$.} \label{Fig04}
\end{figure}

{\bf 2}.  Consider a similar non-autonomous system, but with another perturbation phase:
\begin{gather}
\label{ex2}
\frac{dx}{dt}=\partial_y H_0(x,y), \quad \frac{dy}{dt}=-\partial_x H_0(x,y)+ t^{-\frac 12} \big(a(S(t))x+b(S(t))y\big), \quad t \geq 1,\\
\nonumber S(t)\equiv 2t+s_1 t^{1/2}+s_2 \log t, \quad a(S)\equiv a_0+a_1\cos S, \quad b(S)\equiv b_0+b_1\cos S.
\end{gather}
This system is of form \eqref{FulSys} with $q=2$ and $\varkappa=2$.
Under the transformation described in section~\ref{sec2} with $l=1$, $N=M=2$,
\begin{eqnarray*}
 v_2&=& -\frac{\mathcal E}{8}\Big(  4 c_0 \cos(S+2 \theta+\delta_0)+ c_1 \cos(2 S+2 \theta+\delta_1)+4 b_1 \sin S\Big),\\
 \psi_2&=& \frac{1}{16} \Big( 4 c_0  \sin(S+2\theta+\delta_0)+c_1 \sin(2 S+2\theta+\delta_1)  + 4 a_1 \sin S \Big),
 \end{eqnarray*}
system \eqref{ex2} is reduced to the following:
\begin{gather*}
\label{e2vp}
\frac{dv}{dt}=t^{-\frac 12}\Lambda_2(v,\psi)+\tilde \Lambda_2(v,\psi,t),\quad
\frac{d\psi}{dt}=t^{-\frac 12}\Omega_2(v,\psi)+\tilde \Omega_2(v,\psi,t),
\end{gather*}
where
\begin{align*}
&\Lambda_2(v,\psi)=\Big(b_0 -\frac{c_1}{2} \sin (2\psi+\delta_1)\Big) v,\quad \Omega_2(v,\psi)=-\frac{1}{4} \Big(2a_0 +s_1 + c_1 \cos(2\psi+\delta_1)+3 h v\Big),
\end{align*}
and $\tilde \Lambda_2=\mathcal O(t^{-1})$, $\tilde \Omega_2=\mathcal O(t^{-1})$ as $t\to\infty$ for all $v\in [0,\Delta_0]$, $\psi\in\mathbb R$.

If $-c_1-2a_0<s_1<c_1-2a_0$, then $\Omega_2$ satisfies \eqref{as01} with
\begin{gather*}
    \psi_\ast=-\frac{1}{2}\arccos\Big(-\frac{2a_0+s_1}{c_1}\Big)-\frac{\delta_1}{2}+\pi j, \quad j\in \mathbb Z, \quad
    \vartheta_2=\frac{c_1}{2}\sin(2\psi_\ast+\delta_1)<0.
 \end{gather*}
In this case, $\lambda_2(\psi_\ast)=b_0-\vartheta_2$. Hence, it follows from Theorem~\ref{Th1} that the equilibrium $(0,0)$ is exponentially stable if $b_0-\vartheta_2< 0$ and unstable if $b_0-\vartheta_2> 0$. (see~Fig.~\ref{Fig21}).

If either $s_1<-c_1-2a_0$ or $s_1>c_1-2a_0$, system \eqref{e2vp} satisfies the conditions of Theorem~\ref{Th5} with $n=m=2<2q$, $\widehat\gamma_{2,2}=b_0 \widehat\chi_{2}$, and $\widehat \chi_2:=\langle|\Omega_2(0,\psi)|^{-1}\rangle_\psi>0$. Therefore, the stability of equilibrium $(0,0)$ is determined by the sign of $b_0$ (see Fig.~\ref{Fig22}).

\begin{figure}
\centering
\subfigure[$b_0=-0.6$]{
\includegraphics[width=0.28\linewidth]{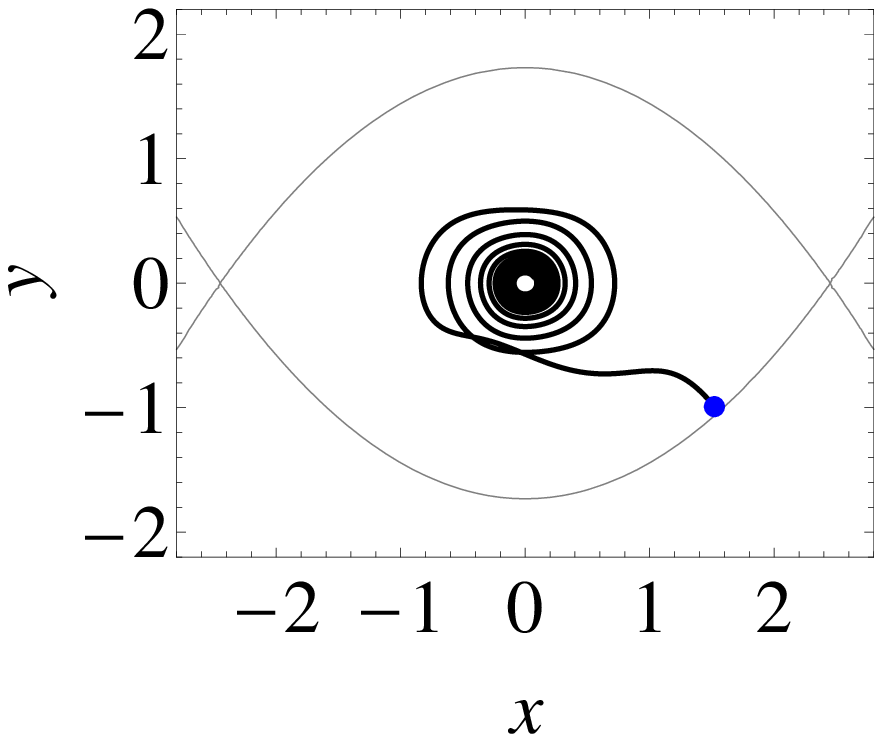} \ \
\includegraphics[width=0.32\linewidth]{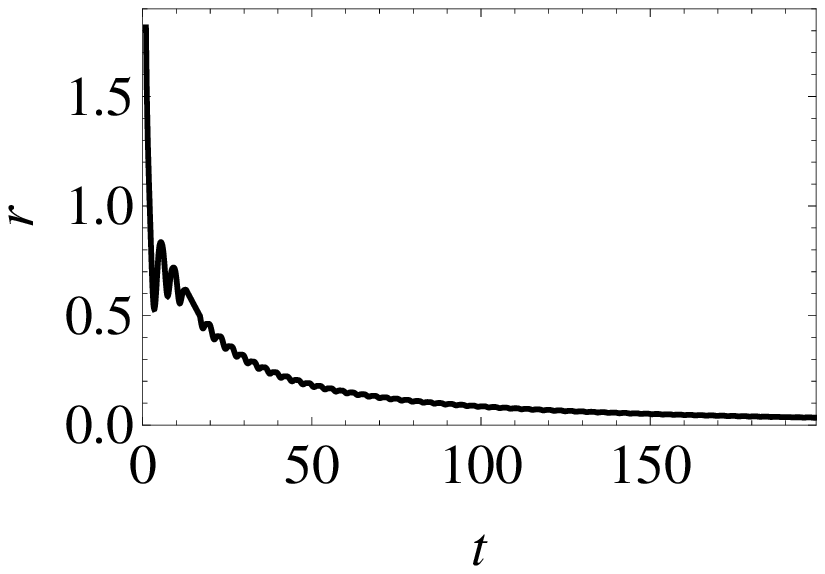} \ \
\includegraphics[width=0.32\linewidth]{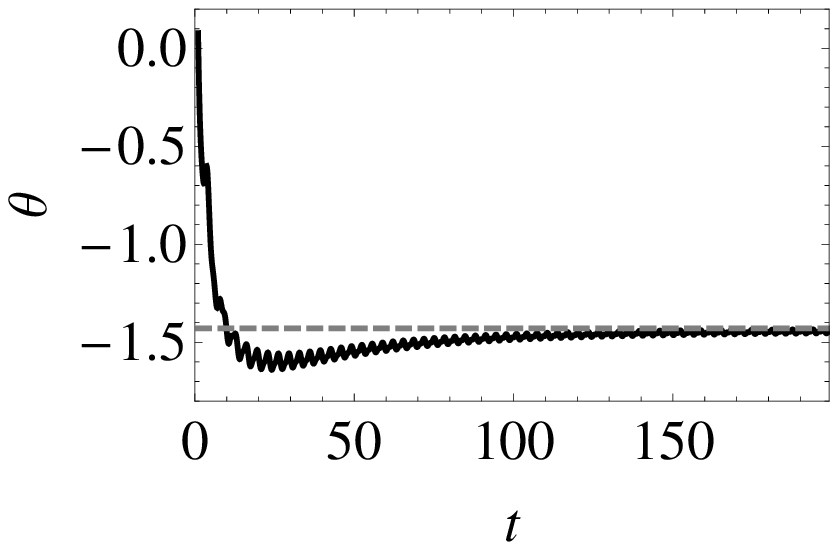}
}
\\
\subfigure[$b_0=-0.2$]{
\includegraphics[width=0.28\linewidth]{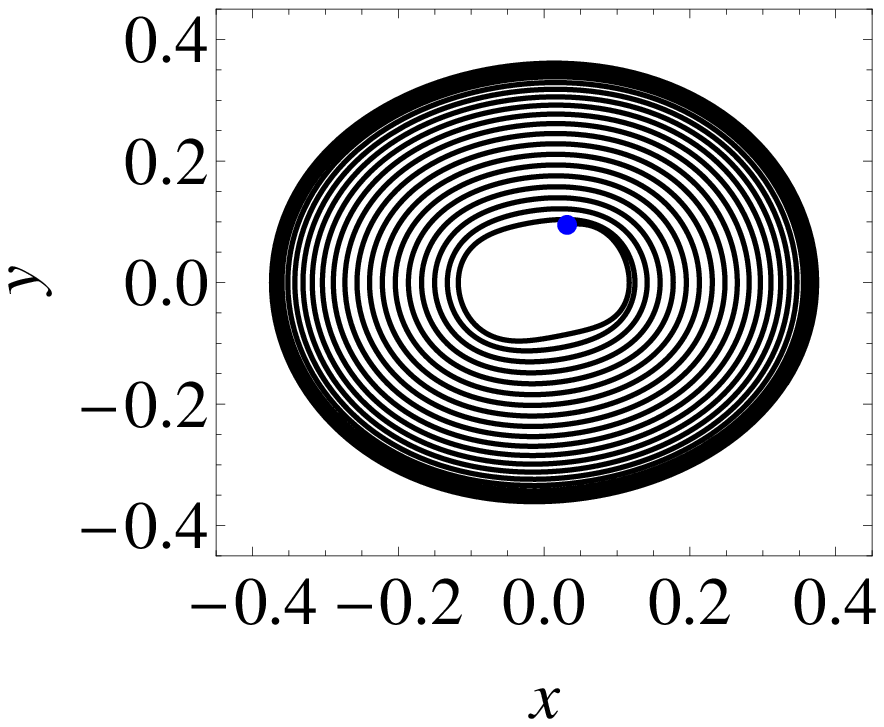} \ \
\includegraphics[width=0.32\linewidth]{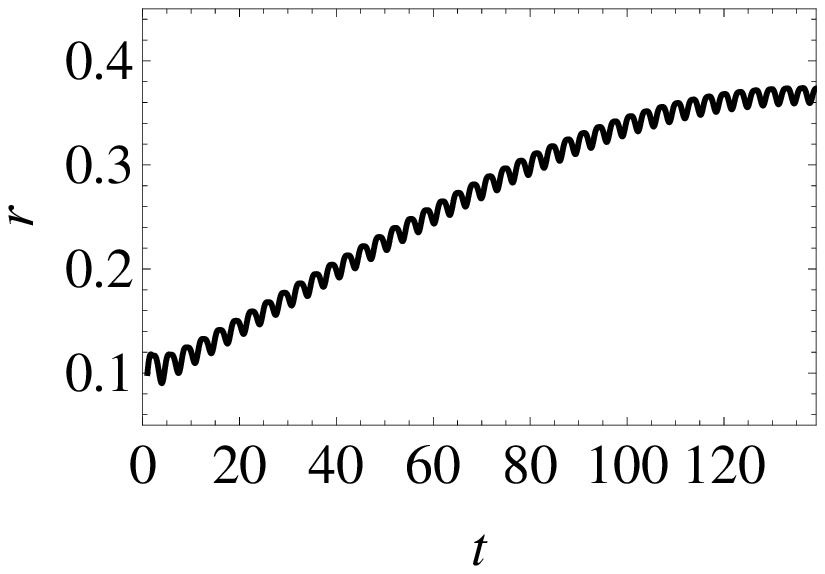} \ \
\includegraphics[width=0.32\linewidth]{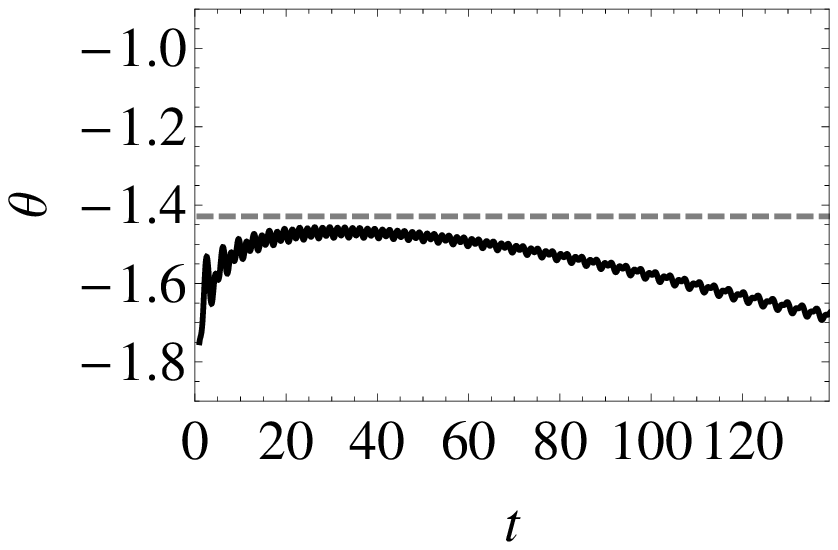}
}
\caption{\footnotesize The evolution of $(x(t),y(t))$, $R(t)$, $\theta(t)$ for solutions of \eqref{ex2} with $a_0=a_1=0.8$, $b_1=0.6$, $s_1=-1$, $s_2=0$, $h=1/6$, where $x(t)=r(t)\cos(\theta(t)+S(t)/2)$, $y(t)=-r(t)\sin(\theta(t)+S(t)/2)$. The blue points correspond to initial data $(x(1),y(1))$. The gray solid curves correspond to level lines of $H_0(x,y)$. The gray dashed curves correspond to $\theta= \psi_\ast$, where $\psi_\ast\approx -1.4289$.} \label{Fig21}
\end{figure}

\begin{figure}
\centering
\subfigure[$b_0=-0.1$]{
\includegraphics[width=0.28\linewidth]{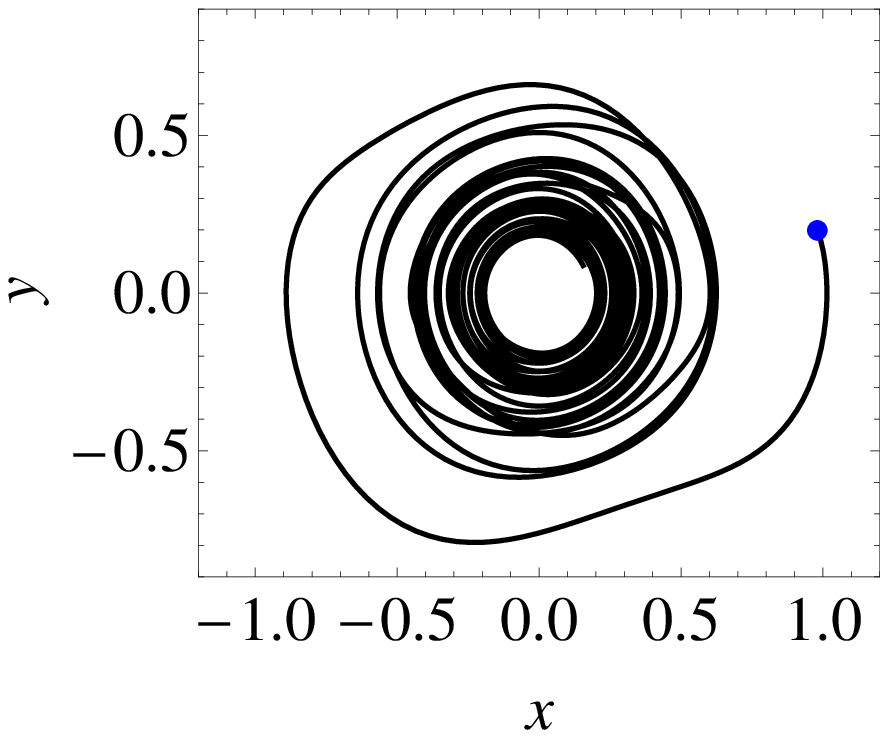} \ \
\includegraphics[width=0.32\linewidth]{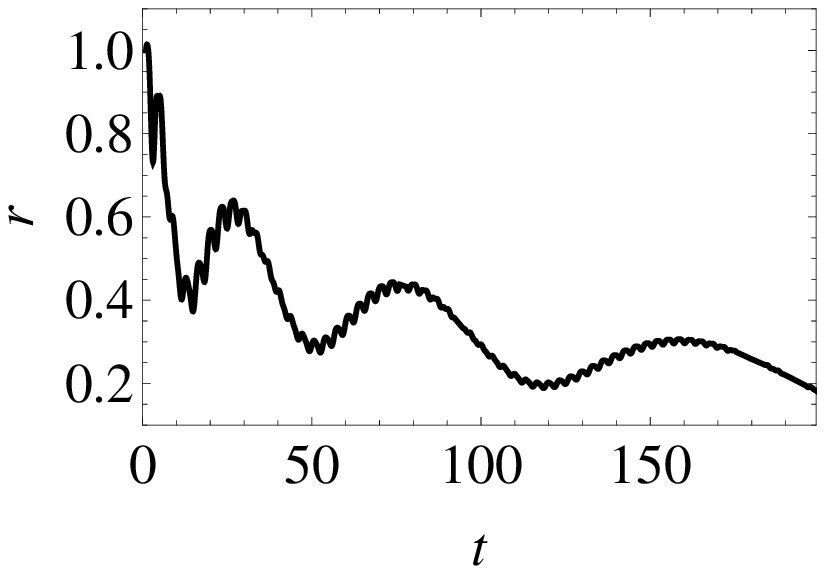} \ \
\includegraphics[width=0.32\linewidth]{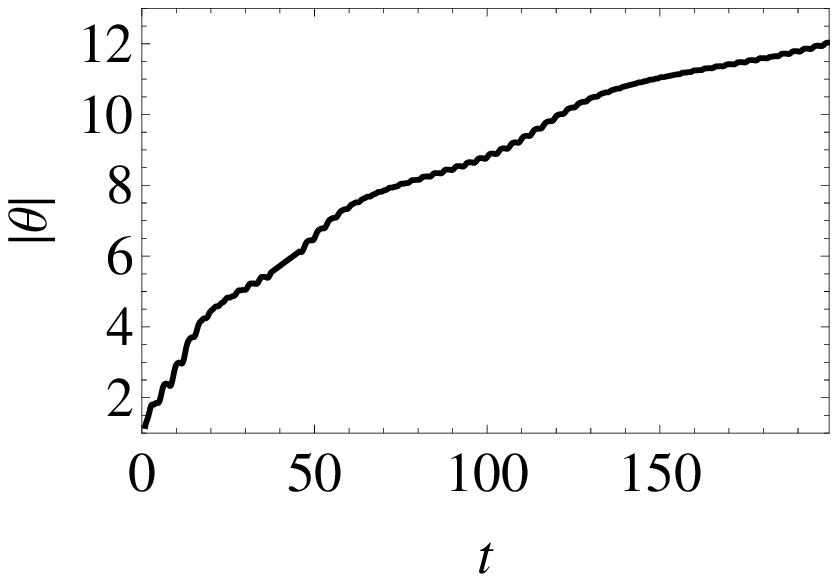}
}
\\
\subfigure[$b_0=0.1$]{
\includegraphics[width=0.28\linewidth]{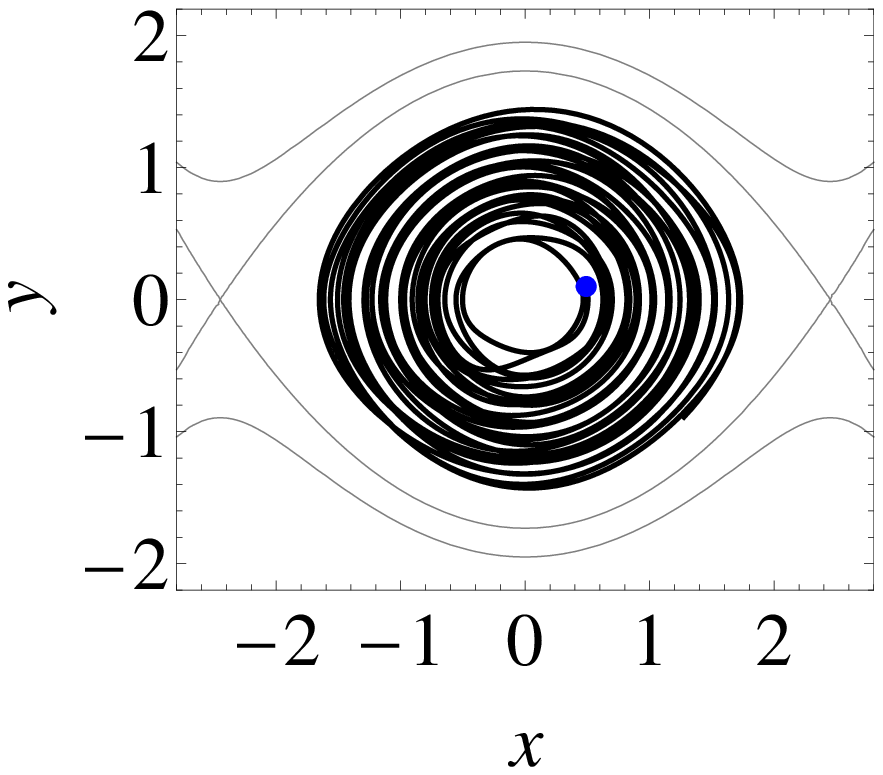} \ \
\includegraphics[width=0.32\linewidth]{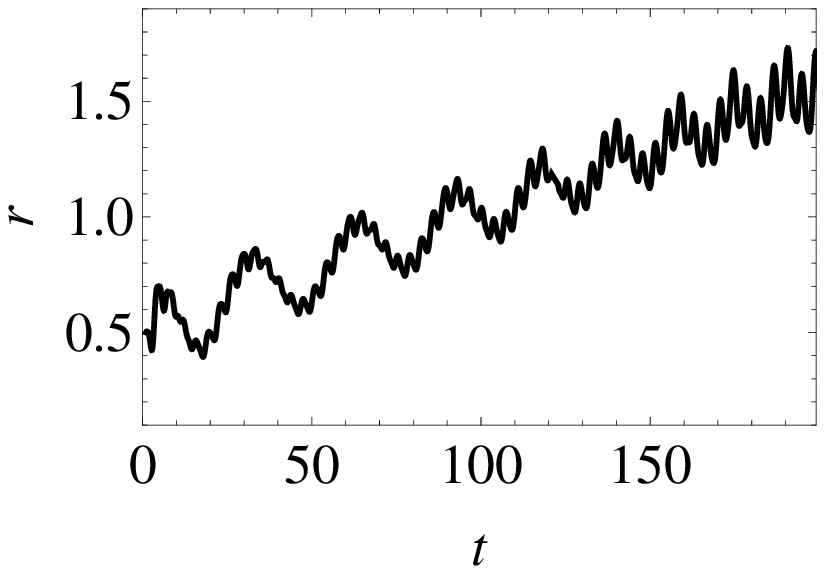} \ \
\includegraphics[width=0.32\linewidth]{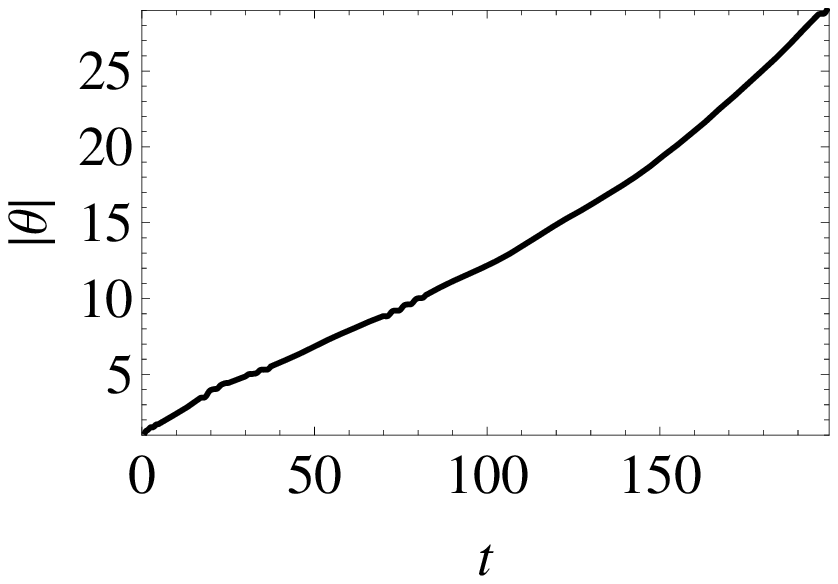}
}
\caption{\footnotesize The evolution of $(x(t),y(t))$, $R(t)$, $|\theta(t)|$ for solutions of \eqref{ex2} with $a_0=a_1=0.8$, $b_1=0.6$, $s_1=s_2=0$, $h=1/6$, where $x(t)=r(t)\cos(\theta(t)+S(t)/2)$, $y(t)=-r(t)\sin(\theta(t)+S(t)/2)$. The blue points correspond to initial data $(x(1),y(1))$. The gray solid curves correspond to level lines of $H_0(x,y)$. } \label{Fig22}
\end{figure}

{\bf 3}.  Finally, consider a little more complicated system:
\begin{gather}
\begin{split}
\label{ex3}
&
\frac{dx}{dt}=\partial_y H_0(x,y), \\
&\frac{dy}{dt}=-\partial_x H_0(x,y)+ t^{-\frac 14}  z(S(t)) x^2 y + t^{-\frac{1}{2}} a(S(t)) x+t^{-1}b(S(t)) y, \quad t \geq 1,
\end{split}
\end{gather}
with $S(t)\equiv t+s_2 t^{\frac 12} + s_4 \log t$,   $z(S)\equiv z_0+z_1 \cos S$, $a(S)\equiv a_0+a_1\cos S$, and $b(S)\equiv b_0+b_1\cos S$.
It is clear that system \eqref{ex3} is of form \eqref{FulSys} with $q=4$ and $\varkappa=1$.
The transformation, described in section~\ref{sec2}  with $l=2$, $N=M=8$, $v_k\equiv \psi_k\equiv 0$ for $k\in\{2,3,5,7\}$,
\begin{eqnarray*}
   v_4
    &\equiv &
    -\frac{\mathcal E}{6} \Big( 3 a_0 \cos(2S + 2\theta)+3 a_1\cos(S + 2 \theta) + a_1 \cos(3 S + 2 \theta)\Big),\\
    \psi_4
    &\equiv &  \frac{1}{12}
\Big(6 a_1 \sin S + 3 a_0 \sin (2S + 2\theta) + 3 a_1 \sin (S + 2 \theta) + a_1 \sin (3 S + 2 \theta)\Big),\\
    v_6
    &\equiv &  \frac{\mathcal E^2}{2}\Big(\frac{ z_0}{4}  \sin (4 S+4\theta) -  z_1 \sin S +  \frac{z_1}{6} \sin(3 S + 4 \theta)+ \frac{z_1}{10}\sin (5 S + 4 \theta)\Big),\\
    \psi_6
        &\equiv &   \frac{\mathcal E}{4}\Big(\frac{z_0}{4}(
 8\cos^4(S + \theta)-3)  + \frac{z_1}{12} (4\cos(3 S + 2 \theta) + \cos(3 S + 4 \theta))  + \frac{z_1}{10} \cos(5 S + 4 \theta) \\
    &&+ z_1 \cos(S+2 \theta) \Big),
\\
    v_8
    &\equiv & \frac{\mathcal E}{12}
    \Big(
        4 a_0 a_1 \cos S + a_1^2 \cos 2 S -  9 a_0 a_1 \cos(S+2 \theta) - 3 a_1 s_2 \cos(S+2 \theta)- b_1 \sin S\\
    &&-3 a_0^2 \cos(2 S + 2 \theta) - 2 a_1^2 \cos(2 S + 2 \theta) - \frac{5}{3} a_0 a_1 \cos(3 S + 2 \theta)  + 6 b_1 \sin(S+2 \theta) \\
    &&  + 6 b_0 \sin(2 S + 2 \theta) +  2 b_1 \sin(3 S + 2 \theta)+ \frac{a_1 s_2}{3}  \cos(3 S + 2 \theta)   - \frac{a_1^2}{4}\cos (4 S + 2 \theta)\Big) \\
    && - \frac{h\mathcal E^2}{16}\Big(
 11 a_1 \cos(S-2 \theta) + 5 a_0  \cos(2 S + 2 \theta)  + a_1 \cos(3 S + 2 \theta) - \frac{2a_1}{3} \cos(3 S + 4 \theta) \\
    && - a_0 \cos( 4S + 4 \theta)  - \frac{2 a_1}{5}  \cos (5 S + 4 \theta)\Big)
,\\
    \psi_8
        &\equiv &
\frac{ b_0}{4} \cos(2 S + 2 \theta) + \frac{b_1}{4} \cos(S+2 \theta)+
    \frac{ b_1}{12} \cos(3S+ 2 \theta) + \frac{7 a_0 a_1}{24} \sin S - \frac{a_1 s_2}{4} \sin S \\
&& + \frac{3a_0 a_1}{8}\sin(S+2 \theta) + \frac{a_1^2}{24} \sin(2S) +
    \frac{a_1 s_2}{8} \sin(S+2 \theta)  +
    \frac{ a_0^2}{8} \sin(2 S + 2 \theta) + \frac{a_1^2}{12} \sin(2 S + 2 \theta)\\
&& + \frac{a_1^2}{32} \sin(2 S + 4 \theta)+ \frac{5 a_0 a_1}{72} \sin(3 S + 2 \theta) -  \frac{ a_1 s_2}{72} \sin(3 S + 2 \theta) + \frac{ a_1^2}{96} \sin(4 S + 2 \theta) \\
&&   + \frac{a_0 a_1}{16} \sin(3 S + 4 \theta) + \frac{ a_0^2}{21} \sin(4 S + 4 \theta) + \frac{a_1^2}{48} \sin(4 S + 4 \theta) + \frac{ a_0 a_1}{48} \sin(5 S + 4 \theta)  \\
&& + \frac{a_1^2}{288} \sin(6 S + 4 \theta)
    + \frac{ h \mathcal E}{8}  \Big(3 a_1 \sin S + 7a_1 \sin (S+ 2 \theta) + 4 a_0 \sin (2 S + 2 \theta) + a_1  \sin (3 S + 2 \theta)\\
&&   - \frac{ a_1}{6} \sin(3 S + 4 \theta) -\frac{a_0}{4} \sin(4 S + 4 \theta) - \frac{ a_1}{10} \sin(5 S + 4 \theta)\Big),
\end{eqnarray*}
reduces system \eqref{ex3} to
\begin{gather}
\label{e3vp}
      \frac{dv}{dt}=\sum_{k=4}^8 t^{- \frac{k}{8}}\Lambda_k(v,\psi)+ \tilde \Lambda_8(v,\psi,t),\quad
    \frac{d\psi}{dt}=\sum_{k=4}^8 t^{-\frac k8}\Omega_k(v,\psi)+\tilde \Omega_8(v,\psi,t),
 \end{gather}
where $\Lambda_4(v,\psi) \equiv \Lambda_5(v,\psi) \equiv \Lambda_7(v,\psi) \equiv \Omega_5(v,\psi)\equiv \Omega_7(v,\psi)\equiv 0$,
\begin{align*}
    &   \Lambda_6(v,\psi)   =   \frac{z_0}{2} v^2,\quad
        \Lambda_8(v,\psi)  =   \frac{v}{4} \Big(2 + 4 b_0 - a_1^2 \sin 2\psi \Big),\\
    &  \Omega_4(v,\psi)    =   -\frac{1}{2} (s_2+a_0) - \frac{3 h}{4} v, \quad
        \Omega_6(v,\psi)   = - \frac{375h^3}{256} v^3, \\
    &   \Omega_8(v,\psi) = - s_4 - \frac{1}{24} \Big(3a_0^2 + 2a_1^2+3a_1^2\cos 2\psi\Big)-\frac{3a_0h }{8} v -\frac{375 h^3}{256}v^3,
\end{align*}
and $\tilde \Lambda_8=\mathcal O(t^{-5/4})$, $\tilde \Omega_8=\mathcal O(t^{-5/4})$ as $t\to\infty$ for all $v\in [0,\Delta_0]$, $\psi\in\mathbb R$.
We see that system \eqref{e3vp} satisfies \eqref{as2} with $n=6$, $d=2$, $\sigma=3$, $\lambda_{n,\sigma}(\psi)\equiv z_0/2$, $\lambda_{n+d}(\psi)\equiv (2+4b_0-a_1^2\sin2\psi)/2$.

Let us consider several possible cases.

(I) Let $s_2+a_0\neq 0$. Then condition \eqref{as02} holds with $m=4$ and a phase drifting regime occurs in system \eqref{FulSys}. If  $2+4b_0-a_1^2>0$ and $z_0>0$, then it follows from Theorem~\ref{Th6} that the equilibrium $(0,0)$ is unstable  (see Fig.~\ref{Fig311}).  By applying Theorem~\ref{Th7} with $(n+d,m)\in\Gamma_{01}$, we conclude that if $b_0+5/4<0$, the equilibrium $(0,0)$ is polynomially stable (see Fig.~\ref{Fig31}, a).  If $b_0+1/2>0$ and $z_0<0$, polynomial stability with the asymptotic estimate $x^2(t)+y^2(t)\sim 2 R_\ast t^{-\nu-l/q}$ as $t\to\infty$ follows from Theorem~\ref{Th8}, where $\nu=d/(q(\sigma-1))=1/4$ and the parameter $R_\ast$ is determined by \eqref{Rast} (see Fig.~\ref{Fig31}, b).
\begin{figure}
\centering
\includegraphics[width=0.28\linewidth]{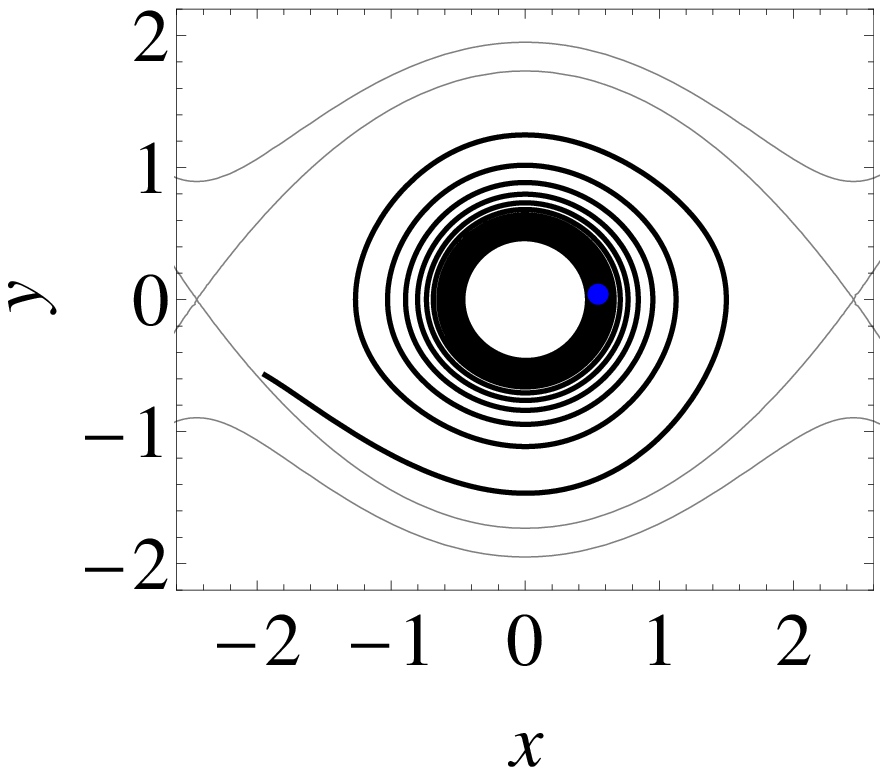} \ \
\includegraphics[width=0.32\linewidth]{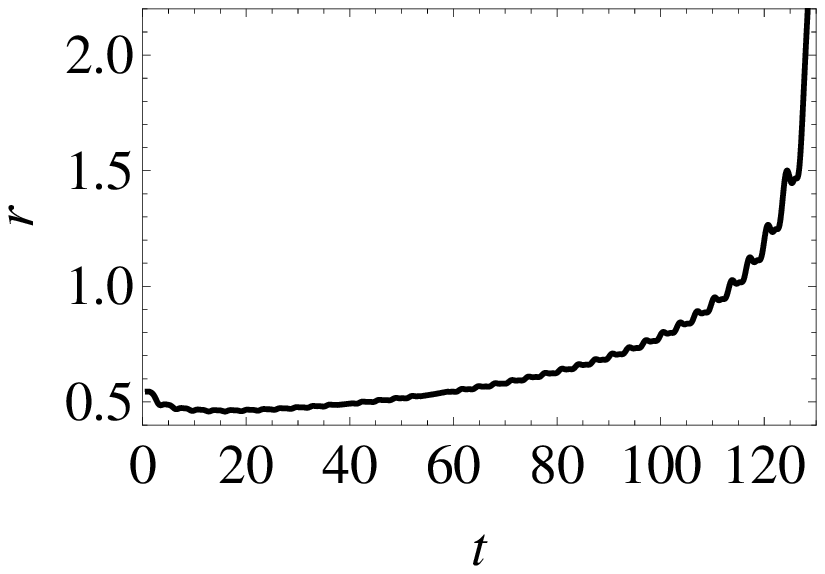} \ \
\includegraphics[width=0.32\linewidth]{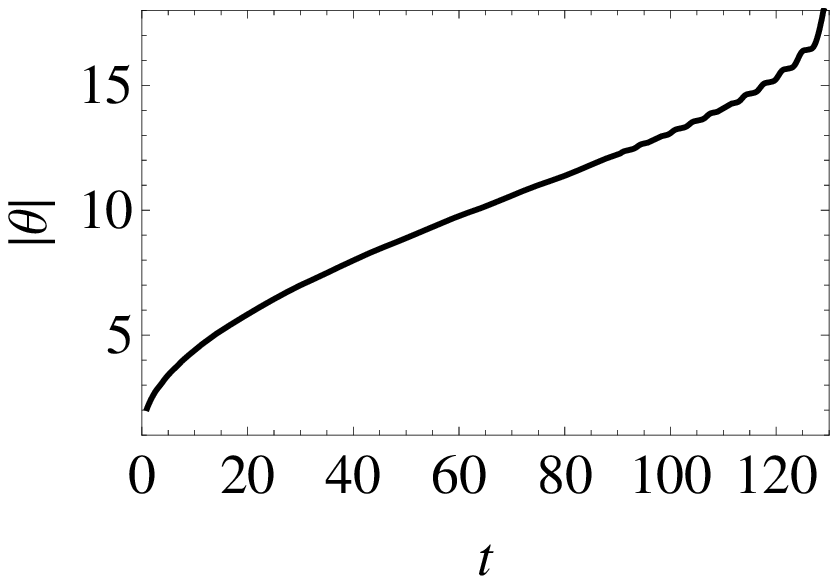}
\caption{\footnotesize The evolution of $(x(t),y(t))$, $r(t)$, $|\theta(t)|$ for solutions of \eqref{ex3} with $b_0=-1/4$, $z_0=0.6$, $s_2=1$, $h=1/6$, $a_0=a_1=b_1=z_1=s_4=0$, where $x(t)=r(t)\cos(\theta(t)+S(t))$, $y(t)=-r(t)\sin(\theta(t)+S(t))$. The blue point corresponds to initial data $(x(1),y(1))$. The gray solid curves correspond to level lines of $H_0(x,y)$.} \label{Fig311}
\end{figure}

\begin{figure}
\centering
\subfigure[$b_0=-1.4$, $z_0=0.6$]{
\includegraphics[width=0.28\linewidth]{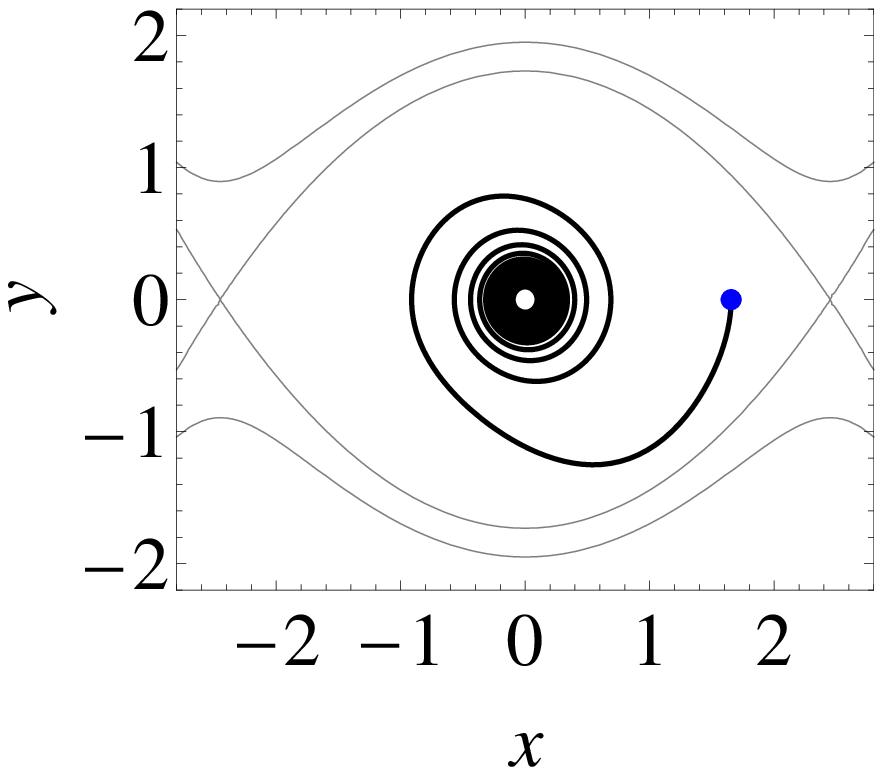} \ \
\includegraphics[width=0.32\linewidth]{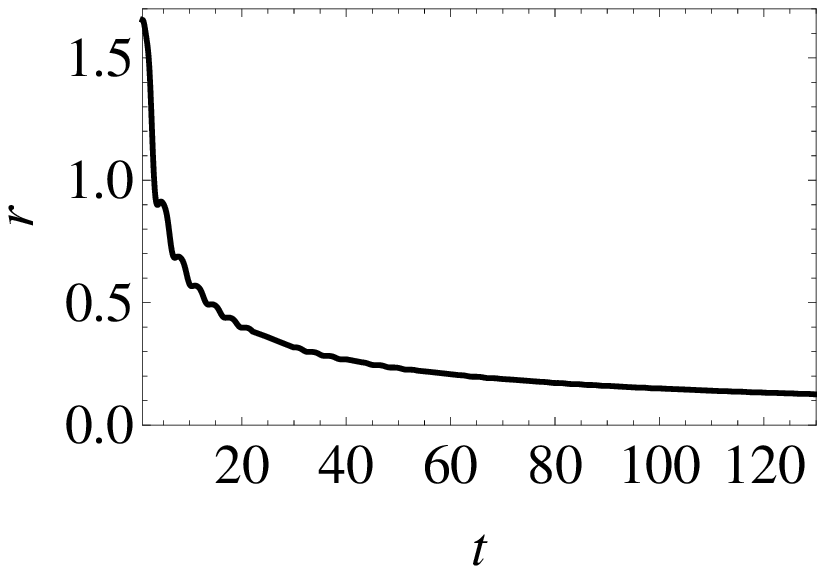} \ \
\includegraphics[width=0.32\linewidth]{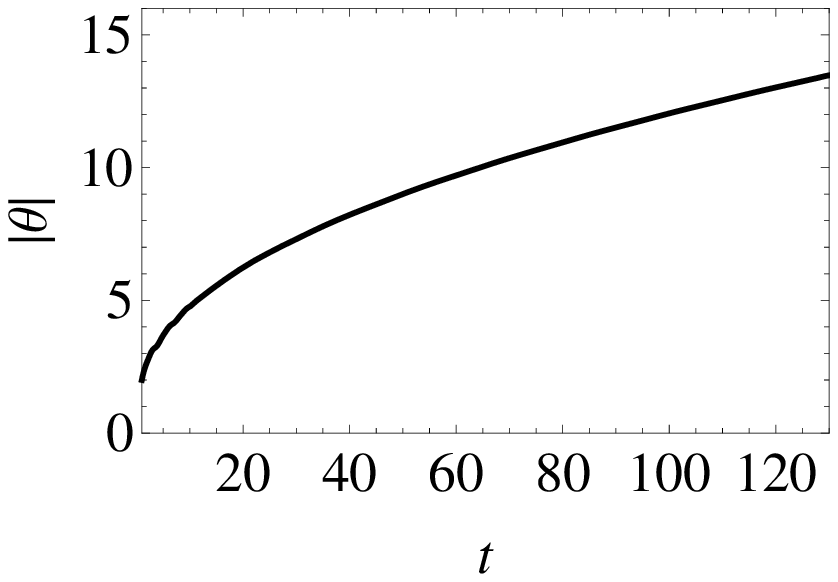}
}
\\
\subfigure[$b_0=1$, $z_0=-0.6$]{
\includegraphics[width=0.28\linewidth]{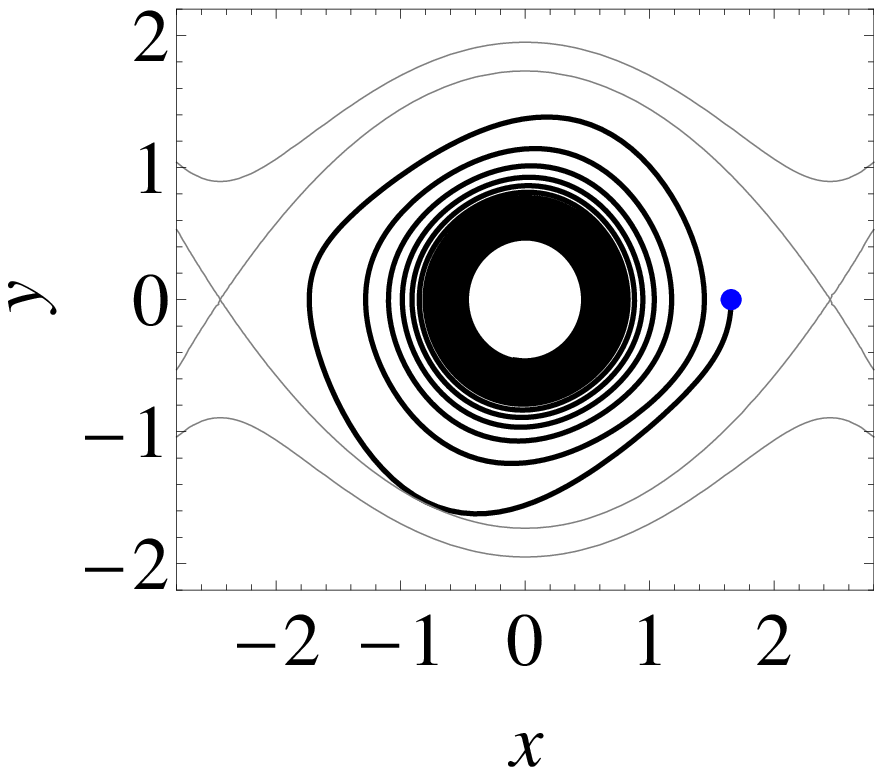} \ \
\includegraphics[width=0.32\linewidth]{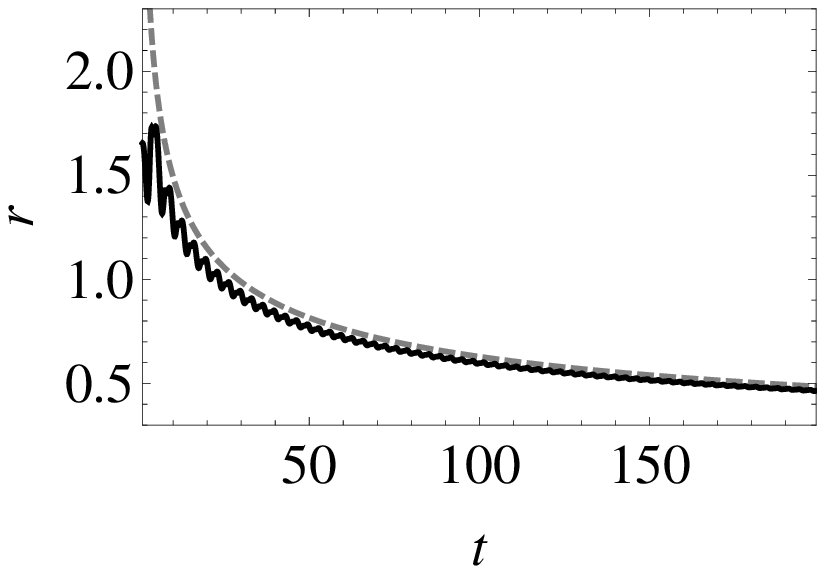} \ \
\includegraphics[width=0.32\linewidth]{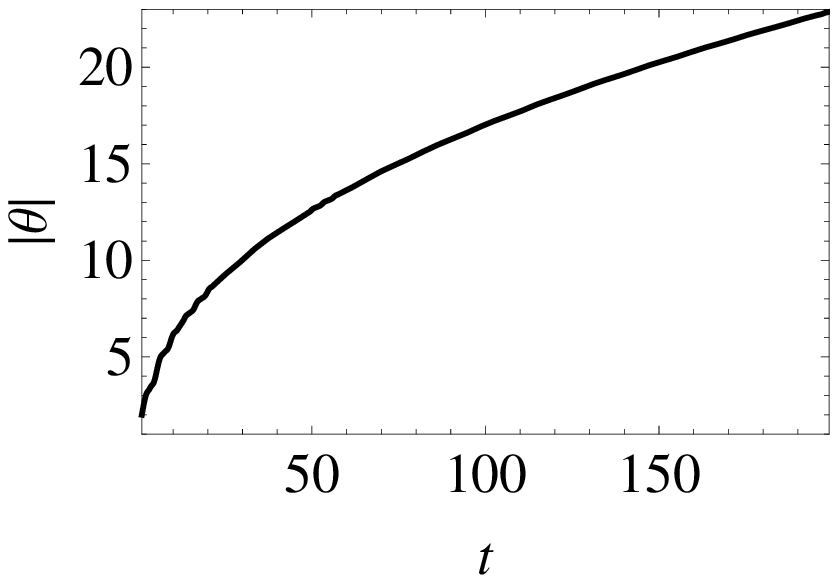}
}
\caption{\footnotesize The evolution of $(x(t),y(t))$, $r(t)$, $|\theta(t)|$ for solutions of \eqref{ex3} with $s_2=1$, $h=1/6$, $a_0=a_1=b_1=z_1=s_4=0$, where $x(t)=r(t)\cos(\theta(t)+S(t))$, $y(t)=-r(t)\sin(\theta(t)+S(t))$. The blue points correspond to initial data $(x(1),y(1))$. The gray solid curves correspond to level lines of $H_0(x,y)$. The gray dashed curve corresponds to $r=2 R_\ast t^{-\nu-l/q}$.} \label{Fig31}
\end{figure}

(II) Let $s_2+a_0= 0$, $a_1>0$ and $h=0$. In this case, $\omega(E)\equiv 1$, $l=0$ and $m=8$. If
\begin{gather}
    \label{e3cond}   -\frac{5a_1^2+3a_0^2}{24}<s_4<\frac{a_1^2-3a_0^2}{24}
\end{gather}
then $\Omega_8$ satisfies condition \eqref{as01} with
\begin{gather*}
\psi_\ast=-\frac{1}{2}\arccos\left(-\frac{24s_4+3a_0^2+2a_1^2}{3a_1^2}\right)+\pi j, \quad j\in\mathbb Z, \quad \vartheta_8=\frac{a_1^2}{4}\sin2\psi_\ast<0.
\end{gather*}
Hence, if either $b_0-\vartheta_8+5/4<0$ or $b_0-\vartheta_8+5/4>0$, $z_0<0$, then, by applying Theorem~\ref{Th2} with $n+d=m=2q$, we obtain polynomial stability of the equilibrium $(0,0)$ in system \eqref{FulSys} (see Fig.~\ref{Fig32}, a). If $b_0-\vartheta_8+1/2>0$ and $z_0>0$, then the equilibrium is unstable (see Fig.~\ref{Fig32}, b).

\begin{figure}
\centering
\subfigure[$b_0=-1.5$, $z_0=0$]{
\includegraphics[width=0.28\linewidth]{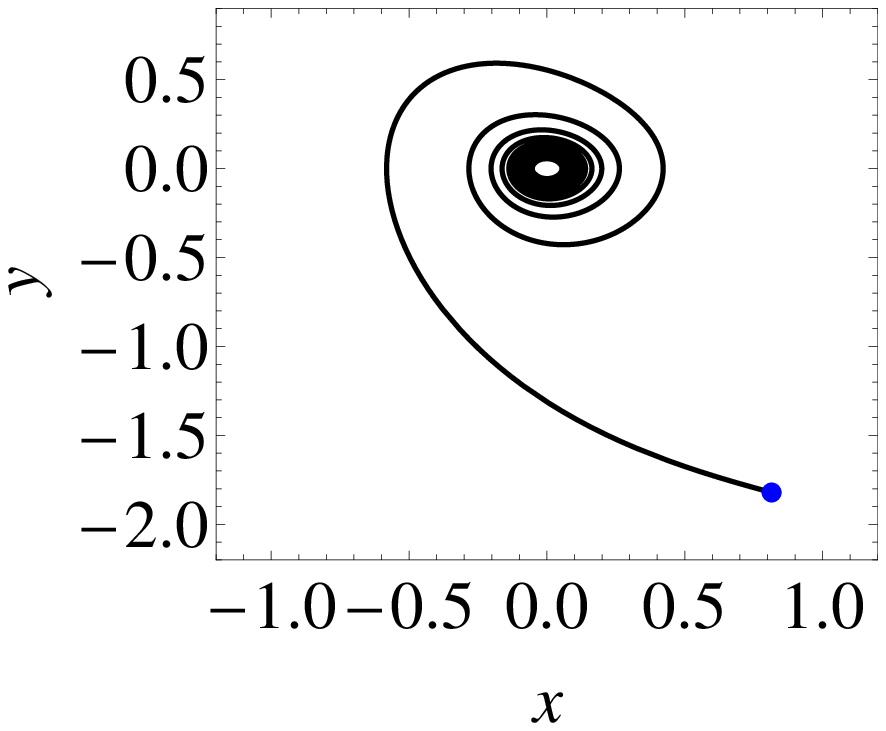} \ \
\includegraphics[width=0.32\linewidth]{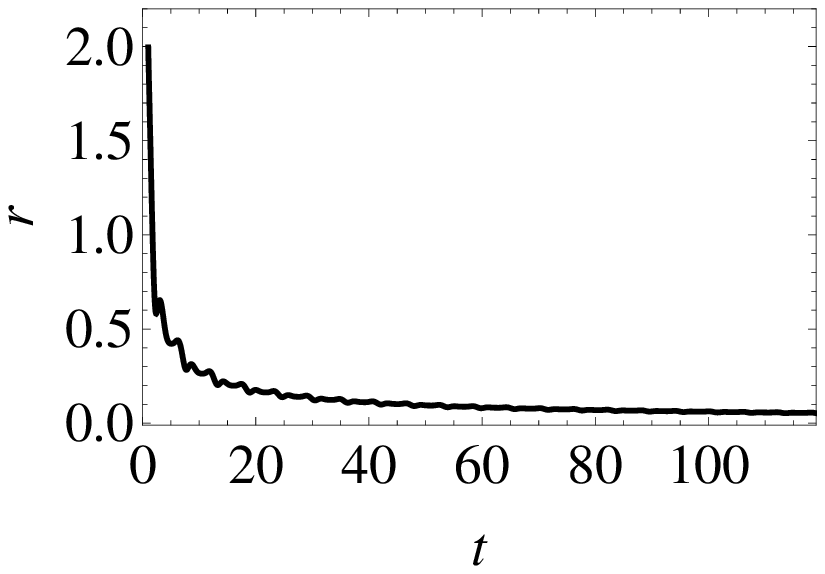} \ \
\includegraphics[width=0.32\linewidth]{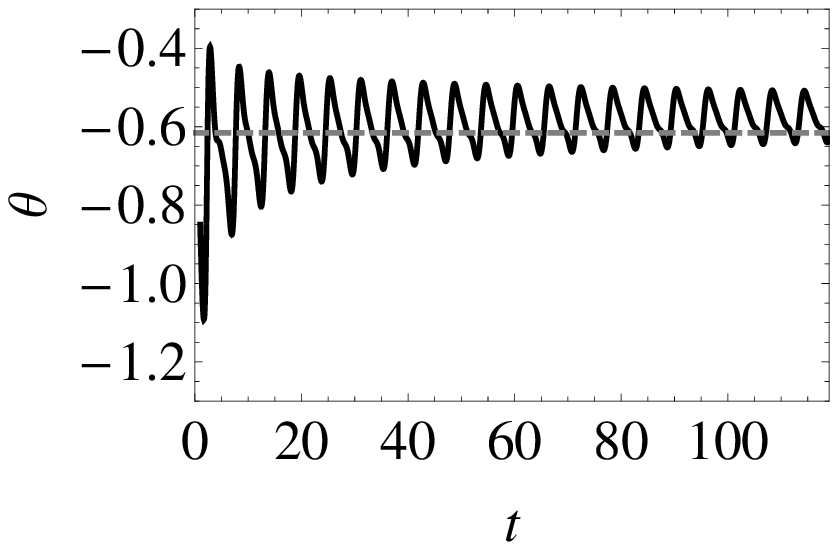}
}
\\
\subfigure[$b_0=-0.2$, $z_0=0.2$]{
\includegraphics[width=0.28\linewidth]{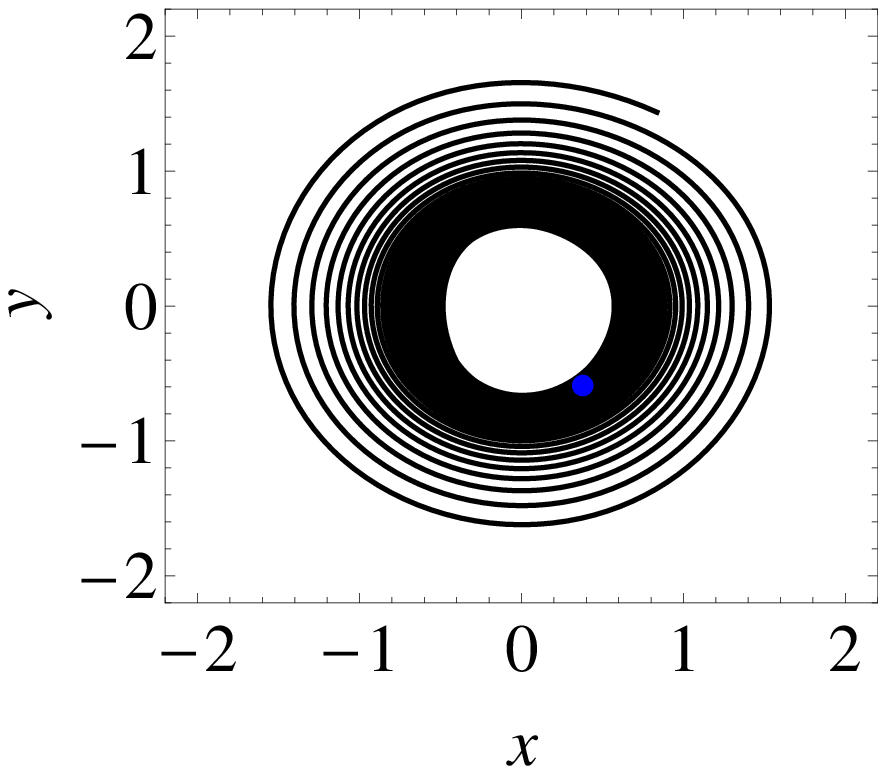} \ \
\includegraphics[width=0.32\linewidth]{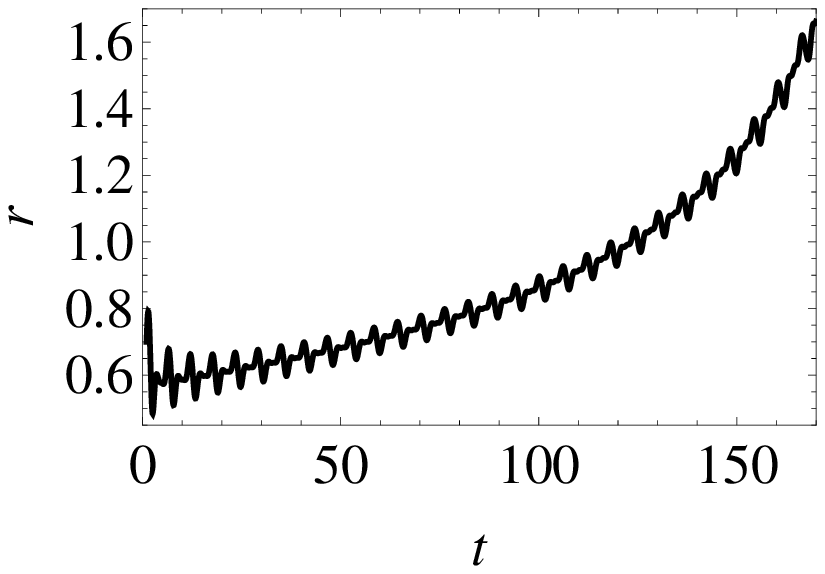} \ \
\includegraphics[width=0.32\linewidth]{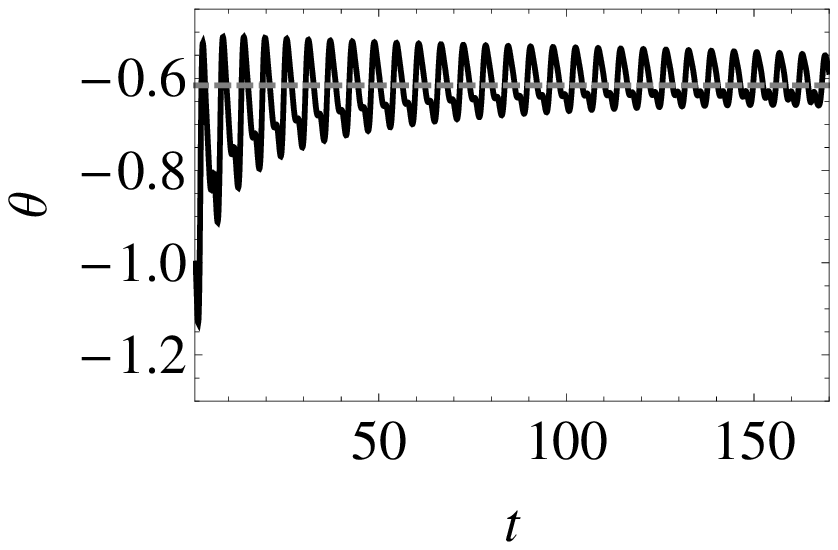}
}
\caption{\footnotesize The evolution of $(x(t),y(t))$, $r(t)$, $\theta(t)$ for solutions of \eqref{ex3} with $a_0=-1$, $a_1=1$, $s_2=1$, $s_4=-1/4$, $b_1=h=0$, $\vartheta_8\approx -0.235$, where $x(t)=r(t)\cos(\theta(t)+S(t))$, $y(t)=-r(t)\sin(\theta(t)+S(t))$. The blue points correspond to initial data $(x(1),y(1))$. The gray dashed curves correspond to $\theta=\psi_\ast$, where $\psi_\ast\approx -0.615$.} \label{Fig32}
\end{figure}

(III) Let $s_2+a_0= 0$, $h=0$ and assumption \eqref{e3cond} does not hold such that $|24s_4+3a_0^2+2a_1^2|>3a_1^2$, then $\Omega_8$ satisfies \eqref{as02}. It follows from Theorem~\ref{Th6} that the equilibrium $(0,0)$ of system \eqref{FulSys} is unstable if $2+4b_0-a_1^2>0$ and $z_0>0$  (see Fig.~\ref{Fig33}, a). By applying Theorem~\ref{Th7}, we obtain exponential stability if $\widehat\gamma_{8,8}=(b_0+1/2)\widehat \chi_8<0$, where $\widehat\chi_8:=\langle|\Omega_8(0,\psi)|^{-1}\rangle_\psi>0$ (see Fig.~\ref{Fig33}, b).

\begin{figure}
\centering
\subfigure[$b_0=0$, $z_0=0.1$]{
\includegraphics[width=0.28\linewidth]{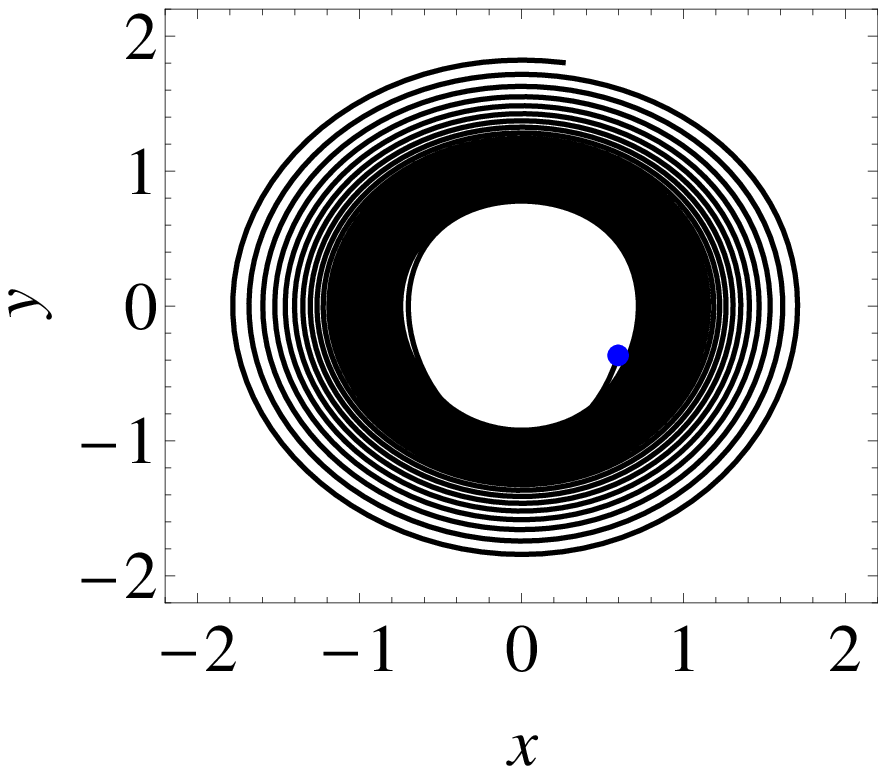} \ \
\includegraphics[width=0.30\linewidth]{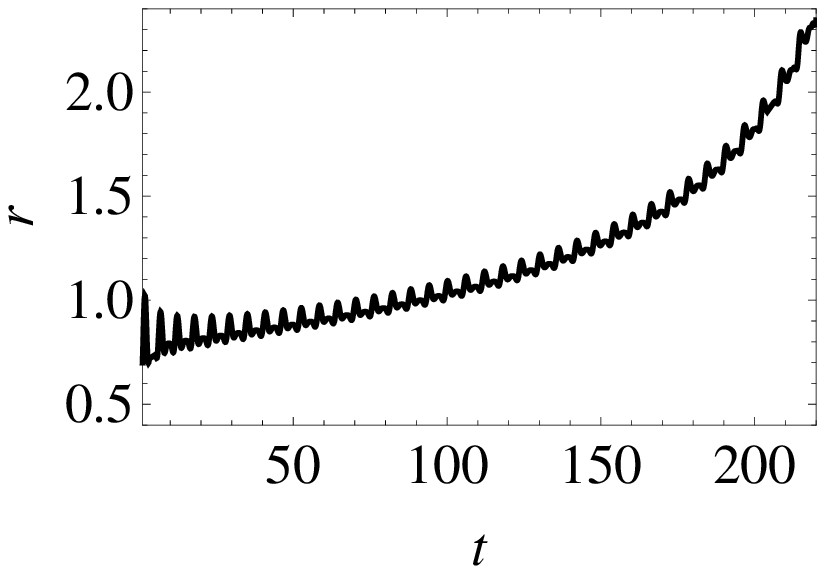} \ \
\includegraphics[width=0.30\linewidth]{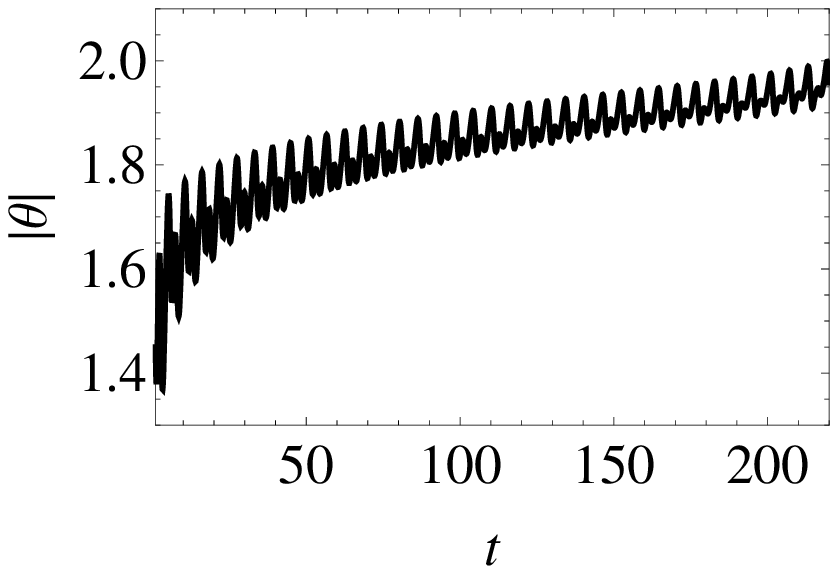}
}
\\
\subfigure[$b_0=-1$, $z_0=0$]{
\includegraphics[width=0.28\linewidth]{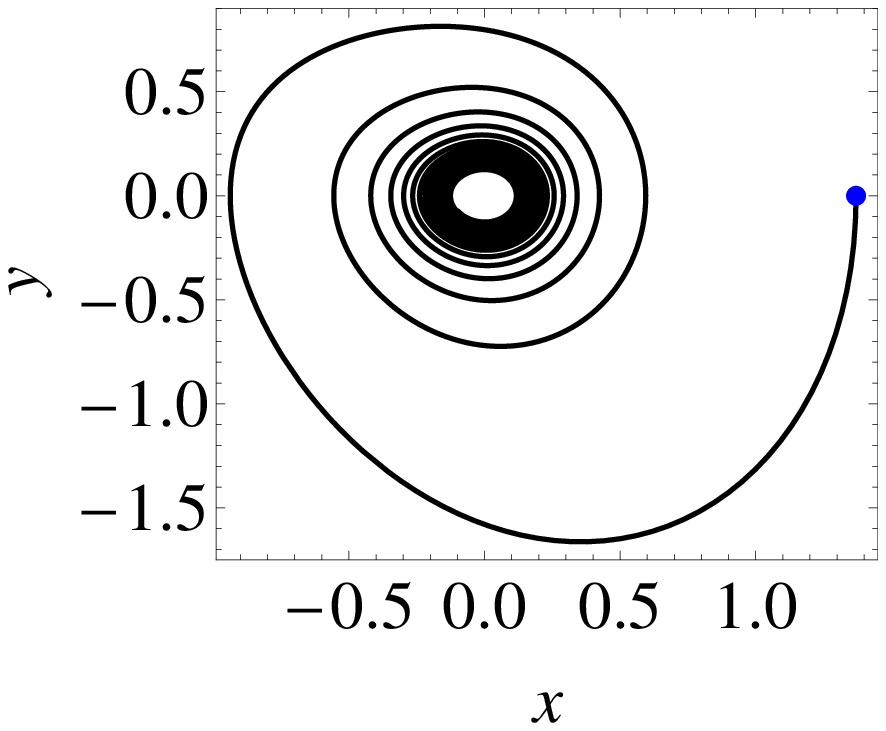} \ \
\includegraphics[width=0.30\linewidth]{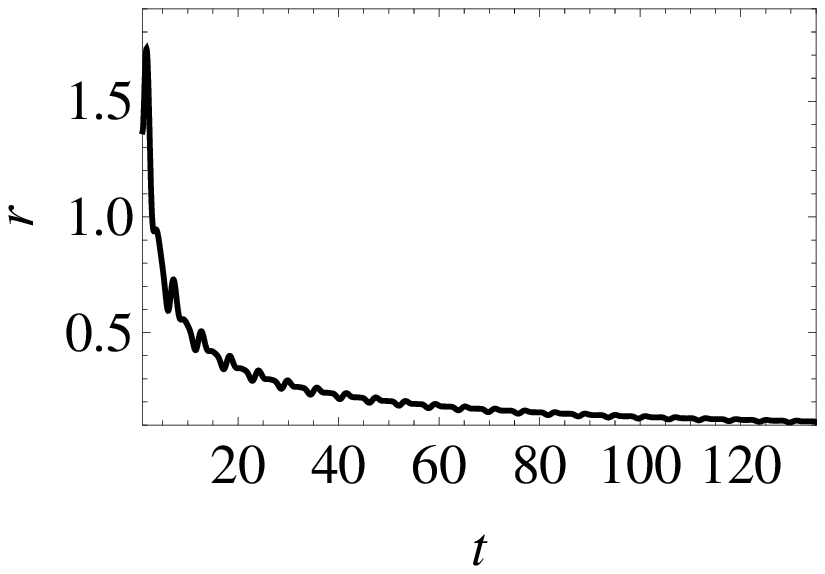} \ \
\includegraphics[width=0.30\linewidth]{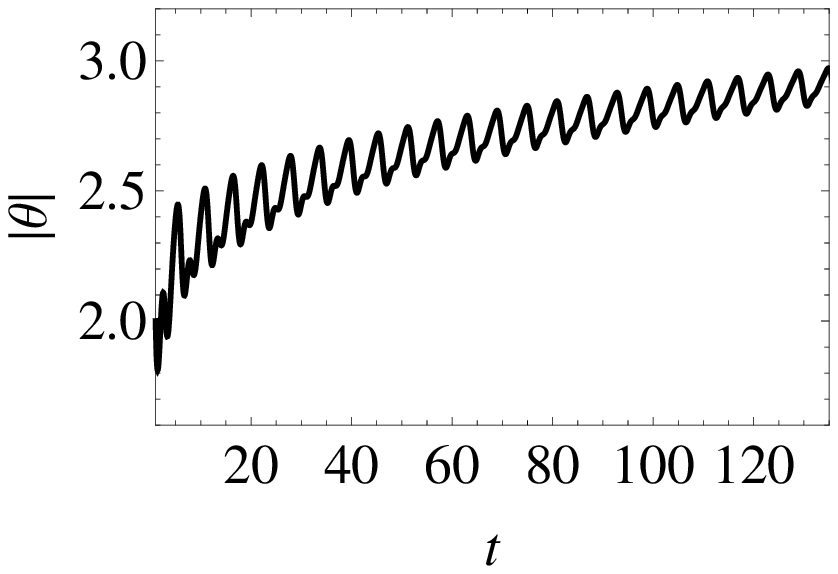}
}
\caption{\footnotesize The evolution of $(x(t),y(t))$, $r(t)$, $|\theta(t)|$ for solutions of \eqref{ex3} with $a_0=-1$, $a_1=1$, $s_2=1$, $s_4=0$, $h=0$, where $x(t)=r(t)\cos(\theta(t)+S(t))$, $y(t)=-r(t)\sin(\theta(t)+S(t))$. The blue points correspond to initial data $(x(1),y(1))$. } \label{Fig33}
\end{figure}

\section{Conclusion}

Thus, we have shown that decaying oscillating perturbations of  Hamiltonian systems in the plane with a neutrally stable equilibrium can lead to the appearance of two different asymptotic regimes: a phase locking and a phase drifting. Which of the modes is realized in the system depends on the structure of the phase equation. In both cases the stability of the equilibrium in the perturbed system depends on the equation for the action variable. We have described the conditions under which the fixed point becomes asymptotically (polynomially or exponentially) stable or losses stability. In some cases, only a weak instability with a weight has been justified  (see, for example, Remark~\ref{Rem4}). In these cases, an additional detailed  analysis of long-term asymptotics for solutions is required.
Note that the emergence of new stable nonzero states when the equilibrium looses the stability has not been investigated in the paper. This will be discussed elsewhere.
\section*{Acknowledgments}
Research is supported by the Russian Science Foundation grant 19-71-30002.

}

 \end{document}